\documentclass[12pt,reqno,colorlinks=true,urlcolor=blue,citecolor=red,linkcolor=blue,linktocpage,pdfpagelabels,bookmarksnumbered,bookmarksopen]{amsart}
\usepackage{orcidlink}
\usepackage{xcolor}
\usepackage{amsthm,amsmath,amssymb,latexsym,soul,cite,mathrsfs}
\usepackage{enumitem}
\usepackage{dsfont}
\usepackage{color,enumitem,graphicx}
\usepackage{hyperref}

\usepackage[left=2.7cm,right=2.7cm,top=2.9cm,bottom=2.9cm]{geometry}

\usepackage[hyperpageref]{backref}
\usepackage{dsfont}

\usepackage{mathtools}

\usepackage{esint}

\usepackage{stackengine}
\stackMath

\usepackage{tikz}
\usepackage{mathtools}
\usepackage{appendix}
\usepackage{ragged2e}
\usepackage{cleveref}
\crefname{section}{Sect.}{Sect. }
\Crefname{section}{\textsection}{\textsection}
\numberwithin{equation}{section}

\theoremstyle{plain}
\newtheorem{theorem}{Theorem}[section]
\newtheorem{theoremletter}{Theorem}

\newtheorem{lemma}[theorem]{Lemma}
\newtheorem{lemmaletter}[theoremletter]{Lemma}
\newtheorem{corollary}[theorem]{Corollary}

\theoremstyle{remark}

\newtheorem{remark}[theorem]{Remark}
\newtheorem{claim}{Claim}
\theoremstyle{definition}

\newtheorem{definition}[theorem]{Definition}

\def\Xint#1{\mathchoice
{\XXint\displaystyle\textstyle{#1}}%
{\XXint\textstyle\scriptstyle{#1}}%
{\XXint\scriptstyle\scriptscriptstyle{#1}}%
{\XXint\scriptscriptstyle\scriptscriptstyle{#1}}%
\!\int}
\def\XXint#1#2#3{{\setbox0=\hbox{$#1{#2#3}{\int}$ }
\vcenter{\hbox{$#2#3$ }}\kern-.587  \wd0}}

\def\dashint{\Xint-}

\makeatletter
\newcommand{\leqnomode}{\tagsleft@true}
\newcommand{\reqnomode}{\tagsleft@false}
\makeatother

\title[Quasilinear elliptic equations with Orlicz growth]{Finite energy solutions of quasilinear elliptic equations with Orlicz growth}

\author[E. da Silva]{Estevan Luiz da Silva \orcidlink{0009-0005-8039-1084}} 
\author[J. M.\ do \'O]{Jo\~ao Marcos do \'O*\orcidlink{0000-0001-7039-4365}}

\address[E. da Silva]{Department of Mathematics,
	Federal University of Pernambuco
	\newline\indent
	50740-540, Recife - PE, Brazil}
\email{\href{mailto:estevan.luiz@ufpe.br}{estevan.luiz@ufpe.br}}

\address[J.M. do \'O]{Department of Mathematics,
Federal University of Para\'{\i}ba
\newline\indent
58051-900, Jo\~ao Pessoa-PB, Brazil}
\email{\href{mailto:jmbo@mat.ufpb.br}{jmbo@mat.ufpb.br}}

\thanks{* Corresponding author.}
 

\begin{document}

%
%
\begin{abstract} 
We present a sufficient condition, expressed in terms of Wolff potentials, for the existence of a finite energy solution to the measure data $(p,q)$-Laplacian equation with a “sublinear growth” rate. Furthermore, we prove that such a solution is minimal. Additionally, a lower bound of a suitably generalized Wolff-type potential is necessary for the existence of a solution, even if the energy is not finite. Our main tools include integral inequalities closely associated with $(p,q)$-Laplacian equations with measure data and pointwise potential estimates,  which are crucial for establishing the existence of solutions to this type of problem.  This method also enables us to address other nonlinear elliptic problems involving a general class of quasilinear operators.
\end{abstract}

\maketitle



%
%

\section{Introduction}  

This paper aims to establish sufficient conditions for the existence of finite energy solutions to quasilinear elliptic equations of the specified type:
\begin{equation}\label{equation nonstandard}\tag{$P$}
    -\mathrm{div}\bigg(\frac{g(|\nabla u|)}{|\nabla u|}\nabla u\bigg)=\sigma \, g(u^{\gamma}) \quad\mbox{in}\quad \mathds{R}^n,
 \end{equation}
where $n\geq 3$ and $\sigma$ belongs to the set of all nonnegative Radon measures on $\mathds{R}^n$, denoted here by $ M^+(\mathds{R}^n)$.  
A typical example of the nonlinear term we consider is:
\begin{equation}\label{our g}\tag{$A_1$}
    g(t)=t^{p-1}+t^{q-1},
\end{equation}
for  $1 < p < q <\infty$, and $\gamma$ satisfying,
 \begin{equation}\label{condition gamma}\tag{$A_2$}
0<\gamma<\min\left\{\frac{p-1}{q-1},\frac{1}{q-p}\right\}.
\end{equation}
In particular, $\gamma<1$ describes exactly the situation of ``sublinear growth''. 
For this nonlinearity in \eqref{our g},  Eq.~\eqref{equation nonstandard} becomes
\begin{equation}\label{pq-problem}
    -\big(\Delta_p u + \Delta_q u\big)=\sigma\,\big(u^{\gamma(p-1)}+u^{\gamma(q-1)}\big) \quad\mbox{in}\quad \mathds{R}^n,
\end{equation}
where $\Delta_pu=\mathrm{div}(|\nabla u|^{p-2}\nabla u)$ is the usual $p$-Laplace operator. 
An important characteristic of \eqref{pq-problem} is the presence of two differential operators with different growth rates. The problem integrates the effects of a general nonlinearity and an unbalanced operator, resulting in a complex interplay between these elements. This combination presents distinct challenges in both analysis and solution methodologies. 

If $p=q$ we have $g(t)=t^{p-1}$ and Eq.~\eqref{equation nonstandard} simplifies to:
\begin{equation}\label{equation standard}
       -\Delta_p u=\sigma \, u^{\gamma(p-1)} \quad\mbox{in}\quad \mathds{R}^n, 
\end{equation} 
with $0<\gamma<1$. We emphasize that for additional examples of quasilinear operators, we recommend referring to, for instance, \cite{MR1691019, MR1433028}.

Our approach employs tools of Orlicz-Sobolev Spaces. For more details, see Section~\ref{preliminaries orlicz}.
Let $G(t)=\int_0^t g(s)\,\mathrm{d}s=t^p/p+t^q/q$, it is known that $G$ is a $N$-function. 
Associated with $G$, consider the Orlicz space $L^G(\mathds{R}^n)$. We are interested in finite
energy solutions $u\in \mathcal{D}^{1,G}(\mathds{R}^n)$ to \eqref{equation nonstandard}, where $\mathcal{D}^{1,G}(\mathds{R}^n)$ is the homogeneous Orlicz-Sobolev space defined as the space of all functions $u\in L_{\mathrm{loc}}^G(\mathds{R}^n)$ that admit weak
derivatives $\partial_k u\in L^G(\mathds{R}^n)$ for $k=1,\ldots,n$.

Setting $f(t)=g(t^{\gamma})$ and $F(t)=\int_0^t f(s)\,\mathrm{d} s$ for $t\geq 0$, a nonnegative function $u$ is called a finite energy solution to \eqref{equation nonstandard} if $u\in \mathcal{D}^{1,G}(\mathds{R}^n)\cap L^F(\mathds{R}^n,\mathrm{d}\sigma)$, and it satisfies \eqref{equation nonstandard} weakly, that is,
\begin{equation*}
    \int_{\mathds{R}^n}\frac{g(|\nabla u|)}{|\nabla u|}\nabla u\cdot \nabla\varphi\, \mathrm{d} x = \int_{\mathds{R}^n}f(u)\varphi \,\mathrm{d}\sigma \quad \forall \varphi\in C_c^{\infty}(\mathds{R}^n).
\end{equation*}
Note that $u$ is a finite energy solution to \eqref{equation nonstandard} if and only if $u\in \mathcal{D}^{1,G}(\mathds{R}^n)$ and it is a  critical point of the functional
\begin{equation*}
    J(v)=\int_{\mathds{R}^n}G(|\nabla v|)\,\mathrm{d} x - \int_{\mathds{R}^n} F(v)\,\mathrm{d}\sigma, \quad v\in \mathcal{D}^{1,G}(\mathds{R}^n).
\end{equation*}
For the precise definitions of the Orlicz-Sobolev spaces considered here, see Section~\ref{preliminaries orlicz}. 
We will look for positive finite energy solutions.

{\it Some challenges} arise naturally in studying quasilinear equations of the form \eqref{equation nonstandard}. Since the equation does not assume homogeneity, i.e., \( g(|\lambda \xi|) = |\lambda|^{p-1} g(|\xi|) \), the computations are more complicated; in particular, our class of solutions is not invariant under scalar multiplication. When \(\sigma\) is the null measure, the regularity of solutions to \eqref{equation nonstandard}, i.e., critical points of \( J \), is known only for the ratio \( q/p \) sufficiently close to \( 1 \) depending on \( n \) \cite{MR0969900, MR1094446}. This, together with the lack of homogeneity, presents difficulties in showing \textit{a priori} estimates (local boundedness of the function) in the general case \( 1 < p < q < \infty \) for solutions to \eqref{equation nonstandard}. Our approach overcomes these problems by using elements of nonlinear potential theory. To be precise, we consider the potential of
Wolff-type $\mathbf{W}_{G}\sigma$ defined by
\begin{equation}\label{definition wolff}
    \mathbf{W}_{G}\sigma(x)=\int_0^{\infty} g^{-1}\bigg(\frac{\sigma(B(x,t))}{t^{n-1}}\bigg)\,\mathrm{d}t,
\end{equation}
 where $\sigma\in {M}^+(\mathds{R}^n)$ and $B(x,t)$ is the open ball of radius $t>0$ centered at $x\in \mathds{R}^n$. Observe that $\mathbf{W}_G\sigma$ may be infinity. When $p=q$, note that $\mathbf{W}_{G}\sigma$ becomes the so-called Wolff potential
\begin{equation}\label{wolff potential standard}
    \mathbf{W}_p\sigma(x):=\mathbf{W}_{G}\sigma(x)= \int_{0}^{\infty}\bigg(\frac{\sigma(B(x,t))}{t^{n-1}}\bigg)^{\frac{1}{p-1}}\,\mathrm{d} t=\int_{0}^{\infty}\bigg(\frac{\sigma(B(x,t))}{t^{n-p}}\bigg)^{\frac{1}{p-1}}\,\frac{\mathrm{d} t}{t}.
\end{equation}

{\it The proof strategy} to obtain solutions for \eqref{equation nonstandard} can be outlined as follows. Initially, we examine, in a general framework, the following integral equation involving the Wolff potential
\begin{equation}\label{wolff integral equation}\tag{$S$}
    u=\mathbf{W}_{G}\big(f(u)\mathrm{d}\sigma\big) \quad \mbox{in}\quad \mathds{R}^n
\end{equation}
where $f(t)=g(t^\gamma)$, and $u$ is a nonnegative $\sigma$-measurable function which belongs to $L_{\mathrm{loc}}^{f}(\mathds{R}^n, \mathrm{d}\sigma)$. 
 Next, starting from a suitable function, an iterative method is employed to prove the existence of solutions to  \eqref{wolff integral equation}. 
 We use the method of successive approximations to complete the argument for the existence of solutions to \eqref{equation nonstandard} (finite energy). In this context,  solutions to \eqref{wolff integral equation} act as an upper barrier for the solution to \eqref{equation nonstandard}, allowing us to control the successive approximations. As we will see, no one additional relation of $p$ and $q$ will be imposed, much less on the ratio $q/p$. Our results are new even for nonnegative functions $\sigma\in L_{\mathrm{loc}}^1(\mathds{R}^n)$, here $\mathrm{d}\sigma =\sigma\,\mathrm{d} x$.

In the context of quasilinear elliptic equations with measure data, it is natural to try to relate \eqref{equation nonstandard} to \eqref{wolff integral equation} for $g$ given by \eqref{our g}. In the case $p=q$, Wolff potentials appeared in the notable works of T. Kilpel\"{a}inen and J. Mal\'{y} \cite{MR1264000, MR1205885}. Indeed, from \cite[Corollary~4.13]{MR1264000}, there exists a constant $K=K(n,p)\geq 1$ such that it holds
\begin{equation}\label{Kilpelainen-Maly estimates}
K^{-1}\mathbf{W}_{p} \mu(x) \leq u(x) \leq K\,\mathbf{W}_{p}\mu(x) \quad \forall x\in \mathds{R}^n,
\end{equation}
provided $u$ is a nonnegative solution in the potential-theoretic sense of
\begin{equation*}
\left\{
\begin{aligned}
    -\Delta_p u&= \mu, \quad \mbox{in}\quad \mathds{R}^n, \\
  \inf_{\mathds{R}^n}u&=0.
\end{aligned}
     \right.
\end{equation*}
A type of Orlicz counterpart of this result with nonnegative measure $\mu$
was established with the potential \eqref{definition wolff} in \cite{MR1975101,MR4803728}.

Using estimates from \eqref{Kilpelainen-Maly estimates}, C. Dat and I. Verbistky in \cite{MR3311903} were able to provide a necessary and sufficient condition in terms of $\mathbf{W}_p\sigma$ to construct a solution to \eqref{equation standard}, which possesses finite energy with respect to the functional $J$  given previously (considering $p=q$).  For related results, see also \cite{MR3556326, MR3567503}.

Inspired by these ideas, we imposed conditions on the Wolff potentials $\mathbf{W}_p\sigma$ and $\mathbf{W}_q\sigma$ to establish a sufficient condition for the existence of solutions with finite energy to \eqref{equation nonstandard}. 
Additionally, a necessary condition for this existence is presented in terms of $\mathbf{W}_G\sigma$.
This perspective provides new insights into the connection between potential estimates and elliptic equations with Orlicz growth.


\subsection{Assumptions} 
We will require that $\sigma\in M^+(\mathds{R}^n)$ satisfies the following conditions 
\begin{equation}\label{cond sufficient to exist}
    \left\{
    \begin{aligned}
    &\mathbf{W}_p\sigma^{\frac{1}{1-\gamma}}, \, \mathbf{W}_q\sigma^{\frac{1}{1-\gamma}} &\in& &L^{F}(\mathds{R}^n,\mathrm{d} \sigma)&,  \\
    & \mathbf{W}_p\sigma^{\frac{p-1}{p-1-\gamma(q-1)}},\, \mathbf{W}_q\sigma^{\frac{p-1}{p-1-\gamma(q-1)}}&\in& &L^{F}(\mathds{R}^n,\mathrm{d} \sigma)&,  \\
    & \mathbf{W}_p\sigma^{\frac{q-1}{q-1-\gamma(p-1)}},\, \mathbf{W}_q\sigma^{\frac{q-1}{q-1-\gamma(p-1)}}&\in&  &L^{F}(\mathds{R}^n,\mathrm{d} \sigma)&.
    \end{aligned}
    \right.
\end{equation}
This condition extends partially the condition stated in \cite{MR3311903}, see Remark~\ref{remark extension of Igor} below.

\subsection{Description of the results}
We can now state our main findings.

\begin{theorem}\label{existencewolffequation}
Let $g$ be given by \eqref{our g}, and let $\sigma\in M^+(\mathds{R}^n)$. The equation \eqref{wolff integral equation} has a nontrivial solution $u\in L^F(\mathds{R}^n,\mathrm{d}\sigma)$ whenever \eqref{cond sufficient to exist} holds.
\end{theorem}
   It is our interest to know whether condition \eqref{cond sufficient to exist} can be shortened in a more uncomplicated condition (see \eqref{cond sufficient to exist idea} below). We will apply the previous theorem to obtain solutions to Eq. \eqref{equation nonstandard}.

In the following theorem, we obtain a minimal solution to \eqref{equation nonstandard}. This means that $u$ is a finite energy solution to \eqref{equation nonstandard}, and for any solution $v$ to \eqref{equation nonstandard} of finite energy, we must have $u\leq v$ pointwise, i.e.  $u(x) \leq v(x) \ \forall x\in \mathds{R}^n$.

\begin{theorem}\label{solution to orlicz-measure eq}
Suppose $1< p<q<n$. Let $g$ be given by \eqref{our g}, and let $\sigma\in M^+(\mathds{R}^n)$ satisfying \eqref{cond sufficient to exist}. Then there exists a nontrivial solution $u\in \mathcal{D}^{1,G}(\mathds{R}^n)\cap L^F(\mathds{R}^n,\mathrm{d}\sigma)$ to \eqref{equation nonstandard}. Furthermore, $u$ is minimal.
For $q\geq n$, \eqref{equation nonstandard} has only the trivial solution $u=0$.
\end{theorem}

\subsection{Related literature}
Problems involving elliptic operators governed by Orlicz-Sobolev spaces and general $p,q$-growth have been systematically investigated in the literature, particularly those where the right-hand side is a Radon measure. This line of research began with the papers \cite{MR1163440, MR1025884}, which addressed operators modeled on the $p$-Laplacian ($p=q$). We refer to several contributions on this topic \cite{MR1975101, MR2352517, MR3712027, MR3347481, MR3720844, MR4803728} and the references therein. In this context, potential estimates are well-established and precise tools for analyzing these types of problems. In \cite{MR4803728}, I. Chlebicka, F. Giannetti, and A. Zatorska-Goldstein established sharp pointwise bounds expressed in terms of \eqref{definition wolff} for a broad class of solutions to problems with Orlicz growth. They also derived powerful corollaries, providing regularity results for the local behavior of solutions, particularly when the measure data satisfies conditions in the relevant scales of generalized Lorentz, Marcinkiewicz, or Morrey types. Our method relies on these potential estimates in a specific form, as seen on the right side of \eqref{equation nonstandard}. This work aims to motivate the study of equations like \eqref{equation nonstandard} by linking integral equations \eqref{wolff integral equation} with the technique of Wolff-type potentials $\mathbf{W}_G\sigma$, following ideas from \cite{MR3311903}. We also note a different approach from ours, the notion of SOLA (Solutions Obtained as Limits of Approximations), introduced in \cite{MR1025884}. This method is effective when the Radon measure on the right side is bounded, providing known local upper bounds for the eventual solutions.  We mention that the authors have previously employed this approach to investigate other types of problems. For example, in \cite{MR4711145}, we applied this methodology to the study of quasilinear Lane-Emden type systems with sub-natural growth terms, and in \cite{artigo2}, we extended it to the analysis of Hessian Lane-Emden type systems involving measures with sub-natural growth terms.



\subsection{Outline of the paper} 
In Section~\ref{preliminaries orlicz}, we compile fundamental results from nonlinear potential theory in Orlicz-Sobolev spaces. Section~\ref{Wolff potentials section orlicz} establishes the framework for proving Theorem~\ref{existencewolffequation}. In Section~\ref{applications orlicz}, we present a detailed proof of Theorem~\ref{solution to orlicz-measure eq}, along with some preliminary results and remarks. Section~\ref{Potential estimates App} provides proofs of specific potential estimates used in this work, utilizing a Harnack-type inequality. Finally, Section~\ref{Further works} suggests several potential future directions for addressing questions related to quasilinear problems with Orlicz growth, such as \cref{equation nonstandard}.

\subsection{Notations}
\begin{itemize}
    \item $\Omega$ is a domain in $\mathds{R}^n$.
    \item As usual, we use the letters $c$, $\tilde{c}$, $C$, and $\tilde{C}$, with or without subscripts, to denote different constants.
    \item $\chi_E$ denotes the characteristic function of a set $E$.
    \item ${M}^+(\Omega)$ denotes the set of all nonnegative Radon measures $\sigma$ defined on $\Omega$. Often, we use the Greek letters $\mu$ and $\omega$ to denote Radon measures.
    \item $C(\Omega)$ denotes the set of all continuous functions on $\Omega$.
    \item $C_c^{\infty}(\Omega)$ denotes the set of all infinitely differentiable functions with compact support in $\Omega$.
    \item $L^0(\Omega, \mathrm{d}\mu)$ denotes the set of measurable functions on $\Omega$ with respect to $\mu \in {M}^+(\Omega)$. If $\mu$ is the Lebesgue measure, we write $L^0(\Omega)$.
    \item $L^s(\Omega, \mathrm{d}\mu)$ denotes the local $L^s$ space with respect to $\mu \in {M}^+(\Omega)$, $s > 0$. If $\mu$ is the Lebesgue measure, we write $L_{\mathrm{loc}}^s(\Omega)$.
    \item We denote by $\sigma(E) = \int_E \mathrm{d}\sigma$ the measure of any $\sigma$-measurable subset $E$ of $\Omega$. When $\sigma$ is the Lebesgue measure, we write $|E| = \sigma(E)$.
\end{itemize}

\section{Preliminaries}\label{preliminaries orlicz} 
For an overview of Orlicz space theory, we refer to the books  \cite{MR2424078, MR0126722, MR1113700, MR3931352} and references therein. Let us remark that some of their references deal with generalized Orlicz growth.
   \subsection{Young Functions}
\begin{definition}[Young function]
    A function $G:[0,\infty)\to [0,\infty)$ is said to be a Young function if $G$ is convex, strictly increasing and satisfies $G(0)=0$.
\end{definition}
\begin{definition}[$N$-functions and its Conjugate]
 A Young function $G:[0,\infty)\to [0,\infty)$ is an $N$-function if
\begin{equation*}
    \lim_{t\to 0^+}\frac{G(t)}{t}=0 \quad \mbox{and}\quad \lim_{t\to\infty}\frac{G(t)}{t}=\infty.
\end{equation*}
The function $G^{\ast}(s):=\sup_{t>0}\{s t - G(t)\}$, for $s\in [0,\infty)$, is called the complementary function of $G$.
\end{definition}
Equivalently, we may define $G^{\ast}$ by
\begin{equation*}
    G^{\ast}(s)=\int_0^sg^{-1}(r)\,\mathrm{d}r, \quad s\geq 0,
\end{equation*}
where $g=G'$. Clearly, by definition, it holds 
\begin{equation}\label{young inequality}
    t s\leq G(t)+G^{\ast}(s) \quad \forall t,s\geq 0.
\end{equation}
Equality occurs in \eqref{young inequality}  if and only if either $t=g^{-1}(s)$ or $s=g(t)$. In addition, it holds $(G^{\ast})^{\ast}(t)=G(t)$ for all $t\geq0$.

Unless otherwise stated, we assume that all $N$-functions $G$ in this work belong to $C^2(0,\infty)$ where $g=G'$ satisfies
\begin{equation}\label{assumption on g}
  p-1\leq \frac{t g'(t)}{g(t)}\leq q-1 \quad\forall t>0.
\end{equation}
for some $1<p\leq q<\infty$. In particular, we have
\begin{equation}\label{data condition elementary}
    p\leq \frac{t g(t)}{G(t)}\leq q \quad\forall t>0,
\end{equation}
Indeed, by \eqref{assumption on g}, $t\mapsto g(t)/t^{p-1}$ is a nondecreasing function and $t\mapsto g(t)/t^{q-1}$ is a nonincreasing function. Consequently, for all $t>0$
\begin{equation*}
    \begin{aligned}
        G(t)=&\int_0^t g(s)\,\mathrm{d} s=\int_0^t s^{p-1}\frac{g(s)}{s^{p-1}}\,\mathrm{d} s\leq \frac{t g'(t)}{p},\\
        G(t)=&\int_0^t g(s)\,\mathrm{d} s=\int_0^t s^{q-1}\frac{g(s)}{s^{q-1}}\,\mathrm{d} s\geq \frac{t g'(t)}{q}.
    \end{aligned}
\end{equation*}

 In customary terminology, condition~\eqref{data condition elementary} is known as \textit{$\Delta_2$ and $\nabla_2$ condition}. A typical example of a function that satisfies \eqref{data condition elementary} is our model $G(t)=t^p/p + t^q/q$, with $g(t)=t^{p-1}+t^{q-1}$, where $1<p\leq q<\infty$.
The inequalities \eqref{data condition elementary} have some basic applications. In the next lemma, we emphasize some of that, where we omit the easy proof which can be found, e.g., in \cite[Section~2]{MR1975101}, or \cite[Lemma~2.10]{MR4258773}.
\begin{lemmaletter}
    Suppose $g$ satisfies \eqref{assumption on g}. Then \eqref{data condition elementary} holds. In particular, $t\mapsto G(t)/t^p$ is a nondecreasing function, and $t\mapsto G(t)/t^q$ is a nonincreasing function, for $t>0$. Moreover, 

\begin{flushleft}
{$\mathbf {(i)}$} for all $\alpha \geq 0$ and $t>0$,
\begin{align}
     \min\{\alpha^{p-1},\alpha^{q-1}\}g(t)&\leq g(\alpha t) \leq\max\{\alpha^{p-1},\alpha^{q-1}\}g(t),\label{estimate on g} \\
         \min\{\alpha^{p},\alpha^{q}\}G(t)&\leq G(\alpha t) \leq\max\{\alpha^{p},\alpha^{q}\}G(t), \label{estimate on G} \\
         \left(\frac{p}{q}\right)^{\frac{1}{p-1}}\min\{\alpha^{\frac{1}{p-1}},\alpha^{\frac{1}{q-1}}\}g^{-1}(t)&\leq g^{-1}(\alpha t)\leq \left(\frac{q}{p}\right)^{\frac{1}{p-1}}\max\{\alpha^{\frac{1}{p-1}},\alpha^{\frac{1}{q-1}}\}g^{-1}(t) ,\label{estimate on g-1}\\
         \min\{\alpha^{\frac{p}{p-1}},\alpha^{\frac{q}{q-1}}\}G^{\ast}(t)&\leq G^{\ast}(\alpha t) \leq \max\{\alpha^{\frac{p}{p-1}},\alpha^{\frac{q}{q-1}}\}G^{\ast}(t);\label{estimate on G ast}
\end{align}
 \\
{$\mathbf {(ii)}$}   for all $t>0$, setting  ${c}=g^{-1}(1)$, it holds
\begin{equation}\label{estimate g sum}
         g^{-1}(t) \leq {c}\bigg(\frac{q}{p}\bigg)^{\frac{1}{p-1}}\big(t^{\frac{1}{p-1}}+t^{\frac{1}{q-1}}\big),
\end{equation}
{$\mathbf{(iii)}$} for all $t>0$, it holds
\begin{equation}\label{relation G and tilde G}
    c_{1}\,G(t)\leq G^{\ast}(g(t)))\leq c_2\, G(t),
\end{equation}
for some positive constants $c_1,\, c_2$ depending only on $p$ and $q$.
  \end{flushleft}
  \end{lemmaletter}
 
\subsection{Orlicz Spaces}
\begin{definition}\label{Orlicz space}
    The Orlicz space $L^G(\Omega)$ is understood as the set
    \begin{equation*}
        L^G(\Omega):=\left\{u\in L^0(\Omega): \int_{\Omega}G(\lambda|u|)\,\mathrm{d} x<\infty\  \mbox{for some }\lambda>0\right\}.
    \end{equation*}
    $L^G(\Omega)$ is a Banach space with the Luxemburg norm
    \begin{equation*}
        \|u\|_{L^G}:=\inf_{\lambda>0}\left\{\lambda>0: \int_{\Omega}G\Big(\frac{|u|}{\lambda}\Big)\,\mathrm{d} x\leq 1\right\}.
    \end{equation*}
\end{definition}
In view of \eqref{data condition elementary}, by \cite[Lemma~3.1.3]{MR3931352}, we have
\begin{equation*}
     L^G(\Omega)=\left\{u\in L^0(\Omega): \int_{\Omega}G(|u|)\,\mathrm{d} x<\infty\right\}.
\end{equation*}
From \cite[Corollary~3.2.8]{MR3931352}, it holds
\begin{equation}\label{relation norm and modular}
    \|u\|_{L^G}\leq \int_{\Omega}G(|u|)\,\mathrm{d}\sigma +1 \quad \forall u\in  L^G(\Omega).
\end{equation}
Furthermore, putting ${\rho_G(u)}=\int_\Omega G(|u|)\,\mathrm{d} x$, the following lemma establishes the relation between $\rho_G(\cdot)$ and  norm $\|\cdot\|_{L^G}$ \cite[Lemma~3.2.9]{MR3931352}. The function $\rho_G(\cdot)$ is called \textit{modular}.
\begin{lemmaletter}\label{lemma relation norm and modular}
    For all $u\in L^G(\Omega)$, it holds
    \begin{equation*}
        \begin{aligned}
           \min\{{\rho_G(u)}^{\frac{1}{p}},{\rho_G(u)}^{\frac{1}{q}}\}&\leq {\|u\|_{L^G}} \leq \max\{{\rho_G(u)}^{\frac{1}{p}},{\rho_G(u)}^{\frac{1}{q}}\}, \\
         \min\{\|u\|_{L^G}^{p},\|u\|_{L^G}^{q}\}&\leq {\rho_G(u)} \leq \max\{\|u\|_{L^G}^{p},\|u\|_{L^G}^{q}\}. 
        \end{aligned}
    \end{equation*}
\end{lemmaletter}
The following result \cite[Lemma~3.2.11]{MR3931352} is the generalization of the classical H\"{o}lder's inequality in Lebesgue spaces to Orlicz spaces.
\begin{lemmaletter}\label{Holder-Orlicz}
    For all $u\in L^{G}(\Omega)$ and $v\in L^{G^{\ast}}(\Omega)$, it holds
    \begin{equation*}
        \bigg|\int_{\Omega}u v\,\mathrm{d} x\bigg|\leq 2\|u\|_{L^G}\|v\|_{L^{G^{\ast}}}.
    \end{equation*}
\end{lemmaletter}
\begin{remark}
    Under condition \eqref{data condition elementary}, $L^{G}(\Omega)$ is reflexive, separable, and uniformly convex. In light of H\"{o}lder's inequality, the dual space $\big(L^{G}(\Omega)\big)^{\ast}$ coincides with  $L^{G^{\ast}}(\Omega)$ (see \cite[Chapther~2, Section~3]{MR3931352} for more details).
\end{remark}
\begin{definition}
Let  $\{u_j\}\in L^{G}(\Omega)$ be a sequence and let $u\in L^{G}(\Omega)$. We say that $u_j$ converges weakly to $u$ in $L^{G}(\Omega)$ if
\begin{equation*}
    \lim_{j\to \infty}\int_{\Omega}u_jv\,\mathrm{d} x =\int_{\Omega}uv\,\mathrm{d} x \quad\forall v\in L^{G^{\ast}}(\Omega).
\end{equation*}
As usual, we write $u_j\rightharpoonup u$ in $L^G(\Omega)$. The weak convergence of vector-valued  
functions in $L^G(\Omega;\mathds{R}^n)$ has an obvious interpretation regarding the coordinate functions. 
\end{definition}
The following result will be useful in some of our arguments \cite[Theorem~2.1]{MR4274293}
\begin{theoremletter}\label{Lemma weakly convergence}
    Let $\{u_j\}$ be a bounded sequence in $ L^{G}(\Omega)$. If $u=\lim_j u_j$ pointwise in $\Omega$ (almost everywhere), then $u_j\rightharpoonup u$ in $L^G(\Omega)$.
\end{theoremletter}

\subsection{Orlicz-Sobolev spaces}
\begin{definition}
    The Orlicz-Sobolev space $W^{1,G}(\Omega)$ is understood as the set of all functions $u \in L^G(\Omega)$ which admit weak derivatives $\partial_i u \in L^G(\Omega)$ for $i=1,\ldots,n$; that is,
    \begin{equation*}
        W^{1,G}(\Omega)=\{u\in L^G(\Omega): |\nabla u|\in L^{G}(\Omega)\},
    \end{equation*}
    equipped with the norm
    \begin{equation*}
        \|u\|_{W^{1,G}}=\|u\|_{L^G}+\|\nabla u\|_{L^{G}}.
    \end{equation*}
    By $W_0^{1,G}(\Omega)$ we denote the closure of $C_c^{\infty}(\Omega)$ in $W^{1,G}(\Omega)$. As usual, $W_{\mathrm{loc}}^{1,G}(\Omega)$ is the set of all functions $u$ such that $u\in W^{1,G}(U)$ for all open subset $U$ compactly contained in $\Omega$.
\end{definition}
\begin{remark}
    Similar to the preceding remark, assuming \eqref{data condition elementary}, $W^{1,G}(\Omega)$ is a Banach space, separable, uniformly convex, and reflexive (see \cite[Theorem~6.1.4]{MR3931352}). This also holds for  $W_0^{1,G}(\Omega)$.
\end{remark}

Next, we state a \textit{modular} Poincar\'e inequality \cite[Lemma~2.2]{doi:10.1080/03605309108820761}, which will be useful in our results.

\begin{lemmaletter}\label{modular poincare ineq}
Let $B_R$ be a ball with a radius $R$. There exists a constant $c=c(n,p,q)>0$ such that 
\begin{equation*}
    \int_{B_R}G\bigg(\frac{|u|}{R}\bigg)\,\mathrm{d} x\leq c\int_{B_R}G(|\nabla u|)\,\mathrm{d} x \quad \forall u\in W_0^{1,G}(B_R).
\end{equation*}
\end{lemmaletter}

\begin{definition}\label{Dirichelt-Orlicz-Sobolev space}
The homogeneous Sobolev-Orlicz space $\mathcal{D}^{1,G}(\Omega)$ is understood as the set
\begin{equation*}
    \mathcal{D}^{1,G}(\Omega)=\{u\in W_{\mathrm{loc}}^{1,G}(\Omega): |\nabla u|\in L^{G}(\Omega)\}.
\end{equation*}
In this space, we have the following seminorm
\begin{equation}\label{seminorm Dirichlet-Orlicz-Sobolev space}
 \|u\|_{\mathcal{D}^{1,G}}=\|\nabla u\|_{L^{G}}.
\end{equation}
\end{definition}
\begin{remark}\label{remark q>n orlicz}
  Consider $\Omega=\mathds{R}^n$. When  $q\geq n$ in \eqref{data condition elementary}, all constants functions belong to  $\mathcal{D}^{1,G}(\mathds{R}^n)$.  Indeed, appealing to \cite[Lemma~3.7.7]{MR3931352}, from \eqref{data condition elementary} we have the following inclusions
    \begin{equation}\label{immersion dirichlet}
         \mathcal{D}^{1,p}(\mathds{R}^n)\cap  \mathcal{D}^{1,q}(\mathds{R}^n)\subset  \mathcal{D}^{1,G}(\mathds{R}^n)\subset  \mathcal{D}^{1,p}(\mathds{R}^n)+ \mathcal{D}^{1,q}(\mathds{R}^n),
    \end{equation}
    where $\mathcal{D}^{1,p}(\mathds{R}^n),\; \mathcal{D}^{1,q}(\mathds{R}^n)$ are the classical homogenous Sobolev spaces, and $\mathcal{D}^{1,p}(\mathds{R}^n)+ \mathcal{D}^{1,q}(\mathds{R}^n)=\{u+v: u\in\mathcal{D}^{1,p}(\mathds{R}^n), \, v\in \mathcal{D}^{1,q}(\mathds{R}^n)\}$. The space $\mathcal{D}^{1,q}(\mathds{R}^n)$ retains all constants functions if $q\geq n$. This is elucidated in \cite[page~48]{MR1461542} for $q=n$, while the case $q>n$ follows from the classical Morrey inequality.  Consequently, the inclusions in \eqref{immersion dirichlet} provide the desired fact. 

    On the other hand, when $1<p\leq q<n$, in light of the classical Sobolev Inequality (see, for instance, \cite[Corollary~1.77]{MR1461542}), the only constant function in  $\mathcal{D}^{1,p}(\mathds{R}^n)$ and $\mathcal{D}^{1,q}(\mathds{R}^n)$ is the zero function, which also occurs in $\mathcal{D}^{1,G}(\mathds{R}^n)$ by \eqref{immersion dirichlet}. Hence the assignment \eqref{seminorm Dirichlet-Orlicz-Sobolev space} defines a norm on $\mathcal{D}^{1,G}(\mathds{R}^n)$. Moreover, by using the standard mollifier functions, analysis similar to that in the proof of \cite[Theorem~6.4.4]{MR3931352} allows us to infer that $C_c^{\infty}(\mathds{R}^n)$ is dense in $\mathcal{D}^{1,G}(\mathds{R}^n)$, provided \eqref{seminorm Dirichlet-Orlicz-Sobolev space} is a norm.
\end{remark}

\begin{definition}\label{definition supersolution subsolution orlicz}
\begin{flushleft}
    {$\mathbf{(i)}$} A continuous function $u\in W_{\mathrm{loc}}^{1,G}(\Omega) $ is called  a $G$-harmonic function in $\Omega$ if it satisfies $\mathrm{div}\big({g(|\nabla u|)}/{|\nabla u|}\nabla u\big)=0$ weakly in $\Omega$, that is
\begin{equation*}
    \int_\Omega \frac{g(|\nabla u|)}{|\nabla u|}\nabla u\cdot\nabla \varphi\,\mathrm{d} x=0 \quad \forall \varphi\in C_c^{\infty}(\Omega).
\end{equation*}
    {$\mathbf{(ii)}$} We say that $u\in W_{\mathrm{loc}}^{1,G}(\Omega)$ is a $G$-supersolution in $\Omega$ if $-\mathrm{div}\big({g(|\nabla u|)}/{|\nabla u|}\nabla u\big)\geq0$ weakly, that is $u$ satisfies
    \begin{equation*}
    \int_\Omega \frac{g(|\nabla u|)}{|\nabla u|}\nabla u\cdot\nabla \varphi\,\mathrm{d} x\geq 0  \quad \forall \varphi\in C_c^{\infty}(\Omega),\,\varphi\geq 0.
\end{equation*}
Finally, we say that $u\in W_{\mathrm{loc}}^{1,G}(\Omega)$ is a $G$-subsolution in $\Omega$ if $-u$ is $G$-supersolution in $\Omega$. In order to simplify the notation, the term 
``$\,G$-'' is omitted when there is no ambiguity.
\end{flushleft}
\end{definition}
Existence and uniqueness of harmonic functions are proven in \cite{MR4274293, MR4258794, MR2834769}. The following lemma gives a version of the comparison principle for supersolutions and subsolutions \cite[Lemma~3.5]{MR4482108}. 
\begin{lemmaletter}\label{comparison principle super/subsolution}
    Let $u\in W_{\mathrm{loc}}^{1,G}(\Omega)$ be a supersolution and $v\in W_{\mathrm{loc}}^{1,G}(\Omega)$ be a subsolution in $\Omega$. If $\min\{u-v,0\}\in W_0^{1,G}(\Omega)$, then $u\geq v$ almost everywhere in $\Omega$.
\end{lemmaletter}
Let $\mu$ be a Radon measure (not necessarily nonnegative), and consider the following
quasilinear elliptic equation with data measure
 \begin{equation}\label{eq. general}
 -\mathrm{div}\bigg(\frac{g(|\nabla u|)}{|\nabla u|}\nabla u\bigg)=\mu  \quad \mbox{in}\quad \Omega.
\end{equation}
\begin{definition}\label{def supersolution}
    A function $u\in W_{\mathrm{loc}}^{1,G}(\Omega)$ is a solution to \eqref{eq. general} if
    \begin{equation*}
        \int_{\Omega}\frac{g(|\nabla u|)}{|\nabla u|}\nabla u\cdot \nabla\varphi\,\mathrm{d} x=\int_{\Omega}\varphi \,\mathrm{d}\mu \quad \forall\varphi\in C_c^{\infty}(\Omega).
    \end{equation*}
\end{definition}
Note that if $\mu$ is nonnegative, then a solution to \eqref{eq. general} is a supersolution, in sense of Definition~\ref{def supersolution}. The following result will be needed in Section~\ref{applications orlicz}, and it deals with 
the existence of solutions to \eqref{eq. general}. It is a consequence of \cite[Theorem~4.3]{MR4568881}. 
\begin{theoremletter}\label{existence in W0}
Let $\mu$ be a Radon measure in $\big(W_{0}^{1,G}(\Omega)\big)^{\ast}$. Then there exists a unique $u\in W_{0}^{1,G}(\Omega)$ satisfying \eqref{eq. general}. 
\end{theoremletter}

We introduce the notion of the $G$-capacity of a compact subset of $\Omega\subseteq\mathds{R}^n$ following \cite{MR4258773}. The $G$-capacity will be 
 used to ensure that if $u\in \mathcal{D}^{1,G}(\mathds{R}^n)$ is a solution to \eqref{eq. general}, then $u=0$ whether $q\geq n$ in \eqref{data condition elementary}.
\begin{definition}\label{definition G capacity}
    Let $E\subset \Omega$ be a compact subset, we define $\mathrm{cap}_G(E,\Omega)$, $G$-capacity of $E$ with respect to $\Omega$, by
    \begin{equation*}
        \mathrm{cap}_G(E,\Omega)=\inf\left\{\int_{\Omega}G(|\nabla \varphi|)\,\mathrm{d} x: \varphi\in C_c^{\infty}(\Omega),\; \varphi\geq 1 \; \mbox{on}\; E\right\}.
    \end{equation*}
    We set $\mathrm{cap}_G(E)=\mathrm{cap}_G(E,\mathds{R}^n)$ when $\Omega=\mathds{R}^n$.
\end{definition}
Observe that for $G_p(t)=t^p/p$, $\mathrm{cap}_{G_p}(\cdot, \Omega)$ coincides with the usual $p$-capacity with respect to $\Omega$, see for instance \cite{MR2305115, MR1411441, MR1461542}. In general, we may define equivalently
    \begin{equation*}
          \mathrm{cap}_G(E,\Omega)=\inf\left\{\int_{\Omega}G(|\nabla \varphi|)\,\mathrm{d} x: \varphi\in \mathcal{D}^{1,G}(\Omega),\; \varphi\geq 1 \; \mbox{in a neighborhood of}\; E\right\}.
    \end{equation*}
This follows by the same method as in the proof of \cite[Theorem~2.3 (iii)]{MR1461542}. Note that by Remark~\ref{remark q>n orlicz}, $\mathrm{cap}_G(E)=0$ for all compact set $E\subset\mathds{R}^n$, since the function $1$ belongs to $\mathcal{D}^{1,G}(\mathds{R}^n)$ if $q\geq n$ in \eqref{data condition elementary}.
\begin{remark}\label{measure and capacity orlicz}
        Suppose $q\geq n$ and let $\mu\in M^+(\mathds{R}^n)\cap\big(\mathcal{D}^{1,G}(\mathds{R}^n)\big)^{\ast}$, then a solution $u$ to \eqref{eq. general} in $\mathcal{D}^{1,G}(\mathds{R}^n)$ must be constant. To see this, we show first that $\mu$ must be \textit{absolutely continuous} with respect to the $G$-capacity, that is, $\mu(E)=0$ whenever $ \mathrm{cap}_G(E)=0$ for all compact sets $E\subset\mathds{R}^n$. Fix $E\subset\mathds{R}^n$ a compact set and let $\varphi\in C_c^{\infty}(\mathds{R}^n)$ such that $\varphi=1$ on $E$. Then, by testing \eqref{eq. general} with such $\varphi$ and combining Cauchy-Schwarz Inequality with Lemma~\ref{Holder-Orlicz},  it follows
        \begin{equation*}
     \begin{aligned}
    \mu(E)& =\int_E\varphi\,\mathrm{d}\mu\leq \int_{\mathds{R}^n}\varphi\,\mathrm{d}\mu\\
    & =\int_{\mathds{R}^n}\frac{g(|\nabla u|)}{|\nabla u|}\nabla u\cdot \nabla\varphi\,\mathrm{d} x \leq  2\|g(|\nabla u|)\|_{L^{G^{\ast}}}\|\nabla \varphi\|_{L^G}.
        \end{aligned}
       \end{equation*}
From \eqref{relation G and tilde G} and Lemma~\ref{lemma relation norm and modular}, we have 
\begin{equation*}
   \mu(E)\leq 2\,c\,\big(\rho_G(|\nabla u|)+ 1\big) \max\left\{\big(\rho_G(|\nabla \varphi|)\big)^{\frac{1}{p}},\big(\rho_G(|\nabla \varphi|)\big)^{\frac{1}{q}}\right\},
\end{equation*}
where $\rho_G(\cdot)$ is the modular function and $c=c(p,q)>0$. Consequently, 
\begin{equation*}
\mu(E)\leq C\,\max\left\{\big(\mathrm{cap}_G(E)\big)^{\frac{1}{p}},\big(\mathrm{cap}_G(E)\big)^{\frac{1}{q}}\right\},  
\end{equation*}
 with $C=C\big(p,q,\rho_G(|\nabla u|)\big)>0$.
Next, since $q\geq n$, $\mathrm{cap}_G(E)$ vanishes for all compact set $E\subset\mathds{R}^n$, whence $\mu =0$ by the inner regularity. From a Liouville-type theorem \cite[Theorem~4.1]{MR3918402}, $u$ must be constant.
    \end{remark}

\subsection{Superharmonic functions}
Let $\mu\in \big(W_{\mathrm{loc}}^{1,G}(\Omega)\big)^{\ast}$. Here, we extend the notion of the distributional solutions to \eqref{eq. general}, where $u$ does not necessarily belong to $W_{\mathrm{loc}}^{1,G}(\Omega)$.  To be more precise, we will understand solutions in the following potential-theoretic sense using $G$-superharmonic functions.

\begin{definition}\label{def superharmonic}

A function $u:\Omega\to(-\infty,\infty) \cup \{\infty\}$ is $G$-superharmonic in $\Omega$ if 
\begin{flushleft}
{$\mathbf{(i)}$} $u$ is lower semicontinuous,
\\
{$\mathbf{(ii)}$} $u$ is not identically infinite in any component of $\Omega$,
\\
{$\mathbf{(iii)}$} for each open subset  $V$ compactly contained in $\Omega$ and each harmonic function $h$ in $V$ such that $h \in  C(\overline{V})$  and $h\leq u$ in $\partial V$ implies $h\leq u$ in $V$.
\end{flushleft}
We denote $\mathcal{S}_G(\Omega)$, the class of all $G$-superharmonic functions in $\Omega$.
In order to simplify the notation, the term ``$\,G$-'' is omitted when there is no ambiguity.
\end{definition}
   
For $u \in \mathcal{S}_G(\Omega)$ we define its truncation as follows 
\begin{equation*}
    T_k(u)=\min (k,\max(u,-k)), \quad \forall k >0. 
\end{equation*}
We emphasize that  $\{T_k(u)\}$ is a sequence of supersolutions in $\Omega$ \cite[Lemma~4.6]{MR4482108}. Then there exists a unique measurable function $Z_u:\Omega\to\mathds{R}^n$ satisfying
\begin{equation*}
    Z_u(x)=\lim_{k\to \infty}\nabla \big(T_k(u)\big)(x) \quad \mbox{almost everywhere in }\Omega.
\end{equation*}
We denote $Z_u$ by $Du$ and call it a \textit{generalized gradient} of $u$. See details in \cite[Remark~4.13]{MR4482108}.
If $u\in W_{\mathrm{loc}}^{1,G}(\Omega)$, then clearly $Du=\nabla u$ since in this case $\nabla \big(T_k(u)\big)=\chi_{\{-k<u<k\}}\nabla u$. 

By the Riesz Representation Theorem \cite[Theorem~6.22]{MR1817225}, there exists a unique measure $\mu_k\in  {M}^+(\Omega)$ satisfying
\begin{equation*}
    \int_\Omega \frac{g(|\nabla T_k(u)|)}{|\nabla T_k(u)|}\nabla T_k(u)\cdot\nabla \varphi\,\mathrm{d} x =\int_\Omega \varphi\, \mathrm{d}\mu_k, \quad \forall\varphi\in C_c^{\infty}(\Omega), \varphi\geq 0.
\end{equation*}
On the other hand, from  \cite[Lemma~4.12]{MR4482108}, the sequence $\{g(\nabla T_k(u))\}$ is bounded in $L^1(B)$, for all open ball $B\subset\Omega$. Consequently, by Fatou's lemma, $g(|Du|)\in L_{\mathrm{loc}}^1(\Omega)$, and since $Du=\lim_k \nabla \big(T_k(u)\big)$ (pointwise), it holds
\begin{equation*}
    \int_\Omega \frac{g(|D u|)}{|D u|}D u\cdot\nabla \varphi\,\mathrm{d} x = \lim_{k\to \infty}\int_\Omega \frac{g(|\nabla T_k(u)|)}{|\nabla T_k(u)|}\nabla T_k(u)\cdot\nabla \varphi\,\mathrm{d} x, \quad \forall\varphi\in C_c^{\infty}(\Omega).
\end{equation*}
Therefore, using the Riesz Representation Theorem again, there exists a unique measure $\mu=\mu_u\in  {M}^+(\Omega)$ such that 
\begin{equation*}
\int_\Omega \frac{g(|D u|)}{|D u|}D u\cdot\nabla \varphi\,\mathrm{d} x=  \int_\Omega \varphi\, \mathrm{d}\mu, \quad \forall\varphi\in C_c^{\infty}(\Omega), \varphi\geq 0.
\end{equation*}
This means  
\begin{equation*}
-\mathrm{div}\bigg(\frac{g(|D u|)}{|D u|}D u\bigg)= \mu \quad \mbox{in}\quad \Omega.    
\end{equation*}
In the literature, $\mu_u$  is called the \textit{Riesz measure} of $u$.
   \begin{definition}\label{potential-theoretic sense} 
For $ \sigma \in {M}^+(\Omega)$, we say that $u$ is  a solution in the potential-theoretic sense to the equation
\begin{equation*}
  -\mathrm{div}\bigg(\frac{g(|\nabla u|)}{|\nabla u|}\nabla u\bigg)= \sigma \quad \mbox{in}\quad \Omega
\end{equation*}
   if $u \in \mathcal{S}_G(\Omega)$ and $\mu_u=\sigma$.
   \end{definition}
   Let $f(t)=g(t^{\gamma})$ with $\gamma>0$, $t\geq0$. In light of Definition~\ref{potential-theoretic sense}, if $\sigma\in  {M}^+(\Omega)$, then a function $u$ is a solution (in the potential-theoretic sense) to the equation
\begin{equation}
    -\mathrm{div}\bigg(\frac{g(|\nabla u|)}{|\nabla u|}\nabla u\bigg)= \sigma\,f(u) \quad \mbox{in}\quad \Omega
\end{equation}
whenever $u$ is nonnegative and satisfies
\begin{equation}
     \left\{
\begin{aligned}
& u\in \mathcal{S}_G(\Omega)\cap L_{\mathrm{loc}}^{f}(\Omega, \mathrm{d} \sigma), \\ 
& \mathrm{d}\mu_u=f(u)\mathrm{d}\sigma.
\end{aligned}
\right.
\end{equation}
\begin{definition}\label{def supersolution/superharmonic }
    Let $\sigma\in M^{+}(\mathds{R}^n)$. A function $u$ is a supersolution to \eqref{equation nonstandard} if $u$ is nonnegative and satisfies
    \begin{equation}\label{supersolution condition}
     \left\{
\begin{aligned}
& u\in \mathcal{S}_G(\mathds{R}^n)\cap L_{\mathrm{loc}}^{f}(\mathds{R}^n, \mathrm{d} \sigma), \\ 
&\int_{\mathds{R}^n} \frac{g(|D u|)}{|D u|}D u\cdot\nabla \varphi\,\mathrm{d} x \geq   \int_{\mathds{R}^n}\varphi f(u)\,\mathrm{d}\sigma \quad \forall \varphi\in C_c^{\infty}(\mathds{R}^n), \varphi \geq 0.
\end{aligned}
\right.
\end{equation}
Finally, the notion of solution to \eqref{equation nonstandard} is defined similarly by replacing ``$\geq$'' with ``$=$'' in \eqref{supersolution condition}.
\end{definition}
\begin{definition}
     Let $\sigma\in M^{+}(\mathds{R}^n)$. A function $u$ is a solution of finite energy to \eqref{equation nonstandard} whenever $u$ is a nonnegative solution to \eqref{equation nonstandard} and $u\in \mathcal{D}^{1,G}(\mathds{R}^n)\cap L^F(\mathds{R}^n,\mathrm{d}\sigma)$.
\end{definition}

Next, we will state some fundamental results of the potential theory of quasilinear elliptic equations with Orlicz growth,  like \eqref{eq. general}. We start with the following theorem that will be used to prove that a pointwise
limit of a sequence of superharmonic functions is a superharmonic function \cite[Theorem~2]{MR4482108}.
\begin{theoremletter}[Harnack’s Principle]\label{harnack principle orlicz5}
   Let $\{u_j\}$ be a sequence of superharmonic functions, with each $u_j$ finite almost everywhere in $\Omega$. If  $\{u_j\}$ is nondecreasing,
then the pointwise limit function $u = \lim_j u_j$ is superharmonic function, unless $u\equiv \infty$. Moreover,
if $u_j$ is nonnegative for all $j\geq 1$, then up to a subsequence, one has $D u=\lim_j Du_j$ in the set $\{u<\infty\}$.
\end{theoremletter}

The following theorem describes the main technical result, which will be decisive in linking supersolution to \eqref{equation nonstandard} with solutions to \eqref{wolff integral equation}, provided $g$ is the function given by \eqref{our g}. Its proof is postponed to Section~\ref{Potential estimates App}.
\begin{theorem}\label{Maly's type result}
    Let $g$ be the function given by \eqref{our g}. Suppose that $u$ is a superharmonic function in $B(x_0,2R)$, and let $\mu=\mu_u\in M^+(B(x_0,2R))$. Then there exist constants $C_1>0$ and $C_2>0$ depending only on $n, \,p$ and $q$ such that
    \begin{equation}\label{potential estimate thm}
        C_1\mathbf{W}_G^R\mu(x_0)\leq u(x_0)\leq C_2\big(\inf_{B(x_0,R)}u+\mathbf{W}_G^R\mu(x_0)\big).
    \end{equation}
    Here $\mathbf{W}_G^R\sigma$ is the $R$-truncated Wolff potential of $\sigma\in M^+(\mathds{R}^n)$ defined by
    \begin{equation*}
    \mathbf{W}_G^R\sigma(x)=\int_0^Rg^{-1}\bigg(\frac{\sigma(B(x,t))}{t^{n-1}}\bigg)\,\mathrm{d} t, \quad x\in\mathds{R}^n.
    \end{equation*}
\end{theorem}
It will need the following consequence of Theorem~\ref{Maly's type result} for our purpose.
\begin{corollary}\label{maly's result pratical}
    Let $g$ be the function given by \eqref{our g}. Suppose that $u$ is a superharmonic function in $\mathds{R}^n$ with $\inf_{\mathds{R}^n}u=0$, and let $\mu=\mu_u\in M^+(\mathds{R}^n)$. Then there exists constant $K\geq 1$ depending only on $n,\,p$ and $q$ such that
    \begin{equation*}
       K^{-1}\,\mathbf{W}_G\mu(x)\leq u(x)\leq K\,\mathbf{W}_G\mu(x) \quad \forall x\in\mathds{R}^n.
    \end{equation*}
\end{corollary}
\begin{proof}
    Fix $x\in\mathds{R}^n$. Then clearly $u$ is superharmonic in $B(x,2R)$ for all $R>0$. By Theorem~\ref{Maly's type result}, there exist constants $C_1=C_1(n,p,q)>0$ and $C_2=C_2(n,p,q)>0$ such that
    \begin{equation*}
         C_1\mathbf{W}_G^R\mu(x)\leq u(x)\leq C_2\big(\inf_{B(x,R)}u+\mathbf{W}_G^R\mu(x)\big) \quad \forall R>0.
    \end{equation*}
    Since $\inf_{\mathds{R}^n}u=0$, $\lim_{R\to \infty}\inf_{B(x,R)}u=0$. Consequently, being $C_1$ and $C_2$ independent of $R$, letting $R\to\infty$ in the previous bounds, we arrive at
    \begin{equation*}
        C_1\mathbf{W}_G\mu(x)\leq u(x)\leq C_2\mathbf{W}_G\mu(x).
    \end{equation*}
    Setting $K=\max\{C_2, (C_1)^{-1}, 1\}$, Corollary~\ref{maly's result pratical} is proved since $x$ was arbitrary.
\end{proof}

 \begin{remark}\label{remark supersolution and superharmonic}
       Combining \cite[Lemma~4.4]{MR4482108} with \cite[Theorem~7.16]{MR2305115}, we infer that for each $u$ supersolutiuon in $\Omega$ in the sense of Definition~\ref{def supersolution}, the function
       \begin{equation*}
           \Tilde{u}(x):= \mathrm{ess}\varliminf_{y\to x}u(x), \quad x\in\Omega,
       \end{equation*}
       is a superharmonic function in $\Omega$ satisfying $\Tilde{u}=u$ almost everywhere in $\Omega$. Thus each supersolution
       can be redefined in a set of measure zero such that the previous limit holds. From this, a supersolution will be treated as a superharmonic function. In particular, when $\Omega=\mathds{R}^n$, Corollary~\ref{maly's result pratical} holds for supersolutions in $\mathds{R}^n$. 
   \end{remark}

Next, we state a result suitable to apply to the case of Orlicz growth which gives a condition to nonnegative Radon measure $\mu$ belonging to $\big(W_0^{1, G}(\Omega)\big)^\ast$ in terms of the Wolff potential $\mathbf{W}_G\mu$, provided $\Omega$ is bounded and $\mathrm{supp}\,\mu \subset \Omega$. Here $\big(W_0^{1, G}(\Omega)\big)^\ast$ means the dual space of $W_0^{1, G}(\Omega)$. See 
\cite[Theorem~3]{MR4803728}, and also \cite[Theorem~1]{MR727526}.
\begin{theoremletter}\label{Inequality Wolff-Orlicz}
  Suppose that $\Omega$ is bounded. Let $\mu\in M^+(\Omega)$ with $\mathrm{supp}\,\mu\subset\Omega$. Then
  \begin{equation*}
      \mu\in \big(W_0^{1, G}(\Omega)\big)^{\ast} \Longleftrightarrow \int_{\Omega}\mathbf{W}_G^R\mu\,\mathrm{d}\mu< \infty \quad \mbox{for some }R>0.
  \end{equation*}
\end{theoremletter}

\section{Wolff potentials}\label{Wolff potentials section orlicz}

This section will be devoted to establishing preliminary results to the proof of Theorem~\ref{existencewolffequation}. 

From now on, $g$ is the function given by \eqref{our g}, with $p>1$, and $G$ is primitive. We fix
\begin{equation}\label{def of f and F}
    \left\{\begin{aligned}
    & f(t)=g(t^{\gamma}), \quad F(t)=\int_0^t f(s)\,\mathrm{d} s\quad \forall t\geq 0,\quad \mbox{that is} \\
        & f(t)=t^{(p-1)\gamma}+t^{(q-1)\gamma}, \quad F(t)=\frac{t^{(p-1)\gamma+1}}{(p-1)\gamma+1}+\frac{t^{(q-1)\gamma+1}}{(q-1)\gamma+1} \quad \forall t\geq 0,
    \end{aligned}
    \right.
\end{equation}
where $\gamma$ satisfies \eqref{condition gamma}. Using \eqref{assumption on g}, one can verify that
\begin{equation*}
    (p-1)\gamma\leq \frac{tf'(t)}{f(t)}\leq (q-1)\gamma, \quad (p-1)\gamma+1\leq \frac{tf(t)}{F(t)}\leq (q-1)\gamma+1 \quad \forall t>0.
\end{equation*}
From this, $L^{F}(\mathds{R}^n,\mathrm{d} \sigma)$ is a Banach space. Moreover, using \cite[Lemma~3.7.7]{MR3931352}, we can identify $L^{F}(\mathds{R}^n,\mathrm{d} \sigma)$ with $L^{(p-1)\gamma+1}(\mathds{R}^n,\mathrm{d} \sigma)\cap L^{(q-1)\gamma+1}(\mathds{R}^n,\mathrm{d} \sigma)$ as Banach spaces, where the intersection is  equipped with the norm
\begin{equation*}
    \|u\|_{L^{(p-1)\gamma+1}\cap L^{(q-1)\gamma+1}}=\max\left\{\|u\|_{L^{(p-1)\gamma+1}},\|u\|_{L^{(q-1)\gamma+1}}\right\}.
\end{equation*}
Notice that the function $g$ satisfies the ``sub-multiplicity'' condition
\begin{equation}\label{cond. submulti}
    g(a b)\leq g(a)g(b) \quad \forall a,b\geq 0,
\end{equation}
which implies a ``sup-multiplicity'' condition to $g^{-1}$:
\begin{equation}\label{cond. supmulti}
  g^{-1}(ab)\geq g^{-1}(a)g^{-1}(b) \quad \forall a,b\geq 0.  
\end{equation}
In the general context of $N$-functions, these conditions are known as $\Delta'$-\textit{contidion} (see for instance \cite[page~28]{MR1113700}).

We will consider supersolutions to the integral equation \eqref{wolff integral equation}.

\begin{definition}
     A \textit{supersolution} to \eqref{wolff integral equation} is a nonnegative function $u \in L_{\mathrm{loc}}^{f}(\mathds{R}^n,\mathrm{d}\sigma)$ which satisfies (pointwise)
\begin{equation}\label{integral equation-inequalitie}
    u\geq \mathbf{W}_{G}\big(f(u)\mathrm{d}\sigma\big) \quad \mbox{in}\quad \mathds{R}^n.
\end{equation}
The notion of \textit{solution} or \textit{subsolution} to \eqref{wolff integral equation} is defined similarly by replacing ``$\geq$'' by ``$=$'' or ``$\leq$'' in \eqref{integral equation-inequalitie}, respectively.
\end{definition}
In the following theorem, we obtain a lower bound for supersolutions to \eqref{wolff integral equation} in terms of Wolff potentials, whenever \eqref{cond. submulti} holds.
\begin{theorem}\label{estimativainferior orlicz}
  Let $\sigma\in M^+(\mathds{R}^n)$ with  $\mathbf{W}_G\sigma< \infty$ in $\mathds{R}^n$ and let $u$ be a nontrivial supersolution to \eqref{integral equation-inequalitie}. Then there exists a constant $0<C<1$ depending only $n$, $p,q$ and $\gamma$ such that
    \begin{equation}
        u(x)\geq C\left(\mathbf{W}_{G}\sigma(x)\right)^{\frac{1}{1-\gamma}} \quad \forall x\in\mathds{R}^n.
    \end{equation}
\end{theorem}

Let us show the following result before proving Theorem~\ref{estimativainferior orlicz}.
\begin{lemma}\label{estimativelambda}
  Fix $\alpha>0$ and let $\varphi(t)=g(t^{\alpha}), \; t\geq 0$. Let $\sigma\in M^+(\mathds{R}^n)$ with  $\mathbf{W}_G\sigma< \infty$ in $\mathds{R}^n$. Then there exists a constant  $0<\lambda<1$, which depends only on $n$, $p,q$ and $\alpha$, such that
    \begin{equation*}
\mathbf{W}_{G}\big(\varphi(\mathbf{W}_{G}\sigma)\mathrm{d} \sigma\big)(x)\geq \lambda\left(\mathbf{W}_{G}\sigma(x)\right)^{1+\alpha}, \quad \forall x\in \mathds{R}^n.
    \end{equation*}
    \end{lemma} 
    \begin{proof}[Proof of Lemma~\ref{estimativelambda}]
        By definition \eqref{definition wolff}, for any $t>0$
        \begin{equation*}
            \mathbf{W}_{G}\sigma(y)\geq \int_t^{\infty} g^{-1}\bigg(\frac{\sigma(B(y,s))}{s^{n-1}}\bigg)\mathrm{d}s, \quad \forall y\in\mathds{R}^n.
        \end{equation*}
        Notice that $B(y,2s)\supset B(x,s)$ for $y\in B(x,t)$ and $s\geq t$, whence by \eqref{estimate on g-1} and the previous estimate it holds
        \begin{align}
           \mathbf{W}_{G}\sigma(y) & \geq  \int_t^{\infty} g^{-1}\bigg(\frac{\sigma(B(y,s))}{s^{n-1}}\bigg)\,\mathrm{d}s = \int_{t/2}^{\infty} g^{-1}\bigg(\frac{\sigma(B(y,2s))}{(2s)^{n-1}}\bigg)\,\mathrm{d}s \nonumber\\
            &  \geq c_1 \int_{t/2}^{\infty} g^{-1}\bigg(\frac{\sigma(B(y,2s))}{s^{n-1}}\bigg)\,\mathrm{d}s  \geq  c_1  \int_{t}^{\infty} g^{-1}\bigg(\frac{\sigma(B(y,2s))}{s^{n-1}}\bigg)\,\mathrm{d}s\nonumber\\    
            & \geq  c_1 \int_{t}^{\infty} g^{-1}\bigg(\frac{\sigma(B(x,s))}{s^{n-1}}\bigg)\,\mathrm{d}s =: c_1\rho(t), \label{est. 8}
        \end{align}
        where $c_1=c_1(n,p,q)>0$. Since $\varphi(t)=g(t^{\alpha})$ is an increasing function, it follows from \eqref{est. 8} that
        \begin{align}          \mathbf{W}_{G}\big(\varphi(\mathbf{W}_{G}\sigma)\mathrm{d} \sigma\big)(x)  & \geq \int_0^{\infty}g^{-1}\bigg[\frac{1}{t^{n-1}}\int_{B(x,t)}\varphi\big(\mathbf{W}_{G}\sigma(y)\big)\mathrm{d}\sigma(y)\bigg]\,\mathrm{d} t\nonumber\\
            & \geq \int_0^{\infty}g^{-1}\bigg[\frac{1}{t^{n-1}}\int_{B(x,t)}\varphi(c_1 \rho(t))\mathrm{d}\sigma(y)\bigg]\,\mathrm{d} t\nonumber\\
            & = \int_0^{\infty}g^{-1}\bigg[\varphi(c_1 \rho(t))\frac{\sigma(B(x,t))}{t^{n-1}}\bigg]\,\mathrm{d} t\quad  \forall x\in \mathds{R}^n. \label{est. 9}
        \end{align}
        By \eqref{estimate on g} and \eqref{estimate on g-1}, we have respectively $\varphi(c_1 \rho(t))\geq c_2\varphi(\rho(t))$ and  $g^{-1}(c_2 \varphi(\rho(t)))\geq c_3 g^{-1}(\varphi(\rho(t)))$, where $c_2=c_2(n,p,q,\alpha)>0$ and $c_3=c_3(n,p,q,\alpha)>0$. Using these estimates in \eqref{est. 9}, with the aid of \eqref{cond. supmulti}, we obtain
        \begin{align}             \mathbf{W}_{G}\big(\varphi(\mathbf{W}_{G}\sigma)\mathrm{d} \sigma\big)(x)&\geq c_3 \int_0^{\infty}g^{-1}\bigg[\varphi(\rho(t))\frac{\sigma(B(x,t))}{t^{n-1}}\bigg]\,\mathrm{d} t \nonumber\\
             &\geq c_3\int_0^{\infty}g^{-1}(\varphi(\rho(t))) g^{-1}\bigg(\frac{\sigma(B(x,t))}{t^{n-1}}\bigg)\,\mathrm{d} t. \label{est. 10}
        \end{align}
       Note that $g^{-1}(\varphi(\rho(t)))=\rho(t)^{\alpha}$ and,  by Fundamental Theorem of Calculus, 
       \begin{equation*}
           \rho'(t)=\frac{\mathrm{d}}{\mathrm{d} t}\bigg( \int_{t}^{\infty} g^{-1}\bigg(\frac{\sigma(B(x,s))}{s^{n-1}}\bigg)\,\mathrm{d}s\bigg)=-g^{-1}\bigg(\frac{\sigma(B(x,t))}{t^{n-1}}\bigg),
       \end{equation*}
       Hence, we may rewrite \eqref{est. 10} as follows:
       \begin{equation*}
           \mathbf{W}_{G}\big(\varphi(\mathbf{W}_{G}\sigma)\mathrm{d} \sigma\big)(x)\geq c_3\int_0^{\infty}\rho(t)^{\alpha} \left(-\rho'(t)\right) \mathrm{d} t.
       \end{equation*}
      Integrating by parts, we concluded from the previous inequality that
      \begin{equation*}
               \mathbf{W}_{G}\big(\varphi(\mathbf{W}_{G}\sigma)\mathrm{d} \sigma\big)(x)\geq \frac{c_3}{1+\alpha}\rho(0)^{1+\alpha}=\frac{c_3}{1+\alpha}\left(\mathbf{W}_{G}\sigma(x)\right)^{1+\alpha}, 
      \end{equation*}
      which completes the proof of Lemma~\ref{estimativelambda} by taking
      \begin{equation}\label{const lambda.}
          \lambda=\frac{c_3}{1+\alpha}=\frac{1}{1+\alpha}\Big(\frac{p}{q}\Big)^{\frac{2}{p-1}}\Big(\frac{p}{2^{n-1}q}\Big)^{\frac{\alpha (q-1)}{(p-1)^2}}.
      \end{equation}
    \end{proof}
\begin{proof}[Proof of Theorem~\ref{estimativainferior orlicz}]
The main idea of the proof is to iterate the inequality \eqref{integral equation-inequalitie} with Lemma~\ref{estimativelambda}.
First, we prove the following claim.
\begin{claim}\label{claim1 orlicz}
Let $\sigma\in M^{+}(\mathds{R}^n)$  with  $\mathbf{W}_G\sigma< \infty$ in $\mathds{R}^n$. Suppose that $u$ is a nontrivial supersolution to \eqref{wolff integral equation} such that it holds
\begin{equation*}
    u(x)\geq c\left(\mathbf{W}_{G}\sigma(x)\right)^{\delta}, \quad x\in\mathds{R}^n,
\end{equation*}
where $0<c<1$ and $\delta>0$. Then
\begin{equation*}
    u(x)\geq \Big(\frac{p}{q}\Big)^{\frac{2}{p-1}}c^{\frac{\gamma(q-1)}{p-1}}\lambda\left(\mathbf{W}_{G}\sigma(x)\right)^{1+\delta\gamma} \quad x\in \mathds{R}^n,
\end{equation*}
where $\lambda$ is the constant given in Lemma~\ref{estimativelambda}.
\end{claim}
\noindent Indeed, a combination of \eqref{estimate on g} and \eqref{estimate on g-1}  with  Lemma~\ref{estimativelambda} gives
\begin{align*}
   u(x) & \geq \mathbf{W}_{G}\big(f(u)\mathrm{d}\sigma\big)\geq \mathbf{W}_{G}\big(f\big(c\big(\mathbf{W}_{G}\sigma(x)\big)^{\delta}\big)\mathrm{d}\sigma\big) \\
    & \geq \mathbf{W}_{G}\bigg(\frac{p}{q}c^{\gamma (q-1)}g(\big(\mathbf{W}_{G}\sigma(x)\big)^{\delta\gamma})\mathrm{d}\sigma\bigg)\\
    & \geq  \bigg(\frac{p}{q}\bigg)^{\frac{2}{p-1}}c^{\frac{\gamma(q-1)}{p-1}}\mathbf{W}_{G}\Big(g\big((\mathbf{W}_{G}\sigma(x))^{\delta\gamma}\big)\mathrm{d}\sigma\Big)\geq \bigg(\frac{p}{q}\bigg)^{\frac{2}{p-1}}c^{\frac{\gamma(q-1)}{p-1}} \lambda \left(\mathbf{W}_{G}\sigma(x)\right)^{1+\delta\gamma},
\end{align*}
which is our Claim~\ref{claim1 orlicz}. 

Now, fix $x\in \mathds{R}^n$ and $R>|x|$, and let $\sigma_B=\chi_B\sigma$, where $B=B(0,R)$.  Setting $\mathrm{d}\mu=f(u)\mathrm{d}\sigma$, we estimate $\mathbf{W}_{G}\mu(z)$ as follows
\begin{align*}
    \mathbf{W}_{G}\mu(z)&=\int_0^{\infty}g^{-1}\bigg(\frac{\mu (B(z,t))}{t^{n-1}}\bigg)\,\mathrm{d} t \geq \int_R^{\infty} g^{-1}\bigg(\frac{\mu (B(z,t))}{t^{n-1}}\bigg)\,\mathrm{d} t\\
    &=  \int_{R/2}^{\infty} g^{-1}\bigg(\frac{\mu (B(z,2t))}{(2t)^{n-1}}\bigg)2\,\mathrm{d} t\\
    &\geq c_1 \int_{R}^{\infty} g^{-1}\bigg(\frac{\mu (B(z,2t))}{(2t)^{n-1}}\bigg)\,\mathrm{d} t,
\end{align*}
where $c_1=c_1(n,p,q)>0$ was obtained in a light of \eqref{estimate on g-1}. Since $B(z,2t)\supset B(0,t)$ for $t\geq R$ and $z\in B$, it follows from the previous inequality that for all $z\in B$
\begin{equation}\label{est.theo 3.4 - 1}
     \mathbf{W}_{G}\mu(z)\geq c_1 \int_{R}^{\infty} g^{-1}\bigg(\frac{\mu (B(0,t))}{(2t)^{n-1}}\bigg)\mathrm{d} t=: A(R).
\end{equation}
We may assume $A(R)<1$ for $R>0$ large enough. Thus, iterating \eqref{integral equation-inequalitie} with \eqref{est.theo 3.4 - 1}, we obtain for all $x\in \mathds{R}^n$
\begin{align}
 u(x)   & \geq  \mathbf{W}_{G}(f(\mathbf{W}_{G}\mu)\mathrm{d} \sigma_B)(x) \geq  \mathbf{W}_{G}(f(A(R))\mathrm{d} \sigma_B)(x) \nonumber\\
    & \geq \int_{0}^{\infty} g^{-1}\bigg(\frac{f(A(R))\sigma_B(B(x,t))}{t^{n-1}}\bigg)\,\mathrm{d} t \nonumber\\
    & \geq A(R)^{\gamma}\,\mathbf{W}_{G}{\sigma_B}(x), \label{est.theo 3.4 - 2}
\end{align}
where in the last line was used \eqref{cond. supmulti}. Setting  $c_1=A(R)^{\gamma}$ and $\delta_1=1$, by Claim~\ref{claim1 orlicz}, with $\sigma_B$ in place of $\sigma$, and \eqref{est.theo 3.4 - 2}, we construct a sequence of lower bounds for $u$ as follows:
\begin{equation}\label{est.theo 3.4 - 3}
    u(x)\geq c_j\left(\mathbf{W}_{G}\sigma_B(x)\right)^{\delta_j},\quad x\in\mathds{R}^n,
\end{equation}
where for $j=2,3,\ldots$, $\delta_j$ and $c_j$ are given by
\begin{equation}
    \begin{aligned}
    \delta_j&=1+\gamma \delta_{j-1},\label{seqdelta_j} \\
    c_j&= \lambda\,\Big(\frac{p}{q}\Big)^{\frac{2}{p-1}}c_{j-1}^{\frac{\gamma(q-1)}{p-1}}.   
    \end{aligned}
\end{equation}
Since $0<\gamma<(p-1)/(q-1)\leq 1$, letting the limit $j\to \infty$ in \eqref{seqdelta_j}, it is straightforward to conclude that
\begin{align}
    \lim_{j\to \infty}\delta_j&=\frac{1}{1-\gamma}, \label{limit delta_j} \\
    \lim_{j\to \infty}c_j&= \lambda^{\frac{p-1}{p-1 - \gamma (q-1)}}\Big(\frac{p}{q}\Big)^{\frac{2(p-1)}{(p-1)(q-1)-\gamma (q-1)^2}}=: C.\nonumber
\end{align}
Passing to the limit as $j\to \infty$ in \eqref{est.theo 3.4 - 3}, we deduce
\begin{equation*}
     u(x)\geq C\left(\mathbf{W}_{G}\sigma_B(x)\right)^{\frac{1}{1-\gamma}} \quad \forall x\in\mathds{R}^n.
\end{equation*}
Since $C$ does not depend on $R$, the proof of Theorem~\ref{estimativainferior orlicz} is established after letting $R\to \infty$ in the previous inequality.
\end{proof}

Suppose that $\mathbf{W}_G\sigma< \infty$ in $\mathds{R}^n$. In the view of Theorem~\ref{estimativainferior orlicz}, if $u\in L^F(\mathds{R}^n,\mathrm{d}\sigma)$ is a solution to \eqref{wolff integral equation}, then $(\mathbf{W}_{G}\sigma)^{{1}/{(1-\gamma)}}\in L^{F}(\mathds{R}^n,\mathrm{d}\sigma)$, that is  
\begin{equation}\label{necessarycondition}
    \int_{\mathds{R}^n}F\big(\mathbf{W}_G\sigma^{\frac{1}{1-\gamma}}\big)\,\mathrm{d}\sigma < \infty.
\end{equation}
Hence condition \eqref{necessarycondition} is necessary to the existence of solutions to \eqref{wolff integral equation} in $L^F(\mathds{
R}^n,\mathrm{d}\sigma)$. However, this condition is far to be sufficient (at least) to ensure such existence in $L^{F}(\mathds{R}^n,\mathrm{d}\sigma)$.  We will show that \eqref{cond sufficient to exist} is a sufficient condition to the existence of a solution to \eqref{wolff integral equation} in $L^{F}(\mathds{R}^n,\mathrm{d}\sigma)$. Before that, the following result will be needed.
\begin{lemma}\label{technical lemma orlicz} 
    Let $\sigma\in M^+(\mathds{R}^n)$ satisfying \eqref{cond sufficient to exist}. Then there exists a constant $c>0$ such that for all $u\in L^F(\mathds{R}^n,\mathrm{d} \sigma)$, $u\geq 0$, it holds
\begin{multline*}
\int_{\mathds{R}^n}F\big(\mathbf{W}_G(f(u)\mathrm{d}\sigma)\big)\,\mathrm{d}\sigma \\ \leq c\bigg[\bigg(\int_{\mathds{R}^n}F(u)\,\mathrm{d}\sigma\bigg)^{\frac{p-1}{q-1}\gamma}+\bigg(\int_{\mathds{R}^n}F(u)\,\mathrm{d}\sigma\bigg)^{\gamma}+\bigg(\int_{\mathds{R}^n}F(u)\,\mathrm{d}\sigma\bigg)^{\frac{q-1}{p-1}\gamma}\bigg].    
\end{multline*}
The constant $c$ depends only on $n$, $p,q$, and the $L^F$-norms of the Wolff potentials mentioned in \eqref{cond sufficient to exist}.
\end{lemma}
\begin{proof}
    We begin the proof with the following claim.
    \begin{claim}\label{claim2 orlicz}
        Fix $1<s<\infty$, $0<r<s-1$ and $\alpha>r-1$. Let $\sigma\in M^+(\mathds{R}^n)$ satisfying 
        \begin{equation*}
            \big(\mathbf{W}_s\sigma\big)^{\frac{s-1}{s-1-r}}\in L^{1+\alpha}(\mathds{R}^n,\mathrm{d}\sigma).
        \end{equation*}
         Then there exists a constant $c_0=c_0 \Big(n,r,s, \alpha,\|\big(\mathbf{W}_s\sigma\big)^{(s-1)/(s-1-r)}\|_{L^{1+\alpha}}\Big)>0$ such that for all $u\in L^{1+\alpha}(\mathds{R}^n,\mathrm{d} \sigma)$, $u\geq 0$, it holds
        \begin{equation*}
\int_{\mathds{R}^n}\big(\mathbf{W}_s(u^r\mathrm{d}\sigma)\big)^{1+\alpha}\,\mathrm{d}\sigma\leq c_0\bigg(\int_{\mathds{R}^n}u^{1+\alpha}\,\mathrm{d}\sigma\bigg)^{\frac{r}{s-1}}.
        \end{equation*}
    \end{claim}
    \noindent Indeed, fix $0\leq u\in L^{1+\alpha}(\mathds{R}^n,\mathrm{d} \sigma)$. By definition \eqref{wolff potential standard}, we have
    \begin{equation*}
        \begin{aligned}
           \mathbf{W}_s(u^r\mathrm{d}\sigma)(x) &= \int_0^{\infty}\bigg(\frac{\int_{B(x,t)}u^r\,\mathrm{d}\sigma}{t^{n-1}}\bigg)^{\frac{1}{s-1}}\,\mathrm{d} t \\
            &\leq  \int_0^{\infty}\bigg(M_{\sigma}u^r(x)\frac{\sigma(B(x,t))}{t^{n-1}}\bigg)^{\frac{1}{s-1}}\,\mathrm{d} t\\
            &= \big(M_{\sigma}u^r(x)\big)^{\frac{1}{s-1}}\mathbf{W}_s\sigma(x) \quad \forall x\in \mathds{R}^n, 
        \end{aligned}
    \end{equation*}
    where $M_{\sigma}\cdot$ is the centered maximal operator defined by
    \begin{equation*}
        M_{\sigma}v(x)=\sup_{t>0}\frac{1}{\sigma(B(x,t))}\int_{B(x,t)}v\,\mathrm{d}\sigma, \quad v\in L_{\mathrm{loc}}^1(\mathds{R}^n,\mathrm{d}\sigma).
    \end{equation*}
Then using the classical H\"{o}lder's inequality with exponents $\beta=(s-1)/r$ and $\beta'=(s-1)/(s-1-r)$ in the preceding inequality, we obtain
\begin{align}
   \int_{\mathds{R}^n} \big(\mathbf{W}_s(u^r\mathrm{d}\sigma)\big)^{1+\alpha}\,\mathrm{d}\sigma &\leq \int_{\mathds{R}^n}\big(M_{\sigma}u^r\big)^{\frac{1+\alpha}{s-1}}\big(\mathbf{W}_s\sigma\big)^{1+\alpha}\,\mathrm{d}\sigma  \nonumber\\
    &\leq  \bigg(\int_{\mathds{R}^n}\big(\mathbf{W}_s\sigma\big)^{\frac{(s-1)(1+\alpha)}{s-1-r}}\,\mathrm{d}\sigma\bigg)^\frac{s-1-r}{s-1}\bigg(\int_{\mathds{R}^n}\big(M_{\sigma}u^r\big)^{\frac{1+\alpha}{r}}\,\mathrm{d}\sigma\bigg)^\frac{r}{s-1}.\label{estimate1 orlicz}
\end{align}
Being $(1+\alpha)/r>1$,   $M_{\sigma}:L^{(1+\alpha)/r}(\mathds{R}^n,\mathrm{d}\sigma)\to L^{(1+\alpha)/r}(\mathds{R}^n,\mathrm{d} \sigma)$ is a bounded operator (see for instance \cite[Theorem~1.22]{MR1461542}),  that is, there exists a constant $\tilde{c}=\tilde{c}(n,r,\alpha)>0$ such that
\begin{equation*}
    \bigg(\int_{\mathds{R}^n}\big(M_{\sigma}u^r\big)^{\frac{1+\alpha}{r}}\,\mathrm{d}\sigma\bigg)^\frac{r}{s-1}\leq \tilde{c}^{\frac{r}{s-1}}\bigg(\int_{\mathds{R}^n}\big(u^r\big)^{\frac{1+\alpha}{r}}\,\mathrm{d}\sigma\bigg)^\frac{r}{s-1}.
\end{equation*}
Using this in \eqref{estimate1 orlicz}, we arrive at
\begin{equation*}
    \int_{\mathds{R}^n} \big(\mathbf{W}_s(u^r\mathrm{d}\sigma)\big)^{1+\alpha}\,\mathrm{d}\sigma \leq c_0 \bigg(\int_{\mathds{R}^n}u^{1+\alpha}\,\mathrm{d}\sigma\bigg)^\frac{r}{s-1},
\end{equation*}
with $c_0=c_0 \Big(n,r,s, \alpha,\|\big(\mathbf{W}_s\sigma\big)^{(s-1)/(s-1-r)}\|_{L^{1+\alpha}}\Big)>0$, which proves Claim~\ref{claim2 orlicz}. 

Next, fix $0\leq u\in L^F(\mathds{R}^n,\mathrm{d}\sigma)$. Note that from \eqref{estimate g sum} and \eqref{def of f and F},  there exists a constant $c_1=c_1(p,q,\gamma)>0$ such that
\begin{equation*}
    \begin{aligned}
         \mathbf{W}_G\sigma (x) &\leq c_1\big(\mathbf{W}_p\sigma(x)+\mathbf{W}_q\sigma(x)\big) \quad \forall x\in\mathds{R}^n,\\
         F(t)&\leq c_1(t^{(p-1)\gamma+1}+t^{(q-1)\gamma+1}) \quad \forall t\geq 0.
    \end{aligned}
\end{equation*}
Then, combining the previous inequalities, we can show that
\begin{multline}
\int_{\mathds{R}^n}F\big(\mathbf{W}_G(f(u)\mathrm{d}\sigma)\big)\,\mathrm{d}\sigma  \\\leq  \int_{\mathds{R}^n}\big(\mathbf{W}_G(f(u)\mathrm{d}\sigma)\big)^{(p-1)\gamma+1}\mathrm{d}\sigma +\int_{\mathds{R}^n}\big(\mathbf{W}_G(f(u)\mathrm{d}\sigma)\big)^{(q-1)\gamma+1}\mathrm{d}\sigma\\
    \leq  c_2\bigg[\int_{\mathds{R}^n}\big(\mathbf{W}_p(f(u)\mathrm{d}\sigma)\big)^{(p-1)\gamma+1}\mathrm{d}\sigma + \int_{\mathds{R}^n}\big(\mathbf{W}_q(f(u)\mathrm{d}\sigma)\big)^{(p-1)\gamma+1}\mathrm{d}\sigma\\
    + \int_{\mathds{R}^n}\big(\mathbf{W}_p(f(u)\mathrm{d}\sigma)\big)^{(q-1)\gamma+1}\mathrm{d}\sigma + \int_{\mathds{R}^n}\big(\mathbf{W}_q(f(u)\mathrm{d}\sigma)\big)^{(q-1)\gamma+1}\bigg], \label{estimate2 orlicz}
\end{multline}
where $c_2=c_2(p,q)>0$. 
We shall make use of the following elementary inequality \cite[Lemma~1.1]{MR1461542}:
given $\delta>0$, for all $a,b\in \mathds{R}$ it holds
\begin{equation*}
    |a+b|^{\delta}\leq 2^{\delta-1}(|a|^{\delta}+|b|^{\delta}). 
\end{equation*}
Because of this inequality, reminding of the definition of $f(t)$ in \eqref{def of f and F}, we will split each integral in \eqref{estimate2 orlicz} into another two ones:
\begin{multline}
\int_{\mathds{R}^n}F\big(\mathbf{W}_G(f(u)\mathrm{d}\sigma)\big)\,\mathrm{d}\sigma  \\\leq c_2\bigg[\int_{\mathds{R}^n}\big(\mathbf{W}_p(f(u)\mathrm{d}\sigma)\big)^{(p-1)\gamma+1}\mathrm{d}\sigma + \int_{\mathds{R}^n}\big(\mathbf{W}_q(f(u)\mathrm{d}\sigma)\big)^{(p-1)\gamma+1}\mathrm{d}\sigma\\
    + \int_{\mathds{R}^n}\big(\mathbf{W}_p(f(u)\mathrm{d}\sigma)\big)^{(q-1)\gamma+1}\mathrm{d}\sigma + \int_{\mathds{R}^n}\big(\mathbf{W}_q(f(u)\mathrm{d}\sigma)\big)^{(q-1)\gamma+1}\bigg]\\
    \leq c_3\bigg[ \int_{\mathds{R}^n}\big(\mathbf{W}_p(u^{(p-1)\gamma}\mathrm{d}\sigma)\big)^{(p-1)\gamma+1}\mathrm{d}\sigma + \int_{\mathds{R}^n}\big(\mathbf{W}_p(u^{(q-1)\gamma}\mathrm{d}\sigma)\big)^{(p-1)\gamma+1}\mathrm{d}\sigma \\
    + \int_{\mathds{R}^n}\big(\mathbf{W}_q(u^{(p-1)\gamma}\mathrm{d}\sigma)\big)^{(p-1)\gamma+1}\mathrm{d}\sigma + \int_{\mathds{R}^n}\big(\mathbf{W}_q(u^{(q-1)\gamma}\mathrm{d}\sigma)\big)^{(p-1)\gamma+1}\mathrm{d}\sigma\\
    + \int_{\mathds{R}^n}\big(\mathbf{W}_p(u^{(q-1)\gamma}\mathrm{d}\sigma)\big)^{(p-1)\gamma+1}\mathrm{d}\sigma + \int_{\mathds{R}^n}\big(\mathbf{W}_p(u^{(q-1)\gamma}\mathrm{d}\sigma)\big)^{(q-1)\gamma+1}\mathrm{d}\sigma\\
    + \int_{\mathds{R}^n}\big(\mathbf{W}_q(u^{(q-1)\gamma}\mathrm{d}\sigma)\big)^{(p-1)\gamma+1}\mathrm{d}\sigma + \int_{\mathds{R}^n}\big(\mathbf{W}_q(u^{(q-1)\gamma}\mathrm{d}\sigma)\big)^{(q-1)\gamma+1}\mathrm{d}\sigma\bigg]\\
    =:c_3\sum_{j=1}^{8}I_j,\label{estimate3 orlicz}
\end{multline}
where $c_3=c_3(\gamma,p,q)>0$. By assumption on $\gamma$, one has $(p-1)\gamma/(q-1)<\gamma<(q-1)\gamma/(p-1)<1$ and $(p-1)\gamma+1>(q-1)\gamma$. Hence we estimate $I_1,\,I_2,\,,I_3$ and $I_4$ by applying Claim~\ref{claim2 orlicz} with $\alpha=(p-1)\gamma$, $r_1=(p-1)\gamma$, $s_1=p$, $r_2=(q-1)\gamma$, $s_2=p$, $r_3=(p-1)\gamma$, $s_3=q$ and $r_4=(q-1)\gamma$, $s_4=q$, to deduce
\begin{equation}\label{estimate4 orlicz}
    \begin{aligned}
        I_1&= \int_{\mathds{R}^n}\big(\mathbf{W}_p(u^{(p-1)\gamma}\mathrm{d}\sigma)\big)^{(p-1)\gamma+1}\mathrm{d}\sigma \leq c_{0,1}\bigg(\int_{\mathds{R}^n}u^{(p-1)\gamma+1}\,\mathrm{d}\sigma\bigg)^{\gamma},\\
        I_2&= \int_{\mathds{R}^n}\big(\mathbf{W}_p(u^{(q-1)\gamma}\mathrm{d}\sigma)\big)^{(p-1)\gamma+1}\mathrm{d}\sigma \leq c_{0,2}\bigg(\int_{\mathds{R}^n}u^{(p-1)\gamma+1}\,\mathrm{d}\sigma\bigg)^{\frac{q-1}{p-1}\gamma},\\
        I_3&= \int_{\mathds{R}^n}\big(\mathbf{W}_q(u^{(p-1)\gamma}\mathrm{d}\sigma)\big)^{(p-1)\gamma+1}\mathrm{d}\sigma \leq c_{0,3}\bigg(\int_{\mathds{R}^n}u^{(p-1)\gamma+1}\,\mathrm{d}\sigma\bigg)^{\frac{p-1}{q-1}\gamma},\\
        I_4&= \int_{\mathds{R}^n}\big(\mathbf{W}_q(u^{(q-1)\gamma}\mathrm{d}\sigma)\big)^{(p-1)\gamma+1}\mathrm{d}\sigma \leq c_{0,4}\bigg(\int_{\mathds{R}^n}u^{(p-1)\gamma+1}\,\mathrm{d}\sigma\bigg)^{\gamma}.
    \end{aligned}
\end{equation}
Similarly, to estimate $I_5,\,I_6,\,I_7$ and $I_8$, we apply Claim~\ref{claim2 orlicz} with $\alpha=(q-1)\gamma$, $r_5=(p-1)\gamma$, $s_5=p$, $r_6=(q-1)\gamma$, $s_6=p$, $r_7=(p-1)\gamma$, $s_7=q$ and $r_8=(q-1)\gamma$, $s_8=q$, to deduce
\begin{equation}\label{estimate5 orlicz}
    \begin{aligned}
        I_5&= \int_{\mathds{R}^n}\big(\mathbf{W}_p(u^{(p-1)\gamma}\mathrm{d}\sigma)\big)^{(q-1)\gamma+1}\mathrm{d}\sigma \leq c_{0,5}\bigg(\int_{\mathds{R}^n}u^{(q-1)\gamma+1}\,\mathrm{d}\sigma\bigg)^{\gamma},\\
        I_6&= \int_{\mathds{R}^n}\big(\mathbf{W}_p(u^{(q-1)\gamma}\mathrm{d}\sigma)\big)^{(q-1)\gamma+1}\mathrm{d}\sigma \leq c_{0,6}\bigg(\int_{\mathds{R}^n}u^{(q-1)\gamma+1}\,\mathrm{d}\sigma\bigg)^{\frac{q-1}{p-1}\gamma},\\
        I_7&= \int_{\mathds{R}^n}\big(\mathbf{W}_q(u^{(p-1)\gamma}\mathrm{d}\sigma)\big)^{(q-1)\gamma+1}\mathrm{d}\sigma \leq c_{0,7}\bigg(\int_{\mathds{R}^n}u^{(q-1)\gamma+1}\,\mathrm{d}\sigma\bigg)^{\frac{p-1}{q-1}\gamma},\\
        I_8&= \int_{\mathds{R}^n}\big(\mathbf{W}_q(u^{(q-1)\gamma}\mathrm{d}\sigma)\big)^{(q-1)\gamma+1}\mathrm{d}\sigma \leq c_{0,8}\bigg(\int_{\mathds{R}^n}u^{(q-1)\gamma+1}\,\mathrm{d}\sigma\bigg)^{\gamma}.
    \end{aligned}
\end{equation}
Here the constants $c_{0,j}>0$, $j=1,\ldots,8$, are given by Claim~\ref{claim2 orlicz}, whose depend only on $n,\,p,\,q,\,\gamma$ and the $L^F$-norms of the Wolff potentials presented in \eqref{cond sufficient to exist}. Since $\max\{t^{(p-1)\gamma+1},\, t^{(q-1)\gamma+1}\}\leq c_4 F(t)$ for all $t>0$, where $c_4=c_4(p,q,\gamma)>0$, a combination of \eqref{estimate3 orlicz} with \eqref{estimate4 orlicz} and \eqref{estimate5 orlicz}, completes the proof of Lemma~\ref{technical lemma orlicz}. \end{proof}

\begin{remark} \label{remark cond sufficient to exist idea}
Because of \eqref{estimate g sum}, $\mathbf{W}_G\sigma$ is controlled from above by the sum of $\mathbf{W}_p\sigma$ with $\mathbf{W}_q\sigma$. Consequently, \eqref{cond sufficient to exist} implies that
    \begin{equation}\label{cond sufficient to exist idea}
        \mathbf{W}_G\sigma^{\frac{1}{1-\gamma}},\, \mathbf{W}_G\sigma^{\frac{p-1}{p-1-\gamma(q-1)}},\,\mathbf{W}_G\sigma^{\frac{q-1}{q-1-\gamma(p-1)}}\in  L^F(\mathds{R}^n,\mathrm{d}\sigma).
    \end{equation}
  It would be desirable to show that \eqref{cond sufficient to exist idea} is a sufficient condition to ensure the existence of solutions to \eqref{wolff integral equation} in $L^F(\mathds{R}^n,\mathrm{d} \sigma)$, but we have not been able to do
this.     
\end{remark}

\subsection{Proof of Theorem~\ref{existencewolffequation}}
Let $\sigma\in M^+(\mathds{R}^n)$ satisfying \eqref{cond sufficient to exist}. Our proof starts with the assertion that there exists a constant  $\varepsilon>0$ sufficiently small such that
\begin{equation}
    u_0(x)=\varepsilon \,(\mathbf{W}_{G}\sigma(x))^{\frac{1}{1-\gamma}}, \quad x\in\mathds{R}^n,
\end{equation}
is a subsolution to \eqref{wolff integral equation}. Indeed, combining Lemma~\ref{estimativelambda} with \eqref{estimate on g-1}, we obtain
\begin{equation*}
    \mathbf{W}_{G}(f(u_0)\,\mathrm{d}\sigma)\geq \Big(\frac{p}{q}\Big)^{\frac{2}{p-1}}\lambda\,\varepsilon^{\frac{(q-1)\gamma}{p-1}}(\mathbf{W}_{G}\sigma)^{\frac{1}{1-\gamma}},
\end{equation*}
and consequently,  picking $\varepsilon>0$ such that  
\begin{equation*}
    \varepsilon\leq \Big(\Big(\frac{p}{q}\Big)^{\frac{2}{p-1}}\lambda\Big)^{\frac{p-1}{p-1-\gamma(q-1)}},
\end{equation*}
we concluded that $\mathbf{W}_{G}(f(u_0)\,\mathrm{d}\sigma)\geq u_0$, which is our assertion.

Now, let us construct a sequence of iterations of functions $u_j$ as follows:
\begin{equation}
    u_j=\mathbf{W}_{G}(f(u_{j-1})\,\mathrm{d}\sigma), \quad j=1,2,\ldots
\end{equation}
From the previous assertion, $u_1\geq u_0$ in $\mathds{R}^n$. Arguing by induction, one has  $u_{j-1}\leq u_{j}$ in $\mathds{R}^n$ for all $j\geq 1$. By assumption \eqref{cond sufficient to exist}, $u_0\in L^F(\mathds{R}^n,\mathrm{d}\sigma)$. Consequently, we can verify by induction that $u_j\in L^F(\mathds{R}^n,\mathrm{d}\sigma)$ for all $j=1,2,\ldots$. Indeed, suppose that $u_{j-1}\in L^F(\mathds{R}^n,\mathrm{d}\sigma)$ for some $j\geq 1$. Using Lemma~\ref{technical lemma orlicz}, we have
\begin{multline*}
\int_{\mathds{R}^n}F(u_j)\,\mathrm{d}\sigma= \int_{\mathds{R}^n}F\big(\mathbf{W}_G(f(u_{j-1})\mathrm{d}\sigma)\big)\,\mathrm{d}\sigma\\
\leq c\bigg[\bigg(\int_{\mathds{R}^n}F(u_{j-1})\,\mathrm{d}\sigma\bigg)^{\frac{p-1}{q-1}\gamma}+\bigg(\int_{\mathds{R}^n}F(u_{j-1})\,\mathrm{d}\sigma\bigg)^{\gamma}+\bigg(\int_{\mathds{R}^n}F(u_{j-1})\,\mathrm{d}\sigma\bigg)^{\frac{q-1}{p-1}\gamma}\bigg]<\infty.    
\end{multline*}
This shows $u_{j}\in L^F(\mathds{R}^n,\mathrm{d}\sigma)$. Furthermore, being $\{u_j\}$ an increasing sequence (pointwise), it follows from the preceding inequality that
\begin{multline}\label{estimate6 orlicz}
    \int_{\mathds{R}^n}F(u_j)\,\mathrm{d}\sigma\\
\leq c\bigg[\bigg(\int_{\mathds{R}^n}F(u_{j})\,\mathrm{d}\sigma\bigg)^{\frac{p-1}{q-1}\gamma}+\bigg(\int_{\mathds{R}^n}F(u_{j})\,\mathrm{d}\sigma\bigg)^{\gamma}+\bigg(\int_{\mathds{R}^n}F(u_{j})\,\mathrm{d}\sigma\bigg)^{\frac{q-1}{p-1}\gamma}\bigg] \quad \forall j\geq 1.
\end{multline}
We claim that $\{u_j\}$ is uniformly bounded in $L^F(\mathds{R}^n,\mathrm{d}\sigma)$. Indeed, consider the real continuous function on $[0,\infty)$
\begin{equation*}
    h (t)=t-c\big(t^{\frac{p-1}{q-1}\gamma}+t^{\gamma}+t^{\frac{q-1}{p-1}\gamma}\big).
\end{equation*}
Notice that \eqref{estimate6 orlicz} is equivalently to $h\big( \int_{\mathds{R}^n}F(u_j)\,\mathrm{d}\sigma\big)\leq 0$ for all $j\geq 1$. Clearly, $h(0)=0$.
Since $(p-1)\gamma/(q-1)<\gamma<(q-1)\gamma/(p-1)<1$, $h(t)$ decreases for $t$ sufficiently small. Also,
\begin{equation*}
    \lim_{t\to \infty}h(t)=\infty.
\end{equation*}
Hence the subset $\{t\geq 0: h(t)\leq 0\}$ is bounded in $[0,\infty)$, that is, there exists a constant $C=C(p,q,\gamma,c)>0$ such that $h(t)\leq 0$ if and only if $t\leq C$. Thus
\begin{equation*}
     \int_{\mathds{R}^n}F(u_j)\,\mathrm{d}\sigma\leq C \quad \forall j\geq 1.
\end{equation*}

Therefore, letting $j\to\infty$ in the previous inequality, by the Monotone Convergence Theorem, there exists $u=\lim_{j\to\infty}u_j$ such that $u\in L^{F}(\mathds{R}^n,\,\mathrm{d}\sigma)$. Combining H\"{o}lder's inequality (Lemma~\ref{Holder-Orlicz}) with \eqref{relation G and tilde G}, $L^F(\mathds{R}^n,\mathrm{d}\sigma)\subset L_{\mathrm{loc}}^f(\mathds{R}^n,\mathrm{d}\sigma)$, whence $u$ satisfies \eqref{wolff integral equation}. This completes the proof of Theorem~\ref{existencewolffequation}.

\section{Applications}\label{applications orlicz}
In this section, we prove Theorem~\ref{solution to orlicz-measure eq}. Before that, we state and prove some preliminary results concerned with \eqref{equation nonstandard}. It is worth pointing out that these results hold for any $N$-functions $G$ which enjoy the property \eqref{data condition elementary}. First, we will use the following lemma to ensure the existence of solutions to \eqref{equation nonstandard}.
\begin{lemma}\label{wolff inequality lemma}
    Let $v\in L^{F}(\mathds{R}^n,\mathrm{d}\sigma)$ be a supersolution to \eqref{wolff integral equation}. Then $f(v)\,\mathrm{d}\sigma\in \big(\mathcal{D}^{1,G}(\mathds{R}^n)\big)^{\ast}$.
\end{lemma}
\begin{proof}
    Let $\omega=\sigma\,f(v)\in M^+(\mathds{R}^n)$. We need to prove there exists a constant $c>0$ such that
    \begin{equation}\label{estimate7 orlicz}
        \bigg|\int_{\mathds{R}^n}\varphi\,\mathrm{d}\omega\bigg|\leq c\,\|\nabla \varphi\|_{L^G} \quad \forall \varphi\in C_c^{\infty}(\mathds{R}^n).
    \end{equation}
Since $v\geq \mathbf{W}_G(f(v)\mathrm{d}\sigma)$ in $\mathds{R}^n$ with $v\in L^F(\mathds{R}^n,\,\mathrm{d}\sigma)$, it follows from \eqref{def of f and F} that
\begin{align}
    \int_{\mathds{R}^n}\mathbf{W}_G\omega\,\mathrm{d}\omega &= \int_{\mathds{R}^n} \mathbf{W}_G(f(v)\mathrm{d}\sigma)\,\mathrm{d}\sigma \leq \int_{\mathds{R}^n} v f(v)\,\mathrm{d}\sigma \nonumber\\ 
    & \leq \tilde{c}_1\int_{\mathds{R}^n} F(v)\,\mathrm{d}\sigma<\infty,\label{estimate8 orlicz}
\end{align}
where $\tilde{c}_1=\tilde{c}_1(p,q,\gamma)>0$. Let $B_j=B(0,j)$ and let $\sigma_j=\chi_{B_j}\sigma$, for $j>1$. Setting $\omega_j=f(v)\mathrm{d}\sigma_j$, one has $\mathrm{supp}\,\omega_j\subset B_{j}$. By \eqref{estimate8 orlicz},
\begin{equation*}
    \int_{B_{j+1}}\mathbf{W}_G\omega_j\,\mathrm{d}\omega_j \leq \int_{\mathds{R}^n}\mathbf{W}_G\omega\,\mathrm{d}\omega < \infty \quad \forall j>1.
\end{equation*}
Using Theorem~\ref{Inequality Wolff-Orlicz}, we infer from the previous inequality that $\omega_j\in \big(W_0^{1,G}(B_{j+1})\big)^{\ast}$ for all $j>1$.

Next, let $\varphi\in C_c^{\infty}(\mathds{R}^n)$ with $\mathrm{supp}\,\varphi \subset B_{j+1}$, for some $j>1$. Then there exists a constant $c_j>0$ such that
\begin{equation}\label{constants cj orlicz}
    \bigg|\int_{\mathds{R}^n}\varphi\,\mathrm{d}\omega_j \bigg|\leq c_j\,\|\nabla \varphi\|_{L^G} .
\end{equation}
Applying Theorem~\ref{existence in W0}, there exists $u_j\in W_0^{1,G}(B_{j+1})$ satisfying \eqref{eq. general} with $\mu=\omega_j$, that is,

\begin{equation}\label{equation1 orlicz}
    -\mathrm{div}\bigg(\frac{g(|\nabla u_j|)}{|\nabla u_j|}\nabla u_j\bigg)=\omega_j  \quad \mbox{in}\quad B_{j+1}.
\end{equation}
\begin{claim}\label{claim3 orlicz}
    The constant $c_j$ given in \eqref{constants cj orlicz} are uniformly bounded for all $j>1$.
\end{claim}
\noindent Indeed, first note that we may assume 
\begin{equation}\label{estimate11 orlicz}
    c_j\leq \tilde{c}_2\,\bigg(\int_{B_{j+1}}G(|\nabla u_j|)\,\mathrm{d} x\bigg) +1 \quad \forall j>1,
\end{equation}
where $\tilde{c}_2=\tilde{c}_2(p,q)>0$. This is seen by  testing  \eqref{equation1 orlicz} with $\varphi$ and by using a combination of Cauchy–Schwarz inequality and H\"{o}lder's inequality (Lemma~\ref{Holder-Orlicz}) with \eqref{relation norm and modular} and \eqref{relation G and tilde G}:
\begin{equation*}
    \begin{aligned}
\bigg|\int_{B_{j+1}}\varphi\,\mathrm{d}\omega_j\bigg|&=\bigg|\int_{B_{j+1}}\frac{g(|\nabla u_j|)}{|\nabla u_j|}\nabla u_j\cdot \nabla\varphi\,\mathrm{d} x\bigg| \\
&\leq \int_{B_{j+1}}\bigg|\frac{g(|\nabla u_j|)}{|\nabla u_j|}\nabla u_j\bigg||\nabla\varphi|\,\mathrm{d} x  \leq \int_{B_{j+1}}g(|\nabla u_j|)|\nabla \varphi|\,\mathrm{d} x \\
&\leq 2\|g(\nabla u_j)\|_{L^{G^{\ast}}}\|\nabla \varphi\|_{L^G} \leq \bigg[2\bigg(\int_{B_{j+1}}G^{\ast}\big(g(|\nabla u_j|)\big)\,\mathrm{d} x\bigg) +1 \bigg]\|\nabla \varphi\|_{L^G} \\
&\leq \bigg[\tilde{c}_2\,\bigg(\int_{B_{j+1}}G(|\nabla u_j|)\,\mathrm{d} x\bigg) +1 \bigg]\|\nabla \varphi\|_{L^G} .
    \end{aligned}
\end{equation*}
Notice that $u_j$ is harmonic in $B_{j+1}\setminus \overline{B}_{j}$ since $\mathrm{supp}\,\omega_j \subset B_j$, whence $u_j$ takes continuously zero
boundary values on $\partial B_{j+1}$. Extending $u_j$ by zero away from $\partial B_{j+1}$ and with the aid of Remark~\ref{remark supersolution and superharmonic}, we infer from Lemma~\ref{comparison principle super/subsolution} that $u_j$ is a nonnegative (almost everywhere) superharmonic in whole $\mathds{R}^n$. By Corallary~\ref{maly's result pratical}, for almost everywhere in $\mathds{R}^n$ it holds
\begin{equation}\label{estimate10 orlicz}
   0\leq u_j \leq K\,\mathbf{W}_G\omega_j\leq K\,\mathbf{W}_G\omega \quad \forall j>1.
\end{equation}
Testing \eqref{equation1 orlicz} with $u_j$, a combination of \eqref{data condition elementary} and  \eqref{estimate8 orlicz} with \eqref{estimate10 orlicz} yields
\begin{equation*}
    \begin{aligned}
       \int_{\mathds{R}^n}G(|\nabla u_{j}|)\,\mathrm{d} x &\leq \frac{1}{p}\int_{\mathds{R}^n}{g(|\nabla u_j|)}{|\nabla u_j|}\,\mathrm{d} x=  \frac{1}{p}\int_{\mathds{R}^n}u_j\,\mathrm{d}\omega_j\\
       &\leq  \frac{1}{p}\int_{\mathds{R}^n}u_j\,\mathrm{d}\omega\leq  \frac{K}{p}\int_{\mathds{R}^n}\mathbf{W}_{G}\omega\,\mathrm{d}\omega<\infty \quad \forall j>1.
    \end{aligned}
\end{equation*}
Consequently, by \eqref{estimate11 orlicz}, $\{c_j\}$ is uniformly bounded, with
\begin{equation*}
\sup_{j>1}c_j\leq\tilde{c}_3\,\bigg(\int_{\mathds{R}^n}\mathbf{W}_{G}\omega\,\mathrm{d}\omega\bigg) +1, 
\end{equation*}
where $\tilde{c}_3=\tilde{c}_3(p,q, K)>0$. This shows Claim~\ref{claim3 orlicz}.

On the other hand, by the Monotone Theorem Convergence, we have
\begin{equation*}
    \int_{\mathds{R}^n}\varphi\,\mathrm{d}\omega=\lim_{j\to \infty}\int_{\mathds{R}^n}\varphi\,\mathrm{d}\omega_j.
\end{equation*}
Thus, letting $j\to\infty$ in \eqref{constants cj orlicz}, we obtain
\begin{equation*}
\bigg|\int_{\mathds{R}^n}\varphi\,\mathrm{d}\omega\bigg|=\lim_{j\to \infty}\bigg| \int_{\mathds{R}^n}\varphi\,\mathrm{d}\omega_j\bigg|\leq c\,\|\nabla \varphi\|_{L^G},
\end{equation*}
where $c=\sup_j c_{j}$. This shows \eqref{estimate7 orlicz} and proves Lemma~\ref{wolff inequality lemma}.
\end{proof}

The following lemma gives a version of the comparison principle in $\mathcal{D}^{1,G}(\mathds{R}^n)$.
\begin{lemma}\label{comparison principle used}
    Let $\mu,\,\nu \in M^+(\mathds{R}^n)\cap \big(\mathcal{D}^{1,G}(\mathds{R}^n)\big)^{\ast}$ such that $\mu\leq\nu$. Suppose that $u,\,v \in \mathcal{D}^{1,G}(\mathds{R}^n)$ are solutions (respectively) to
    \begin{equation*}
        -\mathrm{div}\bigg(\frac{g(|\nabla u|)}{|\nabla u|}\nabla u\bigg)=\mu, \quad -\mathrm{div}\bigg(\frac{g(|\nabla v|)}{|\nabla v|}\nabla v\bigg)=\nu \quad \mbox{in}\quad \mathds{R}^n.
    \end{equation*}
    Then $u\leq v$ almost everywhere in $\mathds{R}^n$.
\end{lemma}
\begin{proof}
The proof is standard and relies on the use of the test function $\varphi=(u-v)^{+}=\max\{u-v,0\}\in \mathcal{D}^{1,G}(\mathds{R}^n)$. Indeed, testing both equations with such $\varphi$, one has
\begin{equation*}
    \begin{aligned}
    \int_{\mathds{R}^n} \frac{g(|\nabla u|)}{|\nabla u|}\nabla u\cdot\nabla\varphi\,\mathrm{d} x &= \int_{\mathds{R}^n}\varphi\,\mathrm{d}\mu,\\
    \int_{\mathds{R}^n} \frac{g(|\nabla v|)}{|\nabla v|}\nabla v\cdot\nabla\varphi\,\mathrm{d} x&=\int_{\mathds{R}^n}\varphi\,\mathrm{d}\nu.
    \end{aligned}
\end{equation*}
Hence
\begin{align}
    0&\leq \int_{\mathds{R}^n}\varphi\,\mathrm{d}\nu - \int_{\mathds{R}^n}\varphi\,\mathrm{d}\mu=   \int_{\mathds{R}^n} \frac{g(|\nabla v|)}{|\nabla v|}\nabla v\cdot\nabla\varphi\,\mathrm{d} x - \int_{\mathds{R}^n} \frac{g(|\nabla u|)}{|\nabla u|}\nabla u\cdot\nabla\varphi\,\mathrm{d} x \nonumber\\
    & =\int_{\mathds{R}^n}\bigg(\frac{g(|\nabla v|)}{|\nabla v|}\nabla v- \frac{g(|\nabla u|)}{|\nabla u|}\nabla u\bigg)\cdot\nabla\varphi\,\mathrm{d} x.\label{estimate12 orlicz}
\end{align}
By \cite[Theorem~1]{MR4518648}, the following monotonicity for $g$ holds
\begin{equation}\label{monotonicity usual g}
    \bigg(\frac{g(|\xi|)}{|\xi|}\xi- \frac{g(|\eta|)}{|\eta|}\eta\bigg)\cdot (\xi-\eta)>0 \quad \forall \xi, \eta\in \mathds{R}^n,\, \xi\neq \eta.
\end{equation}
Since $\nabla\varphi=\nabla (u-v)=\nabla u- \nabla v$ in $\mathrm{supp}\,\varphi\subset \{u>v\}$, it follows from \eqref{estimate12 orlicz} and \eqref{monotonicity usual g} that
\begin{equation*}
    0\leq \int_{\mathds{R}^n}\bigg(\frac{g(|\nabla v|)}{|\nabla v|}\nabla v- \frac{g(|\nabla u|)}{|\nabla u|}\nabla u\bigg)\cdot(\nabla u-\nabla v)\,\mathrm{d} x\leq 0,
\end{equation*}
which implies that $\nabla\varphi =0$ almost everywhere in $\mathds{R}^n$. Thus $\varphi =0$ almost everywhere in $\mathds{R}^n$, and proves Lemma~\ref{comparison principle used}.
\end{proof}

The following lemma concerning weak compactness in $\mathcal{D}^{1,G}(\mathds{R}^n)$, and it will be useful in the proof of Theorem~\ref{solution to orlicz-measure eq}.
\begin{lemma}\label{weak compactness in D orlicz}
    Suppose that $\{u_j\}\subset \mathcal{D}^{1,G}(\mathds{R}^n)$ is a sequence converging pointwise to $u$ almost everywhere in $\mathds{R}^n$. If $\{\nabla u_j\}$ is bounded in $L^G(\mathds{R}^n)$, then $u\in \mathcal{D}^{1,G}(\mathds{R}^n)$ and $\nabla u_j \rightharpoonup \nabla u$ in $L^G(\mathds{R}^n)$. 
\end{lemma}
\begin{proof}
    The proof is similar in spirit to the proof of \cite[Lemma~1.33]{MR2305115}. Let $B\subset \mathds{R}^n$ a ball. First, with the aid of the Mazur lemma (see for instance \cite[Lemma~1.29]{MR2305115}), we infer that if for some sequence $\{v_j\}\subset W^{1,G}(B)$, there exist $v\in L^G(B)$ and $X$ with $|X| \in L^G(B)$ satisfying
    \begin{equation}\label{consequence mazur lemma}
        \begin{aligned}
            v_j &\rightharpoonup  v  &\mbox{in}& &L^G, &\\
            \nabla v_j &\rightharpoonup  X  &\mbox{in}&  &L^G,&
        \end{aligned}
    \end{equation}
then $v\in W^{1,G}(B)$ and $X=\nabla v$.    

Next, by Lemma~\ref{modular poincare ineq}, the sequence $\{u_j\}$ is  bounded in $W^{1,G}(B)$. Since $u_j$ converges pointwise (almost everywhere) to $u$ in $B$, from Theorem~\ref{Lemma weakly convergence}, $u_j\rightharpoonup u$ in $L^G(B)$. Moreover, a subsequence $\nabla u_{j_k} \rightharpoonup \nabla u$ in $L^G(B)$ by \eqref{consequence mazur lemma}. Since the weak limit is independent of the choice of the subsequence, it follows that $\nabla u_j \rightharpoonup \nabla u\in L^G(B)$.

On the other hand, being $\{\nabla u_j\}$ bounded in $L^G(\mathds{R}^n)$, it has a weakly converging subsequence in $L^G(\mathds{R}^n)$. This subsequence also converge weakly in $L^G(B)$, whence it converges weakly to $\nabla u$, which gives  $\nabla u\in L^G(\mathds{R}^n)$, that is $u\in \mathcal{D}^{1,G}(\mathds{R}^n)$. Therefore, $\nabla u_j\rightharpoonup \nabla u$ in $L^G(\mathds{R}^n)$ since the weak limit is independent of the subsequence. This proves Lemma~\ref{weak compactness in D orlicz}. 
\end{proof}

The following lemma uses the Monotone Operator Theory to prove
the existence of solutions to \eqref{eq. general} in $\mathcal{D}^{1,G}(\mathds{R}^n)$.
\begin{lemma}\label{existence in D1G orlicz}
    For all $\mu\in M^+(\mathds{R}^n)\cap\big(\mathcal{D}^{1,G}(\mathds{R}^n)\big)^{\ast}$, there exists a unique solution $u\in \mathcal{D}^{1,G}(\mathds{R}^n)$ to \eqref{eq. general}, which is a superharmonic and nonnegative function (almost everywhere).
\end{lemma}
\begin{proof}
    Let us consider the operator
    $T:\mathcal{D}^{1,G}(\mathds{R}^n)\to \big(\mathcal{D}^{1,G}(\mathds{R}^n)\big)^{\ast}$ defined by
    \begin{equation*}
        \langle Tu,\varphi\rangle=\int_{\mathds{R}^n}\frac{g(|\nabla u|)}{|\nabla u|}\nabla u\cdot \nabla\varphi\,\mathrm{d} x \quad \forall \varphi\in \mathcal{D}^{1,G}(\mathds{R}^n),
    \end{equation*}
    where $\langle \cdot,\cdot\rangle $ denote the usual pairing between $\big(\mathcal{D}^{1,G}(\mathds{R}^n)\big)^{\ast}$ and $\mathcal{D}^{1,G}(\mathds{R}^n)$.
    $T$ is well defined. Indeed,  by using a combination of Cauchy–Schwarz inequality and H\"{o}lder's inequality (Lemma~\ref{Holder-Orlicz}) with \eqref{relation norm and modular}, we obtain
\begin{equation*}
    \begin{aligned}
        |  \langle Tu,\varphi\rangle|& \leq 2 \|g(|\nabla u|)\|_{L^{G^{\ast}}}\|\nabla\varphi\|_{L^G}\\
        & \leq c\bigg(\int_{\mathds{R}^n}G(|\nabla u|)\,\mathrm{d} x+1\bigg)\|\varphi\|_{\mathcal{D}^{1,G}},
    \end{aligned}
\end{equation*}
which shows $Tu$ is bounded. A classical result \cite[Theorem~2.1]{MR0259693} assures that $T$ is surjective provided $T$ is coercive and weakly continuous on $\mathcal{D}^{1,G}(\mathds{R}^n)$, that is, if $u=\lim_j u_j$ in $\mathcal{D}^{1,G}(\mathds{R}^n)$, then $Tu_j \rightharpoonup Tu$ in  $\big(\mathcal{D}^{1,G}(\mathds{R}^n)\big)^{\ast}$.

 We first verify the weak continuity of $T$. Fix $\{u_j\}\subset \mathcal{D}^{1,G}(\mathds{R}^n)$ a sequence that converges (in norm) to an element $u\in \mathcal{D}^{1,G}(\mathds{R}^n)$. Using \cite[Lemma~2.1]{MR4274293}, there exists a subsequence of $\{u_j\}$, also denoted $\{u_j\}$ (by abuse of notation), which converges almost everywhere to $u$ in $\mathds{R}^n$. From this, by continuity of $g$, we have
 \begin{equation*}
  \frac{g(|\nabla u|)}{|\nabla u|}\nabla u=\lim_{j\to\infty} \frac{g(|\nabla u_j|)}{|\nabla u_j|}\nabla u_j \quad \mbox{in}\quad \mathds{R}^n \quad \mbox{(almost everywhere).}
 \end{equation*}
Since $\{u_j\}$ converges to $u$ in $\mathcal{D}^{1,G}(\mathds{R}^n)$, it follows that $\{\int_{\mathds{R}^n}G(|\nabla u_j|)\,\mathrm{d} x\}$ is bounded, and  
\begin{equation*}
    \left\{\int_{\mathds{R}^n}G^{\ast}\bigg(\frac{g(|\nabla u_j|)}{|\nabla u_j|}\nabla u_j\bigg)\,\mathrm{d} x\right\} \quad \mbox{is bounded}.
\end{equation*}
    By Lemma~\ref{lemma relation norm and modular}, $\{{g(|\nabla u_j|)}/{|\nabla u_j|}\,\nabla u_j\}$ is bounded in ${L^{G^{\ast}}}$. Applying Theorem~\ref{Lemma weakly convergence},
    \begin{equation*}
        \frac{g(|\nabla u_j|)}{|\nabla u_j|}\nabla u_j \rightharpoonup  \frac{g(|\nabla u|)}{|\nabla u|}\nabla u\quad \mbox{in}\quad L^{G^{\ast}}(\mathds{R}^n),
    \end{equation*}
whence our assertion follows since the weak limit is independent of the choice of a subsequence.

Next, we prove that $T$ is coercive. Let $u\in \mathcal{D}^{1,G}(\mathds{R}^n)$ with $\|u\|_{ \mathcal{D}^{1,G}(\mathds{R}^n)}\geq 1$. Combining \eqref{data condition elementary} with Lemma~\ref{lemma relation norm and modular}, we obtain
\begin{equation*}
    \begin{aligned}
        \langle Tu,u\rangle & = \int_{\mathds{R}^n}{g(|\nabla u|)}{|\nabla u|}\,\mathrm{d} x \geq p \int_{\mathds{R}^n}{G(|\nabla u|)}\,\mathrm{d} x \\
        &\geq p \|\nabla u\|_{L^G}^p = p \| u\|_{{\mathcal{D}}^{1,G}}^p. 
    \end{aligned}
\end{equation*}
Since $p>1$,
\begin{equation*}
    \lim_{\| u\|_{{\mathcal{D}}^{1,G}}\to \infty} \frac{ \langle Tu,u\rangle}{\| u\|_{{\mathcal{D}}^{1,G}}}=\infty,
\end{equation*}
which shows that $T$ is coercive. 

Hence, for $\mu\in M^+(\mathds{R}^n)\cap\big(\mathcal{D}^{1,G}(\mathds{R}^n)\big)^{\ast}$, there exists a $u\in \mathcal{D}^{1,G}(\mathds{R}^n)$, such that $T u= \mu$, that is
\begin{equation*}
    \int_{\mathds{R}^n}\frac{g(|\nabla u|)}{|\nabla u|}\nabla u\cdot\nabla \varphi\,\mathrm{d} x = \int_{\mathds{R}^n}\varphi\,\mathrm{d}\mu \quad \forall \varphi\in \mathcal{D}^{1,G}(\mathds{R}^n).
\end{equation*}
Thus $u$ is a solution to \eqref{eq. general}, and by monotonicity \eqref{monotonicity usual g}, it should be unique. Moreover, since $\mu\in M^+(\mathds{R}^n)$, $u$ is a supersolution in $\mathds{R}^n$ in sense of Definition~\ref{def supersolution}, whence $u\geq 0$ almost everywhere by Lemma~\ref{comparison principle super/subsolution}. Because of Remark~\ref{remark supersolution and superharmonic}, we may suppose that $u$ is superharmonic in $\mathds{R}^n$. This completes the proof of Lemma~\ref{existence in D1G orlicz}.
\end{proof}

\subsection{Proof of Theorem~\ref{solution to orlicz-measure eq}}
We need only consider the case $1<p<q<n$, otherwise by Remark~\ref{remark q>n orlicz} if $q\geq n$, Eq.~\eqref{equation nonstandard} has only the trivial solution $u=c\geq 0$. Let $\sigma\in M^{+}(\mathds{R}^n)$ satisfying \eqref{cond sufficient to exist} and let $K\geq 1$ be the constant given in Corollary~\ref{maly's result pratical}. Using Theorem~\ref{existencewolffequation} with $\big((q/p)^{1/(p-1)}K\big)^{q-1}\sigma$ in place of $\sigma$, we deduce that there exists a solution $0\leq v\in L^F(\mathds{R}^n,\mathrm{d}\sigma)$ in $\mathds{R}^n$ to
\begin{equation}\label{eq. wolff integral orlicz}
    v=\mathbf{W}_G\bigg(f(v)\Big({\frac{q}{p}}\Big)^{\frac{q-1}{p-1}}K^{q-1}\mathrm{d}\sigma\bigg)
\end{equation}
By \eqref{estimate on g-1}, we have
\begin{align}
    v&= \mathbf{W}_G\bigg(f(v)\Big({\frac{q}{p}}\Big)^{\frac{q-1}{p-1}}K^{q-1}\mathrm{d}\sigma\bigg)\nonumber\\
    &\geq \Big(\frac{p}{q}\Big)^{\frac{1}{p-1}}\bigg[\Big({\frac{q}{p}}\Big)^{\frac{q-1}{p-1}}K^{q-1}\bigg]^{\frac{1}{q-1}}\mathbf{W}_G(f(v)\mathrm{d}\sigma)\nonumber\\
    &= K\,\mathbf{W}_G(f(v)\mathrm{d}\sigma) \quad \mbox{in}\quad \mathds{R}^n. \label{estimate13 orlcz}
\end{align}
Consequently, $v$ is a supersolution to \eqref{wolff integral equation}. By Lemma~\ref{wolff inequality lemma}, $f(v)\mathrm{d}\sigma\in \big(\mathcal{D}^{1,G}(\mathds{R}^n)\big)^{\ast}$. Moreover, from Theorem~\ref{estimativainferior orlicz}, one has
\begin{equation*}
    \begin{aligned}
        v&\geq C\,\bigg[\mathbf{W}_G\bigg(\Big({\frac{q}{p}}\Big)^{\frac{q-1}{p-1}}K^{q-1}\sigma\bigg)\bigg]^{\frac{1}{1-\gamma}}\\
        &\geq C\,\Big(\frac{p}{q}\Big)^{\frac{1}{p-1}}\bigg[\Big({\frac{q}{p}}\Big)^{\frac{q-1}{p-1}}K^{q-1}\bigg]^{\frac{1}{(q-1)}\frac{1}{1-\gamma}}\big(\mathbf{W}_G\sigma\big)^{\frac{1}{1-\gamma}}\\
        &= C\,\Big({\frac{q}{p}}\Big)^{\frac{\gamma}{(p-1)(1-\gamma)}}K^{\frac{1}{1-\gamma}}\big(\mathbf{W}_G\sigma\big)^{\frac{1}{1-\gamma}},
    \end{aligned}
\end{equation*}
where $C$ is the constant given in Theorem~\ref{estimativainferior orlicz}. 

Let $\varepsilon$ be a positive constant satisfying
\begin{equation}\label{constant epsilon orlicz}
    \varepsilon\leq \min\left\{\Big({\frac{q}{p}}\Big)^{\frac{\gamma}{(p-1)(1-\gamma)}}K^{\frac{1}{1-\gamma}}\,C,\; \big(K^{-1}\lambda\big)^{\frac{p-1}{p-1-\gamma(q-1)}},\, K^{-\frac{p-1}{(q-1)(1-\gamma)}}\,C\right\},
\end{equation}
where $\lambda$ is the constant given in Lemma~\ref{estimativelambda}.
Setting, $u_0=\varepsilon \big(\mathbf{W}_G\sigma\big)^{\frac{1}{1-\gamma}}$, we obtain $u_0\leq v$ in $\mathds{R}^n$, whence
\begin{equation*}
    u_0\in L^F(\mathds{R}^n,\mathrm{d}\sigma) \quad \mbox{and}\quad f(u_0)\mathrm{d}\sigma \in M^+(\mathds{R}^n)\cap\big(\mathcal{D}^{1,G}(\mathds{R}^n)\big)^{\ast}.
\end{equation*}
From Lemma~\ref{existence in D1G orlicz}, there exists a unique nonnegative superharmonic function $u_1\in \mathcal{D}^{1,G}(\mathds{R}^n)$ satisfying
\begin{equation*}
      -\mathrm{div}\bigg(\frac{g(|\nabla u_1|)}{|\nabla u_1|}\nabla u_1\bigg)=\sigma\,f(u_0) \quad \mbox{in}\quad \mathds{R}^n.
\end{equation*}
A combination of Corollary~\ref{maly's result pratical} with \eqref{estimate13 orlcz}, yields
\begin{equation*}
    u_1\leq K\,\mathbf{W}_G(f(u_0)\mathrm{d}\sigma) \leq K\,\mathbf{W}_G(f(v)\mathrm{d}\sigma)\leq v.
\end{equation*}
From this, $u_1\in L^F(\mathds{R}^n,\mathrm{d}\sigma)$ and  $f(u_1)\mathrm{d}\sigma\in M^+(\mathds{R}^n)\cap\big(\mathcal{D}^{1,G}(\mathds{R}^n)\big)^{\ast}$. In addition, using Lemma~\ref{estimativelambda} and \eqref{estimate on g-1}, we obtain that
\begin{equation*}
    \begin{aligned}
        u_1&\geq K^{-1}\,\mathbf{W}_G(f(u_0)\mathrm{d}\sigma)=K^{-1}\,\mathbf{W}_G\big(f\big(\varepsilon \big(\mathbf{W}_G\sigma\big)^{\frac{1}{1-\gamma}}\big)\mathrm{d}\sigma\big)\\
        &\geq K^{-1}\varepsilon^{\frac{q-1}{p-1}\gamma}\,\mathbf{W}_G\big(f \big(\mathbf{W}_G\sigma\big)^{\frac{1}{1-\gamma}}\mathrm{d}\sigma\big)\\
        & \geq K^{-1}\varepsilon^{\frac{q-1}{p-1}\gamma}\lambda\,\big(\mathbf{W}_G\sigma\big)^{\frac{1}{1-\gamma}}.
    \end{aligned}
\end{equation*}
By choice of $\varepsilon$, we deduce that $u_1\geq u_0$. Summarizing, we have
\begin{equation*}
    \left\{
    \begin{aligned}
&\quad u_1\in \mathcal{D}^{1,G}(\mathds{R}^n)\cap L^F(\mathds{R}^n,\mathrm{d}\sigma),& &&&&\\
&\quad  f(u_1)\mathrm{d}\sigma\in M^+(\mathds{R}^n)\cap\big(\mathcal{D}^{1,G}(\mathds{R}^n)\big)^{\ast}, &&&&\\
& -\mathrm{div}\bigg(\frac{g(|\nabla u_1|)}{|\nabla u_1|}\nabla u_1\bigg)=\sigma\,f(u_0) &\mbox{in}& &\mathds{R}^n,&\\
       &\qquad   u_0\leq u_1\leq v  \quad \mbox{almost everywhere} &\mbox{in}& &\mathds{R}^n.&
    \end{aligned}
    \right.
\end{equation*}

Now, we construct by induction a sequence of nonnegative functions $u_j$, for $j=1,2,\ldots$, satisfying
  \begin{equation}\label{seq. uj solution to P orlicz}
    \left\{
    \begin{aligned}
&\quad u_j\in \mathcal{D}^{1,G}(\mathds{R}^n)\cap L^F(\mathds{R}^n,\mathrm{d}\sigma), &&&\forall j>1,&\\
&\quad  f(u_{j})\mathrm{d}\sigma\in M^+(\mathds{R}^n)\cap\big(\mathcal{D}^{1,G}(\mathds{R}^n)\big)^{\ast}, &&&\forall j>1,&\\
& -\mathrm{div}\bigg(\frac{g(|\nabla u_j|)}{|\nabla u_j|}\nabla u_j\bigg)=\sigma\,f(u_{j-1}) &\mbox{in}& &\mathds{R}^n,\quad \forall j>1,&\\
       &\qquad   u_{j-1}\leq u_j\leq v \quad \mbox{almost everywhere} &\mbox{in}& &\mathds{R}^n, \quad\forall j> 1,&\\
&\quad \sup_{j>1}\int_{\mathds{R}^n}G(|\nabla u_j|)\,\mathrm{d} x \leq c \int_{\mathds{R}^n} F(v)\,\mathrm{d} \sigma, &&&&       
    \end{aligned}
    \right.
\end{equation}
where $c=c(p,q,\gamma)>0$.

\noindent Indeed, suppose that $u_j$ has been obtained for some $j>1$. In the same manner as the case $j=1$, since $u_j\leq v$ in $\mathds{R}^n$, it follows from Lemma~\ref{wolff inequality lemma} that
\begin{equation*}
    f(u_j)\mathrm{d}\sigma\in M^+(\mathds{R}^n)\cap\big(\mathcal{D}^{1,G}(\mathds{R}^n)\big)^{\ast}.
\end{equation*}
Using Lemma~\ref{existence in D1G orlicz}, there exists a unique nonnegative superharmonic function $u_{j+1}\in \mathcal{D}^{1,G}(\mathds{R}^n)$ satisfying
\begin{equation}\label{eq uj+1 orlicz}
    -\mathrm{div}\bigg(\frac{g(|\nabla u_{j+1}|)}{|\nabla u_{j+1}|}\nabla u_{j+1} \bigg)=\sigma\,f(u_{j}) \quad \mbox{in}\quad \mathds{R}^n.
\end{equation}
Since $f(u_{j-1})\mathrm{d}\sigma\leq f(u_j)\mathrm{d}\sigma$, by Lemma~\ref{comparison principle used}, $u_j\leq u_{j+1}$ in $\mathds{R}^n$. 
Applying Corollary~\ref{maly's result pratical},
\begin{equation*}
    u_{j+1}\leq K\,\mathbf{W}_G(f(u_j)\mathrm{d}\sigma) \leq K\,\mathbf{W}_G(f(v)\mathrm{d}\sigma)\leq v.
\end{equation*}
Testing \eqref{eq uj+1 orlicz} with $u_{j+1}$, using \eqref{data condition elementary} and \eqref{def of f and F}, we have
\begin{equation*}
\begin{aligned}
    \int_{\mathds{R}^n}G(|\nabla u_{j+1}|)\,\mathrm{d} x& \leq \frac{1}{p}\int_{\mathds{R}^n} g(|\nabla u_{j+1}|)|\nabla u_{j+1}|\,\mathrm{d} x\\
    &= \frac{1}{p}\int_{\mathds{R}^n}\frac{g(|\nabla u_{j+1}|)}{|\nabla u_{j+1}|}\nabla u_{j+1} \cdot \nabla u_{j+1} \,\mathrm{d} x\\
    &= \frac{1}{p}\int_{\mathds{R}^n} u_{j+1}f(u_j)\,\mathrm{d}\sigma\leq \frac{1}{p}\int_{\mathds{R}^n} vf(v)\,\mathrm{d}\sigma\\
    &\leq \frac{(q-1)\gamma +1}{p}\int_{\mathds{R}^n} F(v)\,\mathrm{d}\sigma.
\end{aligned}
\end{equation*}
Thus, the sequence  \eqref{seq. uj solution to P orlicz} has been constructed.

We set $u=\lim_{j}u_j$ in $\mathds{R}^n$  (pointwise). Hence $u\leq v$ almost everywhere and, by Lemma~\ref{weak compactness in D orlicz}, $u\in \mathcal{D}^{1,G}(\mathds{R}^n)$ with $\nabla u_j\rightharpoonup \nabla u$ in $L^G(\mathds{R}^n)$. As in the proof of Lemma~\ref{existence in D1G orlicz}, the last line in \eqref{seq. uj solution to P orlicz} gives 
\begin{equation*}
    \frac{g(|\nabla u_{j}|)}{|\nabla u_{j}|}\nabla u_{j}\ \quad \mbox{is uniformly bounded in}\quad L^{G^{\ast}}(\mathds{R}^n).
\end{equation*}
From Harnack's Principle (Theorem~\ref{harnack principle orlicz5}), up to a subsequence, $\nabla u=\lim_j\nabla u_j$ pointwise in $\mathds{R}^n$, since $u<\infty$. Consequently, the continuity of the function $g$ and Theorem~\ref{Lemma weakly convergence} assure
\begin{equation}\label{estimate14 orlicz}
    \frac{g(|\nabla u_{j}|)}{|\nabla u_{j}|}\nabla u_{j}\rightharpoonup  \frac{g(|\nabla u|)}{|\nabla u|}\nabla u \quad \mbox{in}\quad L^{G^{\ast}}(\mathds{R}^n).
\end{equation}
On the other hand, by the Monotone Convergence Theorem, one has 
\begin{equation}\label{estimate15 orlicz}
   \lim_{j\to\infty}\int_{\mathds{R}^n}\varphi f(u_j)\,\mathrm{d}\sigma=\int_{\mathds{R}^n}\varphi f(u)\,\mathrm{d}\sigma \quad \forall \varphi\in C_c^{\infty}(\mathds{R}^n).
\end{equation}

Therefore, letting $j\to\infty$ in the first line of \eqref{seq. uj solution to P orlicz}, from \eqref{estimate14 orlicz} and \eqref{estimate15 orlicz}, $u\in \mathcal{D}^{1,G}(\mathds{R}^n)\cap L^F(\mathds{R}^n,\mathrm{d}\sigma)$ is a solution to \eqref{equation nonstandard}.

It remains to prove $u$ is minimal. Suppose that $0\leq w\in \mathcal{D}^{1,G}(\mathds{R}^n)\cap L_{\mathrm{loc}}^f(\mathds{R}^n,\mathrm{d}\sigma)$ is any nontrivial solution to \eqref{equation nonstandard}.
First, we will show $w\in L^{F}(\mathds{R}^n,\mathrm{d}\sigma)$, that is, $\int_{\mathds{R}^n}F(w)\,\mathrm{d}\sigma<\infty$. Note that $f(w)\mathrm{d}\sigma\in M^{+}(\mathds{R}^n)\cap\big(\mathcal{D}^{1,G}(\mathds{R}^n)\big)^{\ast}$. By density of $C_c^{\infty}(\mathds{R}^n)$ in $\mathcal{D}^{1,G}(\mathds{R}^n)$, there exists a sequence $\{\varphi_j\}\subset C_c^{\infty}(\mathds{R}^n)$ with $w=\lim_j\varphi_j$ in $\mathcal{D}^{1,G}(\mathds{R}^n)$. It follows from \eqref{def of f and F} that
\begin{equation*}
    \begin{aligned}
         \int_{\mathds{R}^n}F(w)\,\mathrm{d}\sigma & \leq \frac{1}{(p-1)\gamma+1}\int_{\mathds{R}^n}w f(w)\,\mathrm{d}\sigma = \frac{\langle f(w)\mathrm{d}\sigma,w\rangle}{(p-1)\gamma+1}\\
         &= \lim_{j\to\infty}\frac{\langle f(w)\mathrm{d}\sigma,\varphi_j\rangle}{(p-1)\gamma+1}   = \frac{1}{(p-1)\gamma + 1}\lim_{j\to\infty}\int_{\mathds{R}^n} \frac{g(|\nabla w|)}{|\nabla w|}\nabla w \cdot  \nabla\varphi_j\,\mathrm{d} x,
    \end{aligned}
\end{equation*}
where the last equality follows by testing the equation of $w$ with each $\varphi_j$. Since $\nabla \varphi\rightharpoonup \nabla w$ in $L^G(\mathds{R}^n)$ and $ {g(|\nabla w|)}/{|\nabla w|}\,\nabla w\in L^{G^{\ast}}(\mathds{R}^n)$, we deduce
\begin{equation*}
    \begin{aligned}
        \int_{\mathds{R}^n}F(w)\,\mathrm{d}\sigma & \leq \frac{1}{(p-1)\gamma + 1}\int_{\mathds{R}^n} \frac{g(|\nabla w|)}{|\nabla w|}\nabla w \cdot  \nabla w\,\mathrm{d} x \\
        & =\frac{1}{(p-1)\gamma + 1}\int_{\mathds{R}^n}{g(|\nabla w|)}{|\nabla w|}\,\mathrm{d} x\\
        & \leq \frac{q}{(p-1)\gamma + 1} \int_{\mathds{R}^n}G(|\nabla w|)\,\mathrm{d} x< \infty.
    \end{aligned}
\end{equation*}

Next, by Corollary~\ref{maly's result pratical}, we have
\begin{equation*}
    w\geq K^{-1}\,\mathbf{W}_G(f(w)\mathrm{d}\sigma)\geq \mathbf{W}_G(f(w)K^{-(p-1)}\mathrm{d}\sigma),
\end{equation*}
where in the last inequality was used \eqref{estimate on g-1}, since $K\geq 1$. Using Theorem~\ref{estimativainferior orlicz} with $K^{-(p-1)}\sigma$ in place of $\sigma$, we obtain that
\begin{equation*}
    w\geq C\big(\mathbf{W}_G(K^{-(p-1)}\sigma)\big)^{\frac{1}{1-\gamma}}\geq C\,K^{-\frac{p-1}{(q-1)(1-\gamma)}}\big(\mathbf{W}_G\sigma\big)^{\frac{1}{1-\gamma}}.
\end{equation*}
By choice of $\varepsilon$ in \eqref{constant epsilon orlicz}, one has
$u_0\leq w$ in $\mathds{R}^n$, whence $f(u_0)\mathrm{d}\sigma\leq f(w)\mathrm{d}\sigma$ in  $M^{+}(\mathds{R}^n)\cap\big(\mathcal{D}^{1,G}(\mathds{R}^n)\big)^{\ast}$. From Lemma~\ref{comparison principle used}, $u_1\leq w$ almost everywhere in $\mathds{R}^n$. Proceeding by induction, we deduce that $u_j \leq w $ holds almost everywhere in $\mathds{R}^n$ for all $j\geq 1$. Thus
\begin{equation*}
    u=\lim_{j\to\infty}u_j\leq w \quad\mbox{almost everywhere in}\quad \mathds{R}^n,
\end{equation*}
which shows that $u$ is minimal and completes the proof of Theorem~\ref{solution to orlicz-measure eq}.

\subsection{Final comments}

\begin{remark}\label{remark extension of Igor}
    For the case $p=q$ in \eqref{our g}, $\gamma$ only satisfies $0<\gamma<1$. Setting $r=\gamma (p-1)$, obliviously 
    \begin{equation}\label{condition p=q}
        \frac{1}{1-\gamma}=\frac{p-1}{p-1-\gamma(q-1)}=\frac{q-1}{q-1-\gamma(p-1)}=\frac{p-1}{p-1-r}.
    \end{equation}
    Hence, condition \eqref{cond sufficient to exist} reduces to 
    \begin{equation}\label{cond suf/nec Igor}
        \big(\mathbf{W}_p\sigma\big)^{\frac{p-1}{p-1-r}}\in L^{1+r}(\mathds{R}^n,\mathrm{d}\sigma).
    \end{equation}
    Moreover, for $p=q$, Eq.~\eqref{equation nonstandard} becomes in \eqref{equation standard}, that is
    \begin{equation}\label{equation standard 2}
        -\Delta_p u =\sigma\, u^r \quad\mbox{in}\quad \mathds{R}^n.
    \end{equation}
Thus, Theorem~\ref{solution to orlicz-measure eq} is a partial extension of \cite[Theorem~3.8]{MR3311903}, for the case $p=q$. 
\end{remark}

\begin{remark}\label{remark important A-operator orlicz}
The same conclusion of Theorem~\ref{solution to orlicz-measure eq} can be drawn for the $\mathcal{A}$-equation 
\begin{equation}\label{A-div operator orlicz}
    -\mathrm{div}\big(\mathcal{A}(x,\nabla u)\big)=\sigma\,g(u^{\gamma})\quad \mbox{in}\quad\mathds{R}^n,
\end{equation}
where $\mathcal{A}$ is a Carath\'{e}odory regular vector field satisfying the Orlicz growth $\mathcal{A}(x,\xi)\cdot\xi  \approx G(|\xi|)$, where $G$ is the primitive of $g$, given by \eqref{our g}. 

To proceed formally, suppose \eqref{assumption on g}
and let $\mathcal{A}:\mathds{R}^n\times\mathds{R}^n\to \mathds{R}^n$ be a mapping satisfying the following conditions:
\begin{equation}\label{condition A1 orlicz}
    \begin{aligned}
        &x\longmapsto \mathcal{A}(x,\xi) \quad\mbox{is measurable for all }\xi\in\mathds{R}^n, \\
        & \xi\longmapsto \mathcal{A}(x,\xi) \quad\mbox{is continuous for almost everywhere }x\in\mathds{R}^n,
    \end{aligned}
\end{equation}
and there exist structural constants $\alpha>0$ and $\beta>0$ such that for almost everywhere $x\in\mathds{R}^n$, for all $\xi,\,\eta \in \mathds{R}^n$, $\xi\neq \eta$, it holds
\begin{equation}\label{condition A2 orlicz}
    \begin{aligned}
         &\mathcal{A}(x,\xi)\cdot\xi \geq \alpha \,G(|\xi|), \quad |\mathcal{A}(x,\xi)|\leq \beta\,g(|\xi|),\\
         &\big(\mathcal{A}(x,\xi)-\mathcal{A}(x,\eta)\big)\cdot(\xi-\eta)>0.
    \end{aligned}
\end{equation}
 In particular, $\mathcal{A}(x,0)=0$ for almost everywhere $x\in \mathds{R}^n$. A typical example is what was treated in this work:
\begin{equation*}
    \mathcal{A}_0: (x,\xi)\longmapsto\mathcal{A}_0(x,\xi)=\frac{g(|\xi|)}{|\xi|}\xi.
\end{equation*}
Let $\Omega\subseteq \mathds{R}^n$ be a domain. Similarly to Definition~\ref{definition supersolution subsolution orlicz}, we say that a continuous function $u\in W_{\mathrm{loc}}^{1,G}(\Omega) $ is \textit{$\mathcal{A}$-harmonic}  in $\Omega$ if it satisfies $-\mathrm{div}\big(\mathcal{A}(x,\nabla u)\big)=0$ weakly in $\Omega$, that is
\begin{equation*}
    \int_\Omega \mathcal{A}(x,\nabla u)\cdot\nabla \varphi\,\mathrm{d} x=0 \quad \forall \varphi\in C_c^{\infty}(\Omega).
\end{equation*}
 A function $u\in W_{\mathrm{loc}}^{1,G}(\Omega)$ is called \textit{$\mathcal{A}$-supersolution} in $\Omega$ if it satisfies  $-\mathrm{div}\big(\mathcal{A}(x,\nabla u)\big)\geq 0$ weakly in $\Omega$, and by \textit{$\mathcal{A}$-subsolution} in $\Omega$ we mean a function $u$ such that $-u$ is \textit{$\mathcal{A}$-supersolution} in $\Omega$. The classes of \textit{$\mathcal{A}$-superharmonic} and \textit{$\mathcal{A}$-subharmonic} functions are defined likewise Definition~\ref{def superharmonic}. We denote $\mathcal{S}_{\mathcal{A}}(\Omega)$ the set of all $\mathcal{A}$-superharmonic functions in $\Omega$. 

According to the above definitions,  we mention that all basic facts stated in Section~\ref{preliminaries orlicz}, and the preliminaries lemmas of Section~\ref{applications orlicz}, remain 
 true for quasilinear elliptic $\mathcal{A}$-equations with measure data, that is equations of the type
 \begin{equation}\label{eq. general app}
     -\mathrm{div}\big(\mathcal{A}(x,\nabla u)\big)=\mu\quad \mbox{in} \quad \Omega.
 \end{equation}
 See for  instance \cite{MR4274293, MR4258794, MR2834769, MR4482108, MR4568881}. In particular, for $g$ given by \eqref{our g}, Corollary~\ref{maly's result pratical} is still true for $\mathcal{A}$-equations with measure data. In Section~\ref{Potential estimates App} we deal with this approach. Summarizing, we have the following theorem
 \begin{theorem}
      Let $\mathcal{A}$ be a mapping satisfying \eqref{condition A1 orlicz} and \eqref{condition A2 orlicz} with $g$ given by \eqref{our g}. Under the hypotheses of Theorem~\ref{solution to orlicz-measure eq}, there exists a nonnegative solution $u\in \mathcal{D}^{1,G}(\mathds{R}^n)\cap L^F(\mathds{R}^n,\mathrm{d}\sigma)$ to \eqref{A-div operator orlicz}, provided \eqref{cond sufficient to exist} holds.  Moreover, $u$ is minimal.
 \end{theorem}
\end{remark}



\section{Potential estimates}\label{Potential estimates App}
This technical section presents a proof of Theorem~\ref{Maly's type result}. The argument follows from a series of versions of Harnack inequalities. The ideas that inspire our proof in \cite{MR4803728} and \cite{MR2604619}, where the case $G_p(t)=t^p/p$ is treated.

In fact, we will prove an extending version of Theorem~\ref{Maly's type result}. Namely, \textit{ let $\mathcal{A}$ be a mapping satisfying \eqref{condition A1 orlicz} and \eqref{condition A2 orlicz} with $g$ given by \eqref{our g}, suppose $u$ is a nonnegative $\mathcal{A}$-superharmonic in $B(x_0,2R)$, then $u$ satisfies \eqref{potential estimate thm} where 
 \begin{equation}\label{eq. general u appendix super}
     \mu=\mu_u=-\mathrm{div}\big(\mathcal{A}(x,\nabla u)\big).
 \end{equation}
}To simplify notation, we  set $B_R=B(x_0,R)$ for $R>0$, and $\kappa B_R= B(x_0,\kappa R)$ for any $\kappa>0$.
\begin{remark}\label{redution potential estimate}
    We emphasize that \eqref{potential estimate thm} is an estimate only at $x_0$, the center of $B_{2R}$. Hence, we may reduce the proof of the theorem significantly to a more restricted case. Namely, we only consider the class of continuous $\mathcal{A}$-supersolutions functions. 
    
     Indeed, since $u$ is a nonnegative $\mathcal{A}$-superharmonic in $B_{2R}$, applying \cite[Proposition~4.5]{MR4482108}, there exists a nondecreasing sequence of nonnegative functions $\{u_j\}\subset C(\overline{B}_R)\cap \mathcal{S}_{\mathcal{A}}(B_R)$ satisfying $u_j=0$ on $\partial B_R$ (for $j=1,2,\ldots$) and
     \begin{equation}\label{estimate22 orlicz}
     u=\lim_{j\to\infty} u_j \quad \mbox{in}\quad B_R\ \mbox{(pointwise)}.   
     \end{equation}
     From \cite[Lemma~4.6]{MR4482108}, $u_j$ is an $\mathcal{A}$-supersolution in $B_R$, for all $j\geq 1$. This implies that $u_j\in W_{\mathrm{loc}}^{1,G}(B_R)$, and $D u_j=\nabla u_j$ for all $j\geq 1$. By Theorem~\ref{harnack principle orlicz5}, we have $Du=\lim_j \nabla u_j$ pointwise in $B_R$, possibly passing to a subsequence. On the other hand, notice that $\{u_j\}$ is bounded in $B_R$ (pointwise), since $u_j\leq u$ in $B_R$ and $u_j=0$ on $\partial B_R$ for all $j\geq 1$. Extending $u_j$ by zero away from $\partial B_R$, we may consider $ u_j$ as an $\mathcal{A}$-supersolution in $B_{3R}$ for all $j\geq 1$. From this, an appeal to \cite[Lemma~5.2]{MR4258794} ensure that there exists $c_0=c_0(p,q,\alpha,\beta)>0$ such that
     \begin{equation*}
         \begin{aligned}
                \int_{B_{R}}G(|\nabla u_j|)\,\mathrm{d} x & 
         \leq c_0\int_{B_{2R}}G\bigg(\frac{\mathrm{osc}_{B_{2R}} u_j}{R}\bigg)\,\mathrm{d} x\leq c_0\int_{B_{2R}}G\bigg(\frac{\sup_{B_{2R}} u_j}{R}\bigg)\,\mathrm{d} x \nonumber\\
        & \leq c_1\,R^nG\bigg(\frac{\sup_{B_{2R}} u_j}{R}\bigg) \quad \forall j\geq 1,
         \end{aligned}
     \end{equation*}
     where $c_1=c_1(n,p,q,\alpha,\beta)>0$. Consequently, $\{\nabla u_j\}$ is bounded in $L^G(B_R)$, and by Theorem~\ref{Lemma weakly convergence}, $\nabla u_j\rightharpoonup Du$ in $L^G(B_R)$. In particular, $Du\in L^G({B_R})$. From \eqref{condition A1 orlicz} and \eqref{condition A2 orlicz}, we deduce that
     \begin{equation*}
         \mathcal{A}(x,\nabla u_j)\rightharpoonup \mathcal{A}(x,Du) \quad \mbox{in}\quad L^{G^{\ast}}(B_R).
     \end{equation*}
     This follows by the same method as in \eqref{estimate14 orlicz}. Thus we also have the weak convergence of corresponding measures $\mu_j=\mu_{u_j}$ to $\mu=\mu_u$ in $B_R$:
     \begin{equation*}
         \begin{aligned}
            \int_{B_R}\varphi\,\mathrm{d}\mu & = \int_{B_R} \mathcal{A}(x,Du)\cdot \nabla \varphi \,\mathrm{d} x= \lim_{j\to\infty} \int_{B_R} \mathcal{A}(x,\nabla u_j)\cdot\nabla \varphi \,\mathrm{d} x \\
             &=\lim_{j\to\infty}\int_{B_R}\varphi\,\mathrm{d}\mu_j \quad \forall \varphi\in C_c^{\infty}(B_R).
         \end{aligned}
     \end{equation*}
     Applying \cite[Theorem~2.2.5]{MR3837546}, 
     \begin{align}
&\varliminf_{j\to\infty}\mu_j(B(x_0,s))\geq \mu(B(x_0,s)) \quad \forall s\leq R, \label{estimate20 orlicz}\\
& \varlimsup_{j\to\infty}\mu_j(\overline{B(x_0,s)})\leq \mu(\overline{B(x_0,s)}) \quad \forall s\leq R. \label{estimate21 orlicz} 
     \end{align}

Now, suppose the bounds in \eqref{potential estimate thm} holds for $u_j$, $j=1,2,\ldots$. Using \eqref{estimate20 orlicz}, we obtain from the lower bound in \eqref{potential estimate thm} and from \eqref{estimate22 orlicz} that
\begin{equation*}
    \begin{aligned}
        u(x_0)&=\lim_{j\to\infty}u_j(x_0)\geq C_1\varliminf_{j\to\infty}\mathbf{W}_G^R\mu_j(x_0)
        = C_1\varliminf_{j\to\infty}\int_0^Rg^{-1}\bigg(\frac{\mu_j(B(x_0,s))}{s^{n-1}}\bigg)\,\mathrm{d} s\\
&=C_1\int_0^Rg^{-1}\bigg(\frac{\varliminf_{j\to\infty}\mu_j(B(x_0,s))}{s^{n-1}}\bigg)\,\mathrm{d} s
        \geq C_1\int_0^Rg^{-1}\bigg(\frac{\mu(B(x_0,s))}{s^{n-1}}\bigg)\,\mathrm{d} s,
    \end{aligned}
\end{equation*}
which shows the lower bound in \eqref{potential estimate thm} for $u(x_0)$. To verify the upper bound in \eqref{potential estimate thm} for $u(x_0)$, first notice that $\inf_{B_R}u_j\leq \inf_{B_R}u$ for all $j\geq 1$. Combining  this with \eqref{estimate21 orlicz} and \eqref{estimate22 orlicz}, one has 
\begin{equation*}
    \begin{aligned}
        u(x_0)&=\lim_{j\to\infty}u_j(x_0)\leq C_2\bigg(\inf_{B_R}u+\varlimsup_{j\to\infty}\mathbf{W}_G^R\mu_j(x_0)\bigg)\\
        & \leq C_2\bigg(\inf_{B_R}u+\varlimsup_{j\to\infty}\int_0^Rg^{-1}\bigg(\frac{\mu_j(\overline{B(x_0,s))}}{s^{n-1}}\bigg)\,\mathrm{d} s\bigg)\\
        &=C_2\bigg(\inf_{B_R}u+\int_0^Rg^{-1}\bigg(\frac{\varlimsup_{j\to\infty}\mu_j(\overline{B(x_0,s))}}{s^{n-1}}\bigg)\,\mathrm{d} s\bigg)\\
        &\leq C_2\bigg(\inf_{B_R}u+\int_0^Rg^{-1}\bigg(\frac{\mu(\overline{B(x_0,s))}}{s^{n-1}}\bigg)\,\mathrm{d} s\bigg).
    \end{aligned}  
\end{equation*}
The upper bound in \eqref{potential estimate thm} for $u(x_0)$ follows from the fact 
\begin{equation}\label{estimate24 orlicz}
    \int_0^Rg^{-1}\bigg(\frac{\mu(\overline{B(x_0,s))}}{s^{n-1}}\bigg)\,\mathrm{d} s=\int_0^Rg^{-1}\bigg(\frac{\mu({B(x_0,s))}}{s^{n-1}}\bigg)\,\mathrm{d} s.
\end{equation}
To prove \eqref{estimate24 orlicz}, note that the function $t\mapsto\mu(B(x_0,t))$ is monotone in $t\geq 0$, whence the set $\{t_0>0: t\mapsto\mu(B(x_0,t))\mbox{ is discontinuous in }t_0\}$ is enumerable (see for instance \cite[Chapter~6, Theorem~1]{MR1013117}). But this set is equal to the set $\{t_0>0: \mu(\partial B(x_0,t_0))\neq 0\}$, which yields \eqref{estimate24 orlicz}.   
\end{remark}

Let us list some preliminary results. Except for the Harnack inequalities given in \eqref{weak harnack inequality}-\eqref{harnack inequality boundary 2} below, these results hold for $N$-functions $G$ satisfying \eqref{data condition elementary}.
The following type-Caccioppoli estimate will be useful to show the lower bound \cite[Proposition~3.24]{MR4803728}. As usual, $\Omega\subseteq \mathds{R}^n$ means a domain.

\begin{lemmaletter}\label{caccioppoli estimate orlicz} If $u\in W_{\mathrm{loc}}^{1,G}(\Omega)$ is a nonnegative $\mathcal{A}$-subsolution, then  there exists a constant $C=C(p,q,\alpha,\beta)>0$, such that
\begin{equation*}
    \int_{\Omega}G(|\nabla u|)\eta^q\,\mathrm{d} x \leq C\int_{\Omega} G(u|\nabla\eta|)\,\mathrm{d} x \quad \forall \eta \in C_c^{\infty}(\Omega).
\end{equation*}
\end{lemmaletter}

 The following Minimum and Maximum Principles \cite[Corollary~4.15 and Corollary~4.16]{MR4482108} will be a helpful ingredient to work together with Harnack's inequalities and consequently to prove the bounds in \eqref{potential estimate thm}.
\begin{lemmaletter}\label{Minimum and Maximum Principles orlicz}
Suppose $\Omega\subset\mathds{R}^n$ bounded and let $D\subset\Omega$ be a connect open subset compactly contained in $\Omega$.
    \begin{flushleft}
        {$\mathbf{(a)}$}  Suppose $u$ is $\mathcal{A}$-superharmonic function and finite (almost everywhere) in $\Omega$. Then
\begin{equation*}
    \inf_D u=\inf_{\partial D}u.
\end{equation*}
{$\mathbf{(b)}$} Suppose $u$ is $\mathcal{A}$-subharmonic function and finite (almost everywhere) in $\Omega$. Then
        \begin{equation*}
            \sup_D u=\sup_{\partial D}.
\end{equation*}       
    \end{flushleft}
\end{lemmaletter}
In what follows, we use the abbreviation
\begin{equation*}
\dashint_{\Omega}u\,\mathrm{d} x=\frac{1}{|\Omega|}\int_{\Omega}   u\,\mathrm{d} x, \quad \Omega\subset\mathds{R}^n \; \mbox{bounded}.
\end{equation*}

Next, we state versions of Harnack's inequalities, which will play a crucial role in the proof of bounds in \eqref{potential estimate thm}. Recall $W^{1,G}(B_{2R})\subset W^{1,p}(B_{2R})$ \cite[Lemma~6.1.6]{MR3931352}. Suppose $g$ is given by \eqref{our g}. With the aid of  \cite[Lemma~3.7 and Corollary~3.8]{MR4482108}, we infer from \cite[Theorem~2.5]{MR3348922} the weak Harnack inequality for $\mathcal{A}$-supersolutions in $B_{2R}$ and for $\mathcal{A}$-harmonic functions in $B_{2R}$, respectively. This is the content of the following theorem.

\begin{theoremletter}
   Let $g$ be given by \eqref{our g}. Let $u$ be a nonnegative $\mathcal{A}$-supersolution in $B_{2R}$. Then there exist constants $c_0=c_0(n,p,q,\alpha,\beta)>0$ and $s_0=s_0(n,p,q,\alpha,\beta)\in (0,1)$ such that
    \begin{equation}\label{weak harnack inequality}
        \bigg(\dashint_{B_{2R}}u^{s_0}\,\mathrm{d} x\bigg)^{\frac{1}{s_0}}\leq c_0\,\inf_{B_R}u.
    \end{equation}
  Furthermore, if $u$ is $\mathcal{A}$-harmonic in $B_{2R}$, there exists $c=c(n,p,q,\alpha,\beta)>0$ such that
  \begin{equation}\label{harnack inequality}
      \sup_{B_R}u\leq c\,\inf_{B_R}u.
  \end{equation}
\end{theoremletter}
Combining Lemma~\ref{Minimum and Maximum Principles orlicz} with \eqref{harnack inequality}, we establish the following result, which will be useful in the proof of the upper bound \eqref{potential estimate thm}.
\begin{corollary}
    Let $g$ be given by \eqref{our g}. Suppose $u$ is $\mathcal{A}$-harmonic in $B_{3R/2}\setminus B_R$. Then there exists a constant $c=c(n,p,q,\alpha,\beta)>0$ such that
    \begin{equation}\label{harnack inequality boundary}
        \sup_{\partial B_{\frac{4}{3}R}} u\leq c\,\inf_{\partial B_{\frac{4}{3}R}}u.
    \end{equation}
\end{corollary}
\begin{proof}
    Let $\varepsilon>0$ be a constant sufficiently small for which
    \begin{equation*}
        \overline{B}_R\subset B_{\frac{4}{3}R-\varepsilon} \subset \overline{B}_{\frac{4}{3}R+\varepsilon} \subset  B_{\frac{3}{2}R}.
    \end{equation*}
        Then $u$ is $\mathcal{A}$-harmonic in the annulus $A_{\varepsilon}:=B_{4R/3+\varepsilon}\setminus B_{4R/3-\varepsilon}$. Recall that $B_{3R/2}\setminus B_R=\{z: R< |z-x_0|< 3/2 R\}$.  We claim that there exists a constant $\delta>0$ sufficiently small such that for all $x\in A_{\varepsilon}$, and for all $y\in B(x,\delta)$, it holds $y\in A_{\varepsilon}$. Indeed, on the contrary, we would find 
 sequences $x_i\in A_{\varepsilon}$ and $y_i\in B(x_i, 1/i)$, satisfying either $|y_i-x_0|\leq R$ or $|y_i-x_0|\geq 3/2 R$. By choice of $\varepsilon$, $A_{\varepsilon}$ is compactly contained in $B_{3R/2}\setminus B_R$, whence we may assume that $x_i$ converges to $\overline{x}$ in $\overline{A}_{\varepsilon}$. Thus, $y_i$ converges to $\overline{x}$, which implies that either $|\overline{x}-x_0|\leq R$ or $|\overline{x}-x_0|\geq 3/2 R$. This contradicts the fact that $\overline{x}\in \overline{A}_{\varepsilon}$ and establishes the claim.
    
    We may cover $A_{\varepsilon}$ with finite number of balls of the form $\{B(x_i,\delta)\}$, $x_i\in A_{\varepsilon}$, $i=1,\ldots, N$. From \eqref{harnack inequality}, it follows
    \begin{equation*}
        \sup_{B(x_i,\delta)}u\leq c\,\inf_{B(x_i,\delta)}u \quad \forall i=1,\ldots, N.
    \end{equation*}
    Using Lemma~\ref{Minimum and Maximum Principles orlicz}, we deduce that
    \begin{equation*}
    \begin{aligned}
        \sup_{\partial {B}_{\frac{4}{3}R+\varepsilon}}u&\leq \sup_{\partial A_{\varepsilon}} u= \sup_{A_{\varepsilon}} u \leq \sup_{\bigcup_{i=1}^N B(x_i,\delta)} u \\
        &\leq c \inf_{\bigcup_{i=1}^N B(x_i,\delta)} u \leq c\, \inf_{A_{\varepsilon}} u\leq  c\,\inf_{\partial A_{\varepsilon}} u\leq c\inf_{\partial{B}_{\frac{4}{3}R+\varepsilon}} u.
    \end{aligned}
    \end{equation*}
    Letting $\varepsilon\to 0$, the corollary follows since $u$ is continuous in $B_{3R/2}\setminus B_R$.
\end{proof}
The same reasoning applies to the following case: \textit{if $u$ is an $\mathcal{A}$-harmonic in $B_{2R}\setminus B_{5R/4}$, then there exists $c=c(n,p,q,\alpha,\beta)>0$ such that
\begin{equation}\label{harnack inequality boundary 2}
    \sup_{\partial B_{\frac{4}{3}R}}u\leq c\inf_{\partial B_{\frac{4}{3}R}}u
\end{equation}
}
 
The following result involves the Poisson modification of superharmonic functions, which together with \eqref{weak harnack inequality} and \eqref{harnack inequality} will be decisive in the proof of the upper bound \eqref{potential estimate thm}. We recall the notion of a bounded \textit{regular} set $\Omega$. A bounded domain $\Omega$ is called regular if, on any boundary point, the boundary value of any
$\mathcal{A}$-harmonic function is attained in the distributional sense and pointwise. A known criterion, so-called \textit{the Wiener Criterion}, characterizes a bounded regular set by a geometric quantity on $\partial\Omega$. We refer the reader to \cite[Remark~3.8]{MR4258773} for more details. In particular, balls and annuli are regular sets. 

Now let $\Omega'$ be a bounded domain and $\Omega\subset \Omega'$ an open subset compactly contained in $\Omega'$ with $\overline{\Omega}$ regular. Let $u$ be an $\mathcal{A}$-superharmonic and finite (almost everywhere) in $\Omega'$, that is $u\in \mathcal{S}_{\mathcal{A}}(\Omega')$. For $x\in \Omega$, we define
\begin{equation*}
    u_{\Omega}(x)=\inf\left\{v(x): v\in\mathcal{S}_{\mathcal{A}}(\Omega),\;\; \varliminf\limits_{\;y\to x\; y\in \Omega}v(y)\geq u(x)\right\},
\end{equation*}
and the \textit{Poisson modification} of $u$ in $\Omega$ is given by
\begin{equation*}
    P(u,\Omega)(x):=\left\{\begin{aligned}
        & u_{\Omega}(x) &\mbox{if }&x\in \Omega  \\
        & u(x) & \mbox{if }& x\in \Omega'\setminus \Omega.
    \end{aligned}
    \right.
\end{equation*}
The Poisson modification carries the idea of local smoothing of an $\mathcal{A}$-superharmonic function in a regular set. This is the content of the next theorem \cite[Theorem~3]{MR4482108}
\begin{theoremletter}\label{Poisson modification thm}
    Let $u\in \mathcal{S}_{\mathcal{A}}(\Omega')$ and let $\Omega\subset \Omega'$ an open subset compactly contained in $\Omega'$ with $\overline{\Omega}$ regular. Then
    \begin{flushleft}
 {$\mathbf{(i)}$} $P(u,\Omega)\in \mathcal{S}_{\mathcal{A}}(\Omega')$,   \\
 {$\mathbf{(ii)}$} $P(u,\Omega)$ is $\mathcal{A}$-harmonic in $\Omega$, \\
 {$\mathbf{(iii)}$} $P(u,\Omega)\leq u$ in $\Omega'$.
    \end{flushleft}
\end{theoremletter}

We need the following result regarding Sobolev functions \cite[Lemma~3.5]{MR3606780}.
\begin{lemmaletter}\label{lemma regarding sobolev function}
    If $u\in W^{1,G}(\Omega)$ with $\mathrm{supp }\,u\subset \Omega$, then $u\in W_0^{1,G}(\Omega)$.
\end{lemmaletter}

 We are now able to prove the estimates in \eqref{potential estimate thm}.  By Remark~\ref{redution potential estimate}, we suppose $u$ is a continuous bounded $\mathcal{A}$-supersolution in $B_{2R}=B(x_0,2R)$. To prove the upper bound, we may also reduce it to a simpler case. We will modify $u$ to be a $\mathcal{A}$-harmonic function in a countable union of disjoint annuli shrinking to the reference point $x_0$.  For this purpose, we use the Poisson Modification of $u$ over a family of annuli. The crucial fact is that the corresponding measure in each annulus concentrates on the boundary of the particular annulus, but in a controllable way since the measure corresponding to the new solution belongs to $\big(W^{1,G}(B_R)\big)^{\ast}$.

 To be more precise, let $R_k=2^{1-k}R$ and $B_k=2^{1-k}B_R=B(x_0,R_k)$, $k=0,1,2,\ldots$. Consider the union of annuli
 \begin{equation*}
     \Omega=\bigcup_{k=1}^\infty\frac{3}{2}B_k\setminus \overline{B}_k.
 \end{equation*}
By definition, $\Omega$ is regular, and we can consider $v:=P(u,\Omega)$. From Theorem~\ref{Poisson modification thm}, $v=u$ in $B_{2R}\setminus \Omega$, $\,v\in \mathcal{S}_{\mathcal{A}}(B_{2R})$ and $v$ is $\mathcal{A}$-harmonic in $\Omega$, that is
 \begin{equation*}
     -\mathrm{div}\big(\mathcal{A}(x,\nabla v)\big)=0 \quad \mbox{in}\quad \Omega.
 \end{equation*}
Notice that $v$ is continuous by the assumed continuity of $u$, whence $v$ is also $\mathcal{A}$-supersolution in $B_{2R}$. Consequently, there exists $\mu_v\in M^+(B_{2R})$ such that
\begin{equation}\label{eq. general Poisson v}
    -\mathrm{div}\big(\mathcal{A}(x,\nabla v)\big)=\mu_v \quad \mbox{in}\quad B_{2R}.
\end{equation}
 
\begin{proof}[Proof of the upper bound in Theorem~\ref{Maly's type result}]
 The basic idea is introducing comparison solutions with zero boundary values and measures given by $\mu_v$. Recall $\mu=\mu_u$. We begin by establishing that
 \begin{equation}\label{estimate25 orlicz}
     \mu_v(B_k)=\mu(B_k) \quad \forall k\geq 0.
 \end{equation}
 This is a consequence of the inner regularity. Indeed, let $\varphi\in C_c^{\infty}(B_k)$ such that $0\leq \varphi\leq 1$ and $\varphi=1$ on a compact set $K$ satisfying
 \begin{equation*}
     \frac{3}{2}\overline{B}_{k+1}\subset K\subset B_k.
 \end{equation*}
Note that $u=v$ in $\mathrm{supp}\nabla\varphi \subset B_k\setminus (3/2\overline{B}_{k+1})\nsubseteq \Omega$. From this, we deduce by testing $\varphi$ in \eqref{eq. general u appendix super} and  in \eqref{eq. general Poisson v} that
\begin{equation*}
\begin{aligned}
     \int_{B_k}\varphi\,\mathrm{d}\mu_v &= \int_{B_k}\mathcal{A}(x,\nabla v)\cdot \nabla \varphi \,\mathrm{d} x\\
     &= \int_{B_k}\mathcal{A}(x,\nabla u)\cdot \nabla \varphi \,\mathrm{d} x =   \int_{B_k}\varphi\,\mathrm{d}\mu.
\end{aligned}
\end{equation*}
Consequently, $\mu_v(K)=\mu(K)$ and, by exhausting $B_k$ with such $K$, the inner regularity of these measures yields \eqref{estimate25 orlicz}.

Next, notice that $\mu_v$ belongs to $ \big(W_0^{1,G}(4/3 B_{k+1})\big)^{\ast}$ for all $k\geq 0$, since $\mu_v\in \big(W_{\mathrm{loc}}^{1,G}(B_{2R})\big)^{\ast}\hookrightarrow \big(W_0^{1,G}(4/3 B_{k+1})\big)^{\ast}$ for all $k\geq 0$. From
 Theorem~\ref{existence in W0}, there exists $v_k\in W_0^{1,G}(4/3 B_{k+1})$ satisfying
 \begin{equation}\label{estimate26 orlicz}              
      -\mathrm{div}\big(\mathcal{A}(x,\nabla v_k)\big)=\mu_v \quad \mbox{in}\quad \frac{4}{3}B_{k+1}.
 \end{equation}
Setting $\mu_k=\mu_{v_k}$, one has
\begin{equation}\label{estimate27 orlicz}
    \mu_k\Big(\frac{4}{3}B_{k+1}\Big)=\mu_v\Big(\frac{4}{3}B_{k+1}\Big) \quad \forall k\geq 0.
\end{equation}
In light of Lemma~\ref{comparison principle super/subsolution}, we may assume $v_k\geq 0$ (almost everywhere). Since $4/3 B_{k+1}\setminus \overline{B}_{k+1} \subset 3/2 B_{k+1}\setminus \overline{B}_{k+1}$ and $v$ is $\mathcal{A}$-harmonic in $3/2 B_{k+1}\setminus \overline{B}_{k+1}$, it follows from \eqref{estimate27 orlicz} that $v_k$ is  $\mathcal{A}$-harmonic in $4/3 B_{k+1}\setminus \overline{B}_{k+1}$, whence $v_k$ takes continuously zero boundary value on $\partial (4/3 B_{k+1})$. Moreover, $v-\max_{\partial (4/3 B_{k+1})}v\leq 0$ on $\partial (4/3 B_{k+1})$. From this,

\begin{equation*}
    \Big( v-\max_{\partial \frac{4}{3} B_{k+1}}v - v_k \Big)_+ =0 \quad \mbox{on} \quad \partial \frac{4}{3} B_{k+1}.
\end{equation*}
This means that $\mathrm{supp}\big(v-\max_{\partial (4/3 B_{k+1})}v - v_k\big)_+ \subset 4/3 B_{k+1}$ and, by Lemma~\ref{lemma regarding sobolev function},
\begin{equation*}
    \Big( v-\max_{\partial \frac{4}{3} B_{k+1}}v - v_k \Big)_+\in W_0^{1,G}\Big(\frac{4}{3}B_{k+1}\Big).
\end{equation*}
A subtraction of $v$ and $v_k$ equations, \eqref{eq. general Poisson v} and \eqref{estimate26 orlicz} respectively, with the previous test function, gives
\begin{multline*}
     0=\int_{\frac{4}{3}B_{k+1}}\Big( v-\max_{\partial \frac{4}{3} B_{k+1}}v - v_k \Big)_+\mathrm{d}\mu_k- \int_{\frac{4}{3}B_{k+1}}\Big( v-\max_{\partial \frac{4}{3} B_{k+1}}v - v_k \Big)_+\mathrm{d}\mu_v\\
        =\int_{\frac{4}{3}B_{k+1}}\mathcal{A}(x,\nabla v_k)\cdot \nabla \Big( v-\max_{\partial \frac{4}{3} B_{k+1}}v - v_k \Big)_+\mathrm{d} x- \int_{\frac{4}{3}B_{k+1}}\mathcal{A}(x,\nabla v)\cdot \nabla \Big( v-\max_{\partial \frac{4}{3} B_{k+1}}v - v_k \Big)_+\mathrm{d} x.
\end{multline*}
On account of $\mathrm{supp}\nabla\big(v-\max_{\partial (4/3 B_{k+1})}v - v_k\big)_+ \subset 4/3 B_{k+1}\cap \{v-\max_{\partial (4/3 B_{k+1})}v \geq  v_k\}$, we have from the previous equality
\begin{equation*}
    \begin{aligned}
    0&=\int_{\frac{4}{3}B_{k+1}\cap \{v-\max_{\partial (4/3 B_{k+1})}v \geq  v_k\}}\mathcal{A}(x,\nabla v_k)\cdot \nabla(v-v_k)\,\mathrm{d} x\\
    & \quad - \int_{\frac{4}{3}B_{k+1}\cap \{v-\max_{\partial (4/3 B_{k+1})}v \geq  v_k\}}\mathcal{A}(x,\nabla v_k)\cdot \nabla(v-v_k)\,\mathrm{d} x\\
    & = \int_{\frac{4}{3}B_{k+1}\cap \{v-\max_{\partial (4/3 B_{k+1})}v \geq  v_k\}}\big(\mathcal{A}(x,\nabla v_k)- \mathcal{A}(x,\nabla v)\big)\cdot \nabla(v-v_k)\,\mathrm{d} x \leq 0,
    \end{aligned}
\end{equation*}
where the last inequality is due the monotonicity of $\mathcal{A}$ in \eqref{condition A2 orlicz}.  Accordingly, $\nabla\big(v-\max_{\partial (4/3 B_{k+1})}v - v_k\big)_+=0$ in $4/3 B_{k+1}$, whence
\begin{equation}\label{estimate28 orlicz}
  v_k\geq  v-\max_{\partial \frac{4}{3} B_{k+1}}v \quad\mbox{in}\quad \frac{4}{3}B_{k+1}.
\end{equation}
Note that $3/2 B_{k+2}\setminus \overline{B}_{k+2}\subset 4/3 B_{k+1}$. By \eqref{estimate27 orlicz}, we have $\mu_k(3/2 B_{k+2}\setminus \overline{B}_{k+2})=\mu_v(3/2 B_{k+2}\setminus \overline{B}_{k+2})=0$, since $v$ is $\mathcal{A}$-harmonic in $3/2 B_{k+2}\setminus \overline{B}_{k+2}$. From this, $v_k$ is $\mathcal{A}$-harmonic in $3/2 B_{k+2}\setminus \overline{B}_{k+2}$. Using Harnack's inequality \eqref{harnack inequality boundary}, there exists a constant $c_1=c_1(n,p,q,\alpha,\beta)>0$ such that
\begin{equation}\label{estimate29 orlicz}
    \max_{\partial \frac{4}{3} B_{k+2}}v_k \leq c_1\,\min_{\partial \frac{4}{3} B_{k+2}}v_k.
\end{equation}

We will consider two cases. First, assume that $\min_{\partial (4/3)B_{k+2}}v_k=0$. From \eqref{estimate29 orlicz}, $\max_{\partial (4/3 B_{k+2})}v_k=0$, and by \eqref{estimate28 orlicz}
\begin{equation}\label{estimate30 orlicz}
     \max_{\partial \frac{4}{3} B_{k+2}}v- \max_{\partial \frac{4}{3} B_{k+1}}v\leq 0.
\end{equation}
Next, suppose that $\min_{\partial (4/3)B_{k+2}}v_k>0$. For $k=0,1,\ldots$, we set
\begin{equation*}
        w_k(x)=\min\Big\{v_k(x),\, \min_{\partial \frac{4}{3} B_{k+2}}v_k\Big\}, \quad x\in \frac{4}{3}B_{k+1}.
\end{equation*}
From \cite[Cor.~4.2]{MR4482108}, $w_k\in \mathcal{S}_{\mathcal{A}}(4/3 B_{k+1})$. We claim that
\begin{equation}\label{estimate31 orlicz}
\mu_{w_k}\Big(\frac{4}{3} B_{k+1}\Big)= \mu_{v}\Big(\frac{4}{3} B_{k+1}\Big)   \quad \forall k\geq 0.
\end{equation}
Indeed, if we prove that $\mu_{w_k}(4/3 B_{k+1})=\mu_k(4/3 B_{k+1})$ for all $k\geq 0$, the assertion follows by \eqref{estimate27 orlicz}. By the inner regularity,
\begin{equation*}
    \mu_{w_k}\Big(\frac{4}{3}B_{k+1}\Big)=\sup \mu_{w_k}(K),
\end{equation*}
where the supremum is taken over all compact sets $K\subset 4/3 B_{k+1}$. Let $\varphi_K\in C_c^{\infty}(4/3 B_{k+1})$ such that $0\leq\varphi_K\leq 1$ and $\varphi_K=1$ on $K$. We may suppose that $4/3 B_{k+2}\subset K$. By continuity of $v_k$, $v_k\leq \min_{\partial (4/3 B_{k+2})}v_k$ in a neighborhood $V$ of $\partial (4/3 B_{k+1})$, since $v_k=0$ on $\partial (4/3 B_{k+1})$. It follows that $w_k=v_k$ in $V$ and $\mathrm{supp}\nabla \varphi_K \subset V$. By taking the supremum in all compact sets $K$ with $4/3 B_{k+1}\setminus K \subset V$, we arrive at
\begin{equation*}
    \begin{aligned}
      \mu_{w_k}\Big(\frac{4}{3}B_{k+1}\Big)   & = \sup \mu_{w_k}(K)= \sup \int_{\frac{4}{3}B_{k+1}} \varphi_K\,\mathrm{d}\mu_{w_k}\\
        &= \sup  \int_{\frac{4}{3}B_{k+1}} \mathcal{A}(x,\nabla w_k)\cdot \nabla \varphi_K\,\mathrm{d} x\\
        &= \sup  \int_{\frac{4}{3}B_{k+1}} \mathcal{A}(x,\nabla v_k)\cdot \nabla \varphi_K\,\mathrm{d} x=\sup \int_{\frac{4}{3}B_{k+1}} \varphi_K\,\mathrm{d}\mu_{k}=\mu_k\Big(\frac{4}{3}B_{k+1}\Big).
    \end{aligned}
\end{equation*}
Accordingly, it follows from \eqref{estimate31 orlicz} and \eqref{condition A2 orlicz} that
\begin{align}
    \Big(\min_{\partial \frac{4}{3} B_{k+2}}v_k\Big)\mu_v\Big(\frac{4}{3}B_{k+1}\Big)&=\int_{\frac{4}{3}B_{k+1}} \Big(\min_{\partial \frac{4}{3} B_{k+2}}v_k\Big)\,\mathrm{d} \mu_v \geq  \int_{\frac{4}{3}B_{k+1}} w_k\,\mathrm{d} \mu_{v}\nonumber\\
&= \int_{\frac{4}{3}B_{k+1}} w_k\,\mathrm{d} \mu_{w_k} =\int_{\frac{4}{3}B_{k+1}} \mathcal{A}(x,\nabla w_k)\cdot \nabla w_k\,\mathrm{d} x\nonumber\\
&\geq \alpha \int_{\frac{4}{3}B_{k+1}} G(|\nabla w_k|)\,\mathrm{d} x.\nonumber
\end{align}
Combining the Modular Poincaré inequality (Lemma~\ref{modular poincare ineq}) with \eqref{estimate on G} in the previous estimate, we obtain
\begin{equation*}
    \begin{aligned}
           \Big(\min_{\partial \frac{4}{3} B_{k+2}}v_k\Big)\mu_v\Big(\frac{4}{3}B_{k+1}\Big)&\geq c_2 \int_{\frac{4}{3}B_{k+1}} G\Big( \frac{w_k}{R_{k+1}}\Big)\,\mathrm{d} x \geq c_3 \int_{\frac{4}{3}B_{k+1}} G\Big( \frac{w_k}{R_{k}}\Big)\,\mathrm{d} x\\
           & \geq \int_{\frac{4}{3}B_{k+2}} G\Big( \frac{w_k}{R_{k}}\Big)\,\mathrm{d} x \geq  \int_{\frac{4}{3}B_{k+2}} G\Big( \frac{\min_{\partial \frac{4}{3} B_{k+2}}v_k}{R_{k}}\Big)\,\mathrm{d} x\\
           &=c_4 R_k^{n}\, G\Big( \frac{\min_{\partial \frac{4}{3} B_{k+2}}v_k}{R_{k}}\Big),
    \end{aligned}
\end{equation*}
where $c_2,\, c_3$ and $c_4$ are positive constants depending only on $n$ and $p,q,\alpha,\beta$. Consequently, by \eqref{data condition elementary}
\begin{equation*}
    \frac{\mu_v\Big(\frac{4}{3}B_{k+1}\Big)}{R_k^{n-1}}\geq c_4 \Big(\frac{\min_{\partial \frac{4}{3} B_{k+2}}v_k}{R_k}\Big)^{-1}\, G\Big( \frac{\min_{\partial \frac{4}{3} B_{k+2}}v_k}{R_{k}}\Big)\geq c_4\, g\Big( \frac{\min_{\partial \frac{4}{3} B_{k+2}}v_k}{R_{k}}\Big),
\end{equation*}
where $c_4=c_4(n,p,q,\alpha,\beta)>0$. Since $\min_{\partial(4/3 B_{k+2})}v_k>0$, \eqref{estimate29 orlicz} leads to
\begin{equation*}
    \max_{\partial \frac{4}{3}B_{k+2} }v_k\leq c_5\,R_k \,g^{-1}\Big(\frac{\mu_v\big(\frac{4}{3}B_{k+1}\big)}{R_{k}^{n-1}}\Big),
\end{equation*}
where $c_5=c_5(n,p,q,\alpha,\beta)>0$ is obtained from combining $c_4$ with \eqref{estimate on g-1}. By \eqref{estimate28 orlicz}, the preceding inequality gives
\begin{equation}\label{estimate32 orlicz}
     \max_{\partial \frac{4}{3}B_{k+2} }v-  \max_{\partial \frac{4}{3}B_{k+1} }v \leq c_5\,R_k \,g^{-1}\Big(\frac{\mu_v\big(\frac{4}{3}B_{k+1}\big)}{R_{k}^{n-1}}\Big).
\end{equation}

Thus, in all cases, by summing up \eqref{estimate30 orlicz} and \eqref{estimate32 orlicz} in $k=2,3,\ldots$, we deduce from \eqref{harnack inequality boundary} that
\begin{align}
    \varlimsup_{k\to \infty} \max_{\partial \frac{4}{3}B_{k+2} }v&\leq \max_{\partial \frac{4}{3}B_{3} }v +c_5\sum_{k=2}^{\infty}\,R_k \,g^{-1}\Big(\frac{\mu_v\big(\frac{4}{3}B_{k+1}\big)}{R_{k}^{n-1}}\Big)\nonumber\\
    & \leq c_6\min_{\partial \frac{4}{3}B_{3} }v +c_5\sum_{k=2}^{\infty}\,R_k \,g^{-1}\Big(\frac{\mu_v\big(\frac{4}{3}B_{k+1}\big)}{R_{k}^{n-1}}\Big),\label{estimate33 orlicz}
\end{align}
    since $v$ is $\mathcal{A}$-harmonic in $3/2 B_3\setminus \overline{B}_3$. Recall that $v\leq u$ in $B_{2R}$. From this, combining the weak Harnack inequality \eqref{weak harnack inequality} with Minimum Principle (Lemma~\ref{Minimum and Maximum Principles orlicz} (\textbf{a})), we obtain
    \begin{equation*}
        \begin{aligned}
          \min_{\partial \frac{4}{3}B_{3} }v&\leq \inf_{\partial \frac{4}{3}B_{3} } u \leq \bigg(\dashint_{ \frac{4}{3}B_{3} } u^{s_0}\,\mathrm{d} x\bigg)^{\frac{1}{s_0}}\\
          &\leq c_7 \bigg(\dashint_{B_{2R}} u^{s_0}\,\mathrm{d} x\bigg)^{\frac{1}{s_0}}\leq c_8\, \inf_{B_R}u,
        \end{aligned}
    \end{equation*}
where $c_7=c_7(n)>0$ and $c_8=c_8(n,p,q,\alpha,\beta)>0$. Using this in \eqref{estimate33 orlicz}, there exists $c_9=c_9(n,p,q,\alpha,\beta)>0$ such that
\begin{equation}\label{estimate34 orlicz}
    \varlimsup_{k\to \infty} \max_{\partial \frac{4}{3}B_{k+2} }v \leq c_9\bigg(\inf_{B_R}u +\sum_{k=2}^{\infty}\,R_k \,g^{-1}\Big(\frac{\mu_v\big(\frac{4}{3}B_{k+1}\big)}{R_{k}^{n-1}}\Big)\bigg).
\end{equation}

On the other hand, by definition of Poisson modification, $u=v$ in $B_k\setminus (3/2 \overline{B}_{k+1}) \subset B_{2R}\setminus\Omega$ for all $k\geq 2$. Due to continuity of $u$ and $v$, we have
\begin{equation*}
    u(x_0)=\lim_{k\to\infty}\min_{B_k\setminus \frac{3}{2} \overline{B}_{k+1}}u=\lim_{k\to\infty}\min_{B_k\setminus \frac{3}{2} \overline{B}_{k+1}}v\leq \varlimsup_{k\to \infty} \max_{\partial \frac{4}{3}B_{k+2} }v. 
\end{equation*}
A combination of \eqref{estimate34 orlicz} with the previous inequality, yields
\begin{equation*}
      u(x_0)\leq c_9\bigg(\inf_{B_R}u +\sum_{k=2}^{\infty}\,R_k \,g^{-1}\Big(\frac{\mu_v\big(\frac{4}{3}B_{k+1}\big)}{R_{k}^{n-1}}\Big)\bigg).
\end{equation*}
According to \eqref{estimate25 orlicz} and reminding of $R_k=2^{1-k}R$ for $k\geq 0$,  we estimate the preceding series as follows:
\begin{equation*}
    \begin{aligned}
\sum_{k=2}^{\infty}\,R_k \,g^{-1}\Big(\frac{\mu_v\big(\frac{4}{3}B_{k+1}\big)}{R_{k}^{n-1}}\Big)&= \sum_{k=2}^{\infty}\,R_k \,g^{-1}\Big(\frac{\mu\big(\frac{4}{3}B_{k+1}\big)}{R_{k}^{n-1}}\Big)\\
&\leq \sum_{k=2}^{\infty}\,R_k \,g^{-1}\Big(\frac{\mu\big(\frac{4}{3}B_{k+1}\big)}{R_{k}^{n-1}}\Big)\\
& =\sum_{k=2}^{\infty}\,(R_{k-1}-R_k) \,g^{-1}\Big(\frac{\mu\big(\frac{2}{3}B_{k}\big)}{R_{k}^{n-1}}\Big)\\
&\leq c_{10}\sum_{k=2}^{\infty}\,(R_{k-1}-R_k) \,g^{-1}\Big(\frac{\mu(B_{k})}{R_{k-1}^{n-1}}\Big)\\
&= c_{10}\sum_{k=2}^{\infty}\int_{R_k}^{R_{k-1}} g^{-1}\Big(\frac{\mu(B(x_0,R_k))}{R_{k-1}^{n-1}}\Big)\,\mathrm{d} x\\
&\leq c_{10}\sum_{k=2}^{\infty}\int_{R_k}^{R_{k-1}} g^{-1}\Big(\frac{\mu(B(x_0,s))}{s^{n-1}}\Big)\,\mathrm{d} x\leq c_{10}\, \mathbf{W}_G^R\mu(x_0),
    \end{aligned}
\end{equation*}
where $c_{10}=c_{10}(p,q,\alpha,\beta)>0$. This completes the proof of the upper bound in \eqref{potential estimate thm} by taking $C_2=c_9\max\{1,\,c_{10}\}$.
\end{proof}

We next prove the lower bound. Observe that here, we do not need to use the Poisson modification of $u$.
\begin{proof}[Proof of the lower bound in Theorem~\ref{Maly's type result}]

For $k=0,1,2,\ldots$, let $\eta_k\in C_c^{\infty}(B_k)$ satisfying $0\leq\eta_k\leq 1$, $\mathrm{supp}\,\eta_k\subset 5/4 B_{k+1}$ and $\eta_k=1$ in $B_{k+1}$. We set $\mu_k=\eta_k\mu$. Notice that $\mu_k\in \big(W_0^{1,G}(B_k)\big)^{\ast}$ and $\mu_k(B_{k+1})=\mu(B_{k+1})$. Using Theorem~\ref{existence in W0}, there exists $u_k\in W_0^{1,G}(B_k)$ satisfying
\begin{equation}\label{estimate40 orlicz}
    -\mathrm{div}\big( \mathcal{A}(x,\nabla u_k)\big)=\mu_k \quad \mbox{in}\quad B_k.
\end{equation}
On account of the $\mathrm{supp}\,\eta_k\subset 5/4 B_{k+1}$, one has
\begin{equation*}
    u_k \quad\mbox{is }\mathcal{A}\mbox{-harmonic in}\quad B_k\setminus \frac{5}{4}\overline{B}_{k+1}, 
\end{equation*}
and $u_k=0$ continuously on $\partial B_k$. Since $(-u+\min_{\partial B_k}u)_+=0$ on $\partial B_k$, it follows that $(u_k-u+\min_{\partial B_k}u)_+=0$ on $\partial B_k$, whence $\mathrm{supp}(u_k-u+\min_{\partial B_k}u)_+\subset B_k$ and, by Lemma~\ref{lemma regarding sobolev function},
\begin{equation*}
    \big(u_k-u+\min_{\partial B_k}u\big)_+\in W_0^{1,G}(B_k).
\end{equation*}
A subtraction of equations of $u$ and $u_k$, \eqref{eq. general u appendix super} and \eqref{estimate40 orlicz}, respectively, with the preceding test function gives
\begin{equation*}
\begin{aligned}
    0&\leq \int_{B_k}   \big(u_k-u+\min_{\partial B_k}u\big)_+\,\mathrm{d}\mu - \int_{B_k}   \big(u_k-u+\min_{\partial B_k}u\big)_+\,\mathrm{d}\mu_k\\
    &=\int_{B_k\cap\{u-\min_{\partial B_k}u\leq u_k\}}\big(\mathcal{A}(x,\nabla u)-\mathcal{A}(x,\nabla u_k)\big)\cdot \nabla (u_k- u)\,\mathrm{d} x\leq 0,
\end{aligned}    
\end{equation*}
where in the last inequality was used the monotonicity of $\mathcal{A}$ \eqref{condition A2 orlicz}. From this, $\nabla(u_k-u+\min_{\partial B_k}u)_+=0$ in $B_k$, and consequently
\begin{equation}\label{estimate41 orlicz}
     u_k\leq u-\min_{\partial B_k}u \quad \mbox{in}\quad B_k.
\end{equation}
Let $\varphi\in C_c^{\infty}(B_k)$ be such that
\begin{equation}\label{estimate42 orlicz}
    \left\{ 
    \begin{aligned}
        & 0\leq \varphi\leq 1\quad\mbox{in}\quad B_k, \\
        &\varphi=1 \quad\mbox{in}\quad \frac{2}{3}B_k, \quad |\nabla \varphi|\leq \frac{c}{R_k}.
    \end{aligned}
    \right.
\end{equation}
Observe that $\mathrm{supp}\nabla \varphi \subset B_k\setminus (2/3\overline{B}_k)\subset B_k\setminus (5/4 \overline{B}_{k+1})$. Hence $u_k$ is $\mathcal{A}$-harmonic in $\mathrm{supp}\nabla \varphi$. By Maximum and Minimum Principles (Lemma~\ref{Minimum and Maximum Principles orlicz}), we obtain respectively
\begin{align}
     u_k(x) &=\min\Big\{u_k(x),\, \max_{\partial \frac{2}{3}B_k}u_k\Big\} \quad x\in\mathrm{supp}\nabla\varphi,\label{estimate43 orlicz}\\
         \min_{\partial \frac{2}{3}B_k} u_k&\leq \min\Big\{ u_k,\, \max_{\partial \frac{2}{3}B_k}u_k\Big\} \quad \mbox{in}\quad \frac{2}{3}B_k.\label{estimate44 orlicz}
\end{align}

We will consider two cases. First assume $\min_{\partial B_{k+1}}u_k>0$. By the weak Harnack inequality \eqref{weak harnack inequality}, this positivity implies
\begin{equation*}
    \begin{aligned}
    \min_{\partial \frac{2}{3}B_k}u_k\geq \frac{1}{c_1}\bigg(\dashint_{\frac{2}{3}B_k}u_k^{s_0}\,\mathrm{d} x\bigg)^{\frac{1}{s_0}}\geq \Big(\frac{4}{3}\Big)^{n}\frac{1}{c_1}\bigg(\dashint_{B_{k+1}}u_k^{s_0}\,\mathrm{d} x\bigg)^{\frac{1}{s_0}}  \geq \frac{1}{c_2}\min_{\partial B_{k+1}}u_k>0, 
    \end{aligned}
\end{equation*}
where $c_1>0$ and $c_2>0$ are constants depending only on $n,\,p,q,\alpha,\beta$. Using the previous inequality and taking into account $\varphi$ given in \eqref{estimate42 orlicz}, we compute by \eqref{estimate44 orlicz}
\begin{align}
   \Big(\min_{\partial \frac{2}{3}B_k}u_k\Big)\mu(B_{k+1})&=\int_{B_{k+1}} \min_{\partial \frac{2}{3}B_k}u_k\,\mathrm{d}\mu = \int_{B_{k+1}} \min_{\partial \frac{2}{3}B_k}u_k\,\mathrm{d}\mu_k  \nonumber\\
   &= \int_{B_{k+1}} \min_{\partial \frac{2}{3}B_k}u_k\,\varphi^q\,\mathrm{d}\mu_k \leq \int_{\frac{2}{3}B_k} \min_{\partial \frac{2}{3}B_k}u_k\,\varphi^q\,\mathrm{d}\mu_k \nonumber\\
   & \leq \int_{\frac{2}{3}B_k} \min\Big\{ u_k,\, \max_{\partial \frac{2}{3}B_k}u_k\Big\}\,\varphi^q\,\mathrm{d}\mu_k \leq \int_{B_k} \min\Big\{ u_k,\, \max_{\partial \frac{2}{3}B_k}u_k\Big\}\,\varphi^q\,\mathrm{d}\mu_k. \label{estimate45 orlicz}
\end{align}
 From Lemma~\ref{lemma regarding sobolev function}, $\phi:= \min\big\{ u_k,\, \max_{\partial (2/3B_k)}u_k\big\}\varphi^q\in W_0^{1,G}(B_k)$, whence testing $\phi$ in \eqref{estimate40 orlicz} and using \eqref{estimate45 orlicz}, it follows
 \begin{align}
       \Big(\min_{\partial \frac{2}{3}B_k}u_k\Big)\mu(B_{k+1})&\leq \int_{B_k}\mathcal{A}(x,\nabla u_k)\nabla \cdot\Big(\min\Big\{ u_k,\, \max_{\partial \frac{2}{3}B_k}u_k\Big\}\varphi^q\Big) \mathrm{d} x \nonumber\\
      &= \int_{B_k}\mathcal{A}(x,\nabla u_k)\cdot\Big(\varphi^q\, \nabla \min\Big\{ u_k,\, \max_{\partial \frac{2}{3}B_k}u_k\Big\}\Big) \mathrm{d} x \nonumber\\
  & \quad + q\int_{B_k}\mathcal{A}(x,\nabla u_k)\cdot\Big(\min\Big\{ u_k,\, \max_{\partial \frac{2}{3}B_k}u_k\Big\}\varphi^{q-1}\, \nabla \varphi\Big) \mathrm{d} x \nonumber\\
  & =:  I_1+ I_2\label{estimate46 orlicz}
 \end{align}
Since $\nabla \min\big\{ u_k,\, \max_{\partial \frac{2}{3}B_k}u_k\big\}=0$ in $B_k\cap\{u_k>\max_{\partial(2/3)B_k} u_k\}$, combining Cauchy-Schawrz inequality with \eqref{data condition elementary} and \eqref{condition A2 orlicz}, we estimate $I_1$ as follows
\begin{multline}
    I_1=\int_{B_k\cap\{u_k\leq \max_{\partial \frac{2}{3}B_k}u_k\}}\mathcal{A}(x,\nabla u_k)\cdot\Big(\varphi^q\, \nabla \min\Big\{ u_k,\, \max_{\partial \frac{2}{3}B_k}u_k\Big\}\Big) \mathrm{d} x \\
= \int_{B_k\cap\{u_k\leq \max_{\partial \frac{2}{3}B_k}u_k\}}\mathcal{A}(x,\nabla u_k)\cdot\nabla 
u_k\, \varphi^q\,\mathrm{d} x 
\leq q\beta \int_{B_k\cap\{u_k\leq \max_{\partial \frac{2}{3}B_k}u_k\}}G(|\nabla u_k|)\, \varphi^q\,\mathrm{d} x \\
\leq q\beta \int_{B_k\cap\{u_k\leq \max_{\partial \frac{2}{3}B_k}u_k\}}G\Big(\big|\nabla \min\Big\{ u_k,\, \max_{\partial \frac{2}{3}B_k}u_k\big|\Big)\, \varphi^q\,\mathrm{d} x \qquad \qquad \qquad\\
\leq q\beta \int_{B_k}G\Big(\big|\nabla \min\Big\{ u_k,\, \max_{\partial \frac{2}{3}B_k}u_k\big|\Big)\, \varphi^q\,\mathrm{d} x.\qquad\qquad\qquad\qquad \label{estimate47 orlicz}
\end{multline}
Note that being $u_k$ an $\mathcal{A}$-supersolution in $B_k$, $\min\big\{u_k, \max_{\partial (2/3 B_k)}u_k\big\}$ is also an $\mathcal{A}$-supersolution in $B_k$, thence
\begin{equation*}
    w_k:= \max_{\partial \frac{2}{3}B_k}u_k- \min\Big\{u_k, \max_{\partial \frac{2}{3}B_k}u_k\Big\} \quad \mbox{is an }\mathcal{A}\mbox{-subsolution in }B_k,
\end{equation*}
and it is nonnegative by definition. Applying Caccioppoli estimate (Lemma~\ref{caccioppoli estimate orlicz}) to $w_k$ in $B_k$ and taking into account \eqref{estimate42 orlicz},
\begin{equation*}
    \begin{aligned}
       \int_{B_k}G\Big(\big|\nabla \min\Big\{ u_k,\, \max_{\partial \frac{2}{3}B_k}u_k\big|\Big)\, \varphi^q\,\mathrm{d} x &\leq \int_{B_k}G(|\nabla w_k|)\, \varphi^q\,\mathrm{d} x\\
       &\leq c_3\int_{B_k}G(w_k\,|\nabla \varphi|)\,\mathrm{d} x \leq c_3\int_{B_k\cap \,\mathrm{supp }\nabla \varphi}G\Big(\frac{c\,w_k}{R_k}\Big)\,\mathrm{d} x\\
       &\leq c_4\int_{B_k\cap \,\mathrm{supp }\nabla \varphi}G\Big(\frac{\max_{\partial \frac{2}{3}B_k}u_k}{R_k}\Big)\,\mathrm{d} x\\
        &\leq c_4\int_{ B_k\setminus (5/4 \overline{B}_{k+1})}G\Big(\frac{\max_{\partial \frac{2}{3}B_k}u_k}{R_k}\Big)\,\mathrm{d} x,
    \end{aligned}
\end{equation*}
where $c_3,\;c_4$ depending only $n,\,p,\,q,\,\alpha$ and $\beta$. Recall that $u_k$ is $\mathcal{A}$-harmonic in $ B_k\setminus (5/4 \overline{B}_{k+1})$. By \eqref{harnack inequality boundary 2},
\begin{equation*}
   \max_{\partial \frac{2}{3}B_k}u_k\leq c_5 \min_{\partial \frac{2}{3}B_k}u_k.
\end{equation*}
Combining this with \eqref{estimate47 orlicz}, we obtain
\begin{align}
     I_1&\leq c_6\int_{ B_k\setminus (5/4 \overline{B}_{k+1})}G\Big(\frac{ c_5 \min_{\partial \frac{2}{3}B_k}u_k}{R_k}\Big)\,\mathrm{d} x \nonumber\\
    &\leq c_6\int_{ B_k}G\Big(\frac{ c_5 \min_{\partial \frac{2}{3}B_k}u_k}{R_k}\Big)\,\mathrm{d} x \leq c_7\,R_k^n\,  G\Big(\frac{\min_{\partial \frac{2}{3}B_k}u_k}{R_k}\Big) \label{estimate48 orlicz}.
\end{align}
Next, using the Cauchy-Schwarz inequality, \eqref{condition A2 orlicz} and \eqref{estimate43 orlicz}, one has
\begin{align}
    I_2&\leq c_8\int_{B_k\cap\, \mathrm{supp }\nabla \varphi}g(|\nabla u_k|)\Big|\min\Big\{ u_k,\, \max_{\partial \frac{2}{3}B_k}u_k\Big\}\Big|\varphi^{q-1}\, |\nabla \varphi|\, \mathrm{d} x\nonumber\\
    &= c_8\int_{\mathrm{supp }\nabla \varphi}g\Big(\Big|\nabla \min\Big\{ u_k,\, \max_{\partial \frac{2}{3}B_k}u_k\Big\}\Big|\Big)\Big|\min\Big\{ u_k,\, \max_{\partial \frac{2}{3}B_k}u_k\Big\}\Big|\varphi^{q-1}\, |\nabla \varphi|\, \mathrm{d} x\nonumber\\
    & \leq c_8\int_{\mathrm{supp }\nabla \varphi}g\Big(\Big|\nabla\min\Big\{ u_k,\, \max_{\partial \frac{2}{3}B_k}u_k\Big\}\Big|\Big)\varphi^{q-1}\, \max_{\partial \frac{2}{3}B_k}u_k \, |\nabla \varphi|\, \mathrm{d} x \nonumber\\
    &\leq c_8\int_{\mathrm{supp }\nabla \varphi}g\Big(\Big|\nabla\Big(\max_{\partial \frac{2}{3}B_k}u_k- \min\Big\{ u_k,\, \max_{\partial \frac{2}{3}B_k}u_k\Big\}\Big)\Big|\Big)\varphi^{q-1}\, \max_{\partial \frac{2}{3}B_k}u_k \, |\nabla \varphi|\, \mathrm{d} x.\nonumber
\end{align}
Since $0\leq\varphi\leq 1$, $G^{\ast}(\varphi^{q-1}t)\leq c_8\, \varphi^q G^{\ast}(t)$ for all $t\geq 0$ by \eqref{estimate on G ast}. From this, with the aid of Caccioppoli estimate (Lemma~\ref{caccioppoli estimate orlicz}), a combination of Young's inequality \eqref{young inequality} and \eqref{relation G and tilde G} gives
\begin{align}
    I_2&\leq c_8 \int_{\mathrm{supp }\nabla \varphi} G\bigg(\max_{\partial \frac{2}{3}B_k}u_k \, |\nabla \varphi|\bigg)\, \mathrm{d} x \nonumber\\
    &\quad +  c_8 \int_{\mathrm{supp }\nabla \varphi}G^{\ast}\bigg(g\Big(\Big|\nabla\Big(\max_{\partial \frac{2}{3}B_k}u_k- \min\Big\{ u_k,\, \max_{\partial \frac{2}{3}B_k}u_k\Big\}\Big)\Big|\Big)\varphi^{q-1}\bigg)\,\mathrm{d} x \nonumber\\
    &\leq  c_8 \int_{ \mathrm{supp }\nabla \varphi} G\bigg(\max_{\partial \frac{2}{3}B_k}u_k \, |\nabla \varphi|\bigg)\, \mathrm{d} x\nonumber \\
    & \quad+c_{10}\int_{\mathrm{supp }\nabla \varphi}G\bigg(\Big|\nabla\Big(\max_{\partial \frac{2}{3}B_k}u_k- \min\Big\{ u_k,\, \max_{\partial \frac{2}{3}B_k}u_k\Big\}\Big)\Big|\bigg)\varphi^{q}\,\mathrm{d} x \nonumber\\
    & \leq c_8 \int_{ \mathrm{supp }\nabla \varphi} G\bigg(\max_{\partial \frac{2}{3}B_k}u_k \, |\nabla \varphi|\bigg)\, \mathrm{d} x \nonumber\\
    & \quad \quad  +  c_{11}\int_{\mathrm{supp }\nabla \varphi} G\bigg(\Big(\max_{\partial \frac{2}{3}B_k}u_k- \min\big\{ u_k,\, \max_{\partial \frac{2}{3}B_k}u_k \big\}\Big)\, |\nabla \varphi|\bigg)\,\mathrm{d} x\nonumber\\
    & \leq c_{12}\int_{ \mathrm{supp }\nabla \varphi} G\bigg(\max_{\partial \frac{2}{3}B_k}u_k \, |\nabla \varphi|\bigg)\, \mathrm{d} x \leq c_{13}\,R_k^n\,  G\Big(\frac{\min_{\partial \frac{2}{3}B_k}u_k}{R_k}\Big),  \label{estimate49 orlicz}
\end{align}
the last inequality follows the same method as in \eqref{estimate48 orlicz}. Here $c_i>0$, $i=5,\ldots,13$ are constants depending only on $n,\,p,\,q,\alpha$ and $\beta$. Applying \eqref{estimate48 orlicz} and \eqref{estimate49 orlicz} in \eqref{estimate46 orlicz},
\begin{equation*}
    \begin{aligned}
       \frac{\mu(B_{k+1})}{R_k^{n-1}}&\leq c_{14}\frac{R_k}{\min_{\partial \frac{2}{3}B_k}u_k}G\Big(\frac{\min_{\partial \frac{2}{3}B_k}u_k}{R_k}\Big)\\
        &\leq c_{15}\,g\Big(\frac{\min_{\partial \frac{2}{3}B_k}u_k}{R_k}\Big).
    \end{aligned}
\end{equation*}
Accordingly, for all $k\geq 0$ with $\min_{\partial B_{k+1}}u_k>0$, it holds 
\begin{equation*}
    R_kg^{-1}\Big(\frac{\mu(B_{k+1})}{R_k^{n-1}}\Big) \leq c_{16}\min_{\partial \frac{2}{3}B_k}u_k,
\end{equation*}
where $c_i>0$, $i=14,\, 15,\, 16$, depend only on $n,\,p,\,q,\,\alpha$ and $\beta$. Since $u$ is $\mathcal{A}$-superharmonic in $B_k$, we have
\begin{equation*}
    \min_{\partial \frac{2}{3}B_k}u=\min_{\frac{2}{3}B_k}u\leq \min_{B_{k+1}}u\leq \min_{\partial B_{k+1}}u.
\end{equation*}
Using this in \eqref{estimate41 orlicz}, we deduce that
\begin{equation*}
   \min_{\partial \frac{2}{3}B_k}u_k\leq \min_{\partial B_{k+1}}u - \min_{\partial B_{k}}u.
\end{equation*}
Hence
\begin{equation}\label{estimate50 orlicz}
    R_kg^{-1}\Big(\frac{\mu(B_{k+1})}{R_k^{n-1}}\Big) \leq c_{16}\big(\min_{\partial B_{k+1}}u - \min_{\partial B_{k}}u\big).
\end{equation}

Next, assume that $\min_{\partial B_{k_0}}u=0$ for some $k_0\geq 0$. The weak Harnack inequality \eqref{weak harnack inequality} implies that $u_{k_0}=0$ in $B_{k_0}$. From this, we infer that $u$ is $\mathcal{A}$-harmonic in $B_{k_0+1}$ since $\mu(B_{k_0+1})=\mu_{k_0}(B_{k_0+1})=0$, whence
\begin{equation}\label{estimate51 orlicz}
    \mu(B_j)=0 \quad \forall j\geq k_0+1.
\end{equation}
By Minimum Principle (Lemma~\ref{Minimum and Maximum Principles orlicz}), $\min_{\partial B_{k_0+1}}u=\min_{B_{k_0+1}}u\leq u(x_0)$. 

Thus, by summing up all cases, we concluded from \eqref{estimate50 orlicz} and \eqref{estimate51 orlicz} that
\begin{equation*}
\begin{aligned}
    u(x_0)&\geq  u(x_0) - \min_{\partial B_0}u \geq  \lim_{k\to\infty} \big(\min_{\partial B_{k+1}}u- \min_{\partial B_0}u\big)\\
    &= \sum_{k=0}^{\infty}\big(\min_{\partial B_{k+1}}u - \min_{\partial B_{k}}u\big)\geq \frac{1}{c_{16}}\sum_{k=0}^{\infty}R_kg^{-1}\Big(\frac{\mu(B_{k+1})}{R_k^{n-1}}\Big)
\end{aligned}
\end{equation*}
Reminding of $R_k=2^{1-k}R$ for $k\geq 0$, by \eqref{estimate on g-1}, we deduce that
\begin{equation*}
    \begin{aligned}
        u(x_0)& \geq  \frac{1}{c_{16}}\sum_{j=1}^{\infty}R_{j-1}g^{-1}\Big(\frac{\mu(B_{j})}{R_{j-1}^{n-1}}\Big)\\
        &= \frac{4}{c_{16}}\sum_{j=1}^{\infty}R_{j+1}g^{-1}\Big(\frac{\mu(B_{j})}{4^{n-1}R_{j+1}^{n-1}}\Big)\geq c_{17}\sum_{j=1}^{\infty}R_{j+1}g^{-1}\Big(\frac{\mu(B_{j})}{R_{j+1}^{n-1}}\Big)
    \end{aligned}
\end{equation*}
The estimate
\begin{equation*}
    \begin{aligned}
    \mathbf{W}_G^R\mu(x_0)&=\int_0^Rg^{-1}\Big(\frac{\mu(B(x_0,s)}{s^{n-1}}\Big)\,\mathrm{d} s \\
        & = \sum_{k=1}^{\infty}\int_{R_{k+1}}^{R_k}g^{-1}\Big(\frac{\mu(B(x_0,s)}{s^{n-1}}\Big)\,\mathrm{d} s \leq \sum_{k=1}^{\infty}\int_{R_{k+1}}^{R_k}g^{-1}\Big(\frac{\mu(B(x_0,R_k)}{R_{k+1}^{n-1}}\Big)\,\mathrm{d} s\\
        & = \sum_{k=1}^{\infty}(R_k-R_{k+1})g^{-1}\Big(\frac{\mu(B_k)}{R_{k+1}^{n-1}}\Big) = \sum_{k=1}^{\infty}R_{k+1}g^{-1}\Big(\frac{\mu(B_k)}{R_{k+1}^{n-1}}\Big)
    \end{aligned}
\end{equation*}
completes the proof of the lower bound in \eqref{potential estimate thm} by taking $C_1=c_{17}$, which depends only on $n,\,p,\,q,\,\alpha$ and $\beta$.
\end{proof}

\bigskip 

\section{Potential further developments}\label{Further works}
Our viewpoint sheds new light on the class of quasilinear problems with Orlicz growth and measure data. Here we indicate some questions related to this kind of problem.
\medskip
\begin{flushleft}
{$\mathbf{1}$.}  Observe that, by \eqref{condition p=q}, \eqref{cond sufficient to exist idea} coincides with \eqref{cond suf/nec Igor} if $p=q$. As it was mentioned in Remark~\ref{remark cond sufficient to exist idea}, would be interesting to show that condition \eqref{cond sufficient to exist idea} is sufficient to guarantee the existence of a nontrivial solution to \eqref{equation nonstandard} in $\mathcal{D}^{1,G}(\mathds{R}^n)\cap L^F(\mathds{R}^n,\mathrm{d}\sigma)$, whether $1< p<q<n$.  The motivation for this is that in \cite{MR3311903} was proved that condition \eqref{cond suf/nec Igor} is not only necessary but also sufficient to ensure the existence of a nontrivial solution to \eqref{equation standard 2} in $\mathcal{D}^{1,p}(\mathds{R}^n)\cap L^{1+r}(\mathds{R}^n,\mathrm{d} \sigma)$, when $p<n$. Certainly, an answer to this question relies on the refinement of Lemma~\ref{technical lemma orlicz}.
  \\
 {$\mathbf{2}$.} In \cite{MR3556326}, D. Cao and I. Verbitsky using a capacitary condition were able to show that there exists a nontrivial solution to \cref{equation standard 2} which satisfies pointwise the so-called estimates of the Brezis-Kamin type in terms of the usual Wolf potential $\mathbf{W}_p\sigma$. In view of the results obtained in this paper, it should be possible to construct similar pointwise estimates for solutions to ``sublinear'' problems like \cref{equation nonstandard} in terms of the generalized Wolf potential $\mathbf{W}_G\sigma$. The crucial key would be establishing a capacitary condition for the $G$-capacity, given in Definition~\ref{definition G capacity}. 
\end{flushleft}

\bigskip

\begin{flushleft}
 {\bf Funding:}  
 E. da Silva acknowledges partial support  from 
	CNPq through grants 140394/2019-2 and J. M. do \'O acknowledges partial support from CNPq through grants 312340/2021-4, 409764/2023-0, 443594/2023-6, CAPES MATH AMSUD grant 88887.878894/2023-00
and Para\'iba State Research Foundation (FAPESQ), grant no 3034/2021. \\
 {\bf Ethical Approval:}  Not applicable.\\
 {\bf Competing interests:}  Not applicable. \\
 {\bf Authors' contributions:}    All authors contributed to the study conception and design. All authors performed material preparation, data collection, and analysis. The authors read and approved the final manuscript.\\
{\bf Availability of data and material:}  Not applicable.\\
{\bf Ethical Approval:}  All data generated or analyzed during this study are included in this article.\\
{\bf Consent to participate:}  All authors consent to participate in this work.\\
{\bf Conflict of interest:} The authors declare no conflict of interest. \\
{\bf Consent for publication:}  All authors consent for publication. \\
\end{flushleft}

\bigskip





    \end{document}